\documentclass[a4paper, 10pt]{article}
\usepackage[a4paper, top=1.3in, bottom=1.30in, left=0.75in, right=0.75in]{geometry}
\usepackage{amssymb}
\usepackage{amsmath}
\usepackage{tgtermes}
\usepackage{setspace}
\usepackage{amsthm} 
\theoremstyle{plain}%
\newtheorem{theorem}{Theorem}

\newtheorem{proposition}{Proposition}%
\newtheorem{corollary}{Corollary}%
\newtheorem{lemma}{Lemma}%
\newtheorem{remark}{Remark}
\usepackage{algorithm}
\usepackage{algorithmicx}
\usepackage{algpseudocode}

\usepackage{cancel} 
\usepackage{graphicx}
\usepackage{listings}
\usepackage{mathrsfs}
\usepackage{titlesec}
\usepackage{authblk}
\usepackage{xcolor}
\usepackage{hyperref}
\usepackage{booktabs} 
\definecolor{bred}{rgb}{0.8,0,0}
\hypersetup{colorlinks=true,linkcolor={blue},citecolor={bred},urlcolor={blue}}

\newcommand{\citegroup}[1]{[{\let\cite\relax\cite{#1}}]}

\makeatletter
\renewcommand\subsubsection{\@startsection{subsubsection}{3}{\z@}%
  {-3.25ex\@plus -1ex \@minus -.2ex}%
  {1.5ex \@plus .2ex}%
  {\normalfont\small\itshape\mdseries}}
\makeatother
\titleformat{\section}[block]
  {\normalfont\normalsize\bfseries}{\thesection.\ \ }{0em}{}
\titleformat{\subsection}[block]
  {\normalfont\normalsize\itshape}{\thesubsection.\ \ }{0em}{}
\titleformat{\subsubsubsection}[block]
  {\normalfont\small\itshape\mdseries}{\thesubsubsubsection.}{0em}{}
\title{\Large Taming the Interacting Particle Langevin Algorithm: The Superlinear case}
\author[3]{Tim Johnston}
\author[1]{Nikolaos Makras}
\author[1,2,3,4]{Sotirios Sabanis}
\affil[1]{School of Mathematics, University of Edinburgh, UK}
\affil[2]{National Technical University of Athens, Greece}
\affil[3]{Université Paris Dauphine-PSL, France} 
\affil[4]{Alan Turing Institute, UK}
\affil[5]{Archimedes/Athena Research Centre, Greece}
\date{\small \today}
\begin{document}
\maketitle
\begin{abstract}
        \noindent Recent advances in stochastic optimization have yielded the interacting  particle Langevin algorithm (IPLA), which leverages the notion of interacting particle systems (IPS) to efficiently sample from approximate posterior densities. This becomes particularly crucial in relation to the framework of Expectation-Maximization (EM), where the E-step is computationally challenging or even intractable. Although prior research has focused on scenarios involving convex cases with gradients of log densities that grow at most linearly, our work extends this framework to include polynomial growth. Taming techniques are employed to produce an explicit discretization scheme that yields a new class of stable, under such non-linearities, algorithms which are called tamed interacting  particle Langevin algorithms (tIPLA). We obtain non-asymptotic convergence error estimates in Wasserstein-2 distance for the new class under the best known rate.
\end{abstract}
\section{Introduction}
The Expectation-Maximization (EM) algorithm is widely used for approximating maximisers of posterior distributions. Applications span, but are not confined to, hyperparameter estimation, mixture models, hidden variable models, missing or incomplete data and variational inference \cite{EM}. At its essence, each iteration of the EM algorithm consists of two fundamental steps: the Expectation (E) and the Maximization (M) step. The algorithm is defined by alternating iteratively between these two steps. Given a data specification $p_{\theta}(x,y)$ parameterized by $\theta$, where x represents the latent variable (interpreted often as incomplete data) and y the observed data, our aim is to find $\theta^{*}$ that maximizes the marginal likelihood 
$q_{\theta}(y)=\int_{\mathbb{R}^{d^x}}p_{\theta}(x,y)dx$.\newline
\indent When the integral in the E-step is computationally challenging or intractable, Markov Chain Monte Carlo (MCMC) methods are often employed, traditionally using Metropolis-Hastings-type algorithms \cite{MetroMCMC,Wei}. However, those MCEM (Monte Carlo Expectation-Maximization) methods have two considerable limitations. Firstly, the underlying EM algorithm is not guaranteed to converge towards a global maximum and is susceptible to local mode entrapment, as documented in \cite{EM}. Lastly, in the case of the utilization of MALA (Metropolis-Hastings Adjusted Langevin Algorithm), it is well known that scalability issues arise, particularly in high-dimensional settings.\\
\indent In this landscape, a new method was introduced by \cite{Soul}, such that samples of the latent space variable were generated via an Unadjusted Langevin Algorithm (ULA) chain. 
An alternative avenue was pioneered in \cite{Kuntz}, in which the study of the limiting behaviour of various gradient flows associated with appropriate free energy functionals led to an interacting particle system (IPS) that provides efficient estimates for maximum likelihood estimations. Further, \cite{Ipla} expanded on this framework by injecting noise into the dynamics of the parameter $\theta$ itself, thereby transitioning from deterministic to stochastic dynamics. This key modification provides a stochastic system with an invariant measure, allowing the establishment of non-asymptotic convergence estimates for the resulting algorithm that targets $\theta^{*}$, the maximiser of the marginal likelihood.\\
\indent In our work, while the convexity assumption is maintained, we address the challenge posed by superlinear growth exhibited by gradients of log densities, which makes other known algorithms, such as vanilla Langevin based algorithms, unstable. To counteract this, we implement taming techniques, initially researched for non-globally Lipschitz drifts for SDEs in \cite{Hutzenthaler2} and subsequently in \cite{sabanis_2013} and \cite{Sabanis_2016}. The latter approach has found applications in optimization and machine learning and led to the design of new MCMC algorithms as one typically deals with high nonlinear objective function, see e.g. \cite{TULA,lim2023nonasymptotic,lovas2023taming,sabanis2019higher}. The underlying principle of all these algorithms is the rescaling of ULA's drift coefficient in such a way that maintains stability without significantly increasing computational complexity as in the case of implicit schemes, or by introducing additional constrains via adaptive stepsizes.\\
\indent In this paper, we study two new algorithms from the tamed interacting particle Langevin algorithms (tIPLA) class, namely the coordinate wise version, known as tIPLAc, and the uniformly tamed version tIPLAu. Those algorithms are tamed versions of IPLA (developed in \cite{Ipla}) as explained in Subsection \ref{sec3}. The best known rate of convergence is recovered for both algorithms and our estimates are explicit regarding their dependencies on dimension and number of particles employed, which are denoted by $d$ and $N$ respectively.
\subsection{Notation}
We conclude this subsection with some basic notation. For $u,v\in\mathbb{R}^d$, define the scalar product $\langle u,v\rangle=\sum_{i=1}^d u_i v_i$ and the Euclidian norm $|u|={\langle u,u\rangle}^{1/2}$. For all continuously differentiable functions $f:\mathbb{R}^d\to\mathbb{R}$, $\nabla f$ denotes the gradient. The integer part of a real number x is denoted by $\lfloor x \rfloor$. We denote by $\mathcal{P}(\mathbb{R}^d)$ the set of probability measures on $\mathcal{B}(\mathbb{R}^d)$ and for any $p\in\mathbb{N}$, $\mathcal{P}_p(\mathbb{R}^d)=\{\pi\in\mathcal{P}:\int_{\mathbb{R}^d}|x|^p_pd\pi(x)<\infty\}$ denotes the set of all probability measures over $\mathcal{B}(\mathbb{R}^d)$ with finite $p$-th moment. For any two Borel probability measures $\mu$ and $\nu$, we define the Wasserstein distance of order $p\geq 1$ as
$$W_p(\mu,\nu)=\left(\inf_{\zeta\in\prod(\mu,\nu)}\int_{\mathbb{R}^d\times\mathbb{R}^d}|x-y|^p d\zeta(x,y)\right)^{1/p},$$
where $\prod(\mu,\nu)$ is the set of all couplings of $\mu$ and $\nu$. Moreover, for any $\mu,\ \nu\in\mathcal{P}_p(\mathbb{R}^d) $, there exists a coupling $\zeta^*\in\prod(\mu,\nu)$ such that for any coupling $(X,Y)$ distributed according to $\zeta^*$, $W_p(\mu,\nu)=\mathbb{E}^{1/p}\left[\left|X-Y\right|^p\right]$.
\section{Setting and Definitions}
\subsection{Initial setup}
Let $p_{\theta}(x,y)$ be the aforementioned joint probability density function of the latent variable $x$ and for fixed (observed) data $y$. The goal of maximum marginal likelihood estimation (MMLE) is to find the parameter $\theta^*$ that maximises the marginal likelihood, see \cite{EM}. To address the log-density formulation, we define the negative log-likelihood for fixed $y\in\mathbb{R}^{d^y}$ as $$U_{y}(\theta,x):=-\log p_{\theta}(x,y).$$ Thus we arrive at the following notation for the quantity we aim to maximize $$k_{y}(\theta)=q_{\theta}(y)=\int_{\mathbb{R}^{dx}} p_{\theta}(x,y)dx=\int_{\mathbb{R}^{dx}} \exp{(-U_{y}(\theta,x))}dx.$$ Henceforth, for notational brevity, we denote $h_{y}(v)=\nabla U_{y}(v)$ so that $h^{x}_{y}(v)=\nabla_x U_{y}(v)$ and $h^{\theta}_{y}(v)=\nabla_{\theta}U_{y}(v)$, with $v:=(\theta,x)$. Additionally, we drop the reference to the (fixed) data $y$. To summarize, the optimization problem can be described as $$\text{maximize }\mathbb{R}^{d^{\theta}}\ni \theta \rightarrow k(\theta)=\int_{\mathbb{R}^{dx}} \exp{(-U(\theta,x))}dx.$$ Following the convention of \cite{Kuntz}, $N$ particles $\mathcal{X}_t^{i,c,N}$ for $i\in\{1,\ldots,N\}$ are used to estimate the gradient of $k(\theta)$, which are governed by the continuous-time dynamics as given below
\begin{gather}
d\vartheta_t^{c,N}=-\dfrac{1}{N}\sum_{j=1}^N\nabla_{\theta}U(\vartheta_t^{c,N},\mathcal{X}^{j,c,N}_t)dt+\sqrt{\dfrac{2}{N}}dB_t^{0,N},\label{eq:00}\\d\mathcal{X}_t^{i,c,N}=-\nabla_{x}U(\vartheta_t^{c,N},\mathcal{X}^{i,c,N}_t)dt+\sqrt{2}dB_t^{i,N},\ \label{eq:01} \text{ for } i=1,\ldots,N,
\end{gather}
where $\{(B_t^{i,N})_{t\geq 0}\}_{0\leq i\leq N}$ is a family of independent Brownian motions. The discrete time Markov chain associated with the above IPS in equations \eqref{eq:00}-\eqref{eq:01} is obtained by the corresponding Euler-Maruyama discretization scheme of the Langevin SDEs \eqref{eq:00}-\eqref{eq:01}
\begin{align*}
\theta^{{\lambda}_{\text{\tiny{EM}}},c}_0=&\theta_0,\ \theta_{n+1}^{{\lambda}_{\text{\tiny{EM}}},c}=\theta_{n}^{\lambda_{\text{\tiny{EM}}},c}-\dfrac{\lambda}{N}\sum_{i=1}^N h^{\theta}(\theta_{n}^{\lambda_{\text{\tiny{EM}}},c},X_{n}^{i,\lambda_{\text{\tiny{EM}}},c})+\sqrt{\dfrac{2\lambda}{N}}\xi^{(0)}_{n+1},\\
X^{i,\lambda_{\text{\tiny{EM}}},c}_0=&x^i_0,\ X_{n+1}^{i,\lambda_{\text{\tiny{EM}}},c}=X_{n}^{i,\lambda_{\text{\tiny{EM}}},c}-\lambda h^{x}(\theta_{n}^{\lambda_{\text{\tiny{EM}}},c},X_{n}^{i,\lambda_{\text{\tiny{EM}}},c})+\sqrt{2\lambda}\xi^{(i)}_{n+1},\ 
\end{align*}
for $i=1,\ldots,N$, where ${\theta}_0\in\mathbb{R}^{d^{\theta}}$, $x^i_0\in\mathbb{R}^{d^x}$, $\lambda>0$ is the step-size parameter. Moreover, $(\xi^{(0)}_n)_{n\in\mathbb{N}}$ and $(\xi^{(i)}_n)_{n\in\mathbb{N}}$,\ $\forall i\in\{1,\ldots,N\}$, are sequences of i.i.d. standard $d^\theta-$ and $d^x-$dimensional Gaussian variables respectively. A time-scaled version of the original equations \eqref{eq:00}-\eqref{eq:01} is also introduced, playing a central role in the proof strategy as explained in Subsection \hyperref[summary]{3.5}.
\begin{align}
d\vartheta_t^{u,N}=&-\dfrac{1}{N^{p+1}}\sum_{j=1}^N\nabla_{\theta}U(\vartheta_t^{u,N},\mathcal{X}^{j,u,N}_t)dt+\sqrt{\dfrac{2}{N^{p+1}}}dB_t^{0,N},\label{eq:02}\\d\mathcal{X}_t^{i,u,N}=&-\dfrac{1}{N^p}\nabla_{x}U(\vartheta_t^{u,N},\mathcal{X}^{i,u,N}_t)dt+\sqrt{\dfrac{2}{N^p}}dB_t^{i,N},\ \label{eq:03}\text{ for } i=1,\ldots,N,
\end{align}
where $p=2\ell+1$ is controlled by the polynomial growth order $\ell$ in $\nabla U(\theta,x)$, see Remark \hyperlink{Remark1}{1}. Such an adjustment allows the generalization of the results in \cite{Ipla} to handle superlinear drift coefficients while maintaining the same underlying invariant measure $\pi_*^N$. One notes that $\pi_*^N$ is characterized by the density proportional to $\exp(-\sum_{i=1}^NU(\theta,x^i))$ for both systems, as shown in Proposition \hyperlink{Prop1}{1}.
\subsection{Taming approach}\label{sec3}
\subsubsection{A very brief introduction}
\noindent For a fixed $T>0$, consider the SDE given by
\begin{align}
    Y(0)=Y_0,\ dY(t)=b(t,Y(t))dt+\sigma(t,Y(t))dB_t,\ \forall t\in[0,T],\label{eq:04}
\end{align} where $b(t,y)$ and $\sigma(t,y)$ are assumed to be $\mathcal{B}(\mathbb{R}_+)\times \mathcal{B}(\mathbb{R}^d)-$measurable functions and $Y_0$ a random variable. The discrete time Markov chain associated with the ULA algorithm is obtained by the Euler-Maruyama discretization scheme of SDE \eqref{eq:04} and is defined for all $n\in\mathbb{N}$ by
$$Y(t_0)=Y_0,\ Y(t_{n+1})=Y(t_n)-\lambda b(t_n,Y(t_n))+\sqrt{\lambda}\sigma(t_n,Y(t_n))\xi_{n+1},\ n\geq 0,$$
where $\lambda>0$ is the stepsize and $\left(\xi_n\right)_{n\geq 1}$ is an i.i.d. sequence of standard Gaussian random variables. In the case where the drift coefficient $b$ is superlinear, it is shown in \cite{Hutzenthaler3} that ULA is unstable, in the sense that any $p-$absolute moment of the algorithm $(p\geq 1)$ diverges to infinity in finite time as the stepsize converge to zero. In the literature on SDE approximation, a new class of explicit numerical schemes has been introduced to study the case of non-globally Lipschitz conditions. These schemes modify both the drift and diffusion coefficients such that they grow at most linearly, for example see, \cite{kumar2016explicit,Sabanis_2016,Sabanis_2019}. The efficiency of such schemes and their respective properties of $\mathcal{L}^p$ convergence create a strong incentive to extend this technique to sampling and optimization. Typically tamed schemes are given by
$$Y(t_0)=Y_0,\ Y(t_{n+1})=Y(t_n)-\lambda b_{\lambda}(t_n,Y(t_n))+\sqrt{\lambda}\sigma_{\lambda}(t_n,Y(t_n))\xi_{n+1},\ n\geq 0,$$
for an appropriate choice of taming functions $b_{\lambda}:\mathbb{R}^d\mapsto \mathbb{R}^d$ and ${\sigma}_{\lambda}:\mathbb{R}^d\mapsto \mathbb{R}^{d\times d}$. 
\subsubsection{Application in the IPLA framework}\label{subsec2}
We extend the notion of taming to our setting, which involves superlinear drift coefficients, $\nabla U$ and constant diffusion coefficients. To achieve this,  we introduce a family of taming functions $(h_{\lambda})_{\lambda\geq 0}$ with $h_{\lambda}:\mathbb{R}^{d^\theta}\times\mathbb{R}^{d^x}\to \mathbb{R}^{d^\theta}\times\mathbb{R}^{d^x}$ which are close approximations of $\nabla U$ in a manner that is made precise below. We suggest two such taming functions $h_{\lambda}(v)$. One uses a uniform taming approach, while the other employs a coordinate-wise method:
\begin{align}
    h_{\lambda,u}(v)=\dfrac{h(v)-\mu v}{1+{\lambda}^{1/2}N^{-p/2}|h(v)-\mu v|}+\mu v,\label{eq:05} \end{align}\begin{align}
     h_{\lambda,c}(v)=\left(\dfrac{h^{(i)}(v)-\mu v^{(i)}}{1+{\lambda}^{1/2}|h^{(i)}(v)-\mu v^{(i)}|}+\mu v^{(i)}\right)_{i\in\{1,\ldots,d^{\theta}+d^x\}},\label{eq:06}
\end{align}
where $\mu$ is the strong convexity constant given in \hyperlink{Assum2}{A2}. In \cite{TULA}, it is experimentally established that the coordinate-wise version outperforms the uniform taming approach. This result is consistent with the observation that uniform taming does not distinguish between the different levels of contribution each coordinate offers to the gradient. However, there is a trade off for using the coordinate-wise approach, as in order to obtain an appropriate dissipativity condition for the tamed function $h_{\lambda}$, additional smoothness on the potential $U$ is required.\\
\indent We propose two different algorithms for the implementation of interacting particles methods within the tIPLA class, which are determined by the choice of the taming function and the stochastic dynamics. We find that when the uniform taming is used, then the time-scaled dynamics (\ref{eq:02})-(\ref{eq:03}) are more suitable as the new factor $1/N^p$ is essential to guarantee convergence of the algorithms iterates in Wasserstein-2 distance. Hence \eqref{eq:02}-\eqref{eq:03} paired with \eqref{eq:05} leads to the following scheme, for every $n\geq 0$,
\begin{gather}
\theta^{\lambda,u}_0=\theta_0,\ \theta_{n+1}^{\lambda,u}=\theta_{n}^{\lambda,u}-\dfrac{\lambda}{N^{p+1}}\sum_{i=1}^N h_{\lambda,u}^{\theta}(\theta_{n}^{\lambda,u},X_{n}^{i,\lambda,u})+\sqrt{\dfrac{2\lambda}{N^{p+1}}}\xi^{(0)}_{n+1},\label{eq:07}\\
X^{i,\lambda,u}_0=x^i_0,\ X_{n+1}^{i,\lambda,u}=X_{n}^{i,\lambda,u}-\dfrac{\lambda}{N^p} h_{\lambda,u}^{x}(\theta_{n}^{\lambda,u},X_{n}^{i,\lambda,u})+\sqrt{\dfrac{2\lambda}{N^p}}\xi^{(i)}_{n+1},\ \text{ for } i=1,\ldots,N\label{eq:08}.
\end{gather}
In the case of coordinate-wise taming, the extra smoothness required obviates any modification to the original dynamics. Therefore one uses (\ref{eq:00})-(\ref{eq:01}) with \eqref{eq:06} to obtain the second algorithm, for every $n\geq0,$
\begin{gather}
\theta^{\lambda,c}_0=\theta_0,\ \theta_{n+1}^{\lambda,c}=\theta_{n}^{\lambda,c}-\dfrac{\lambda}{N}\sum_{i=1}^N h_{\lambda,c}^{\theta}(\theta_{n}^{\lambda,c},X_{n}^{i,\lambda,c})+\sqrt{\dfrac{2\lambda}{N}}\xi^{(0)}_{n+1},\label{eq:09}\\
X^{i,\lambda,c}_0=x^i_0, X_{n+1}^{i,\lambda,c}=X_{n}^{i,\lambda,c}-\lambda h_{\lambda,c}^{x}(\theta_{n}^{\lambda,c},X_{n}^{i,\lambda,c})+\sqrt{2\lambda}\xi^{(i)}_{n+1},\ \text{ for } i=1,\ldots,N.\label{eq:000}
\end{gather}
It is important to note that, despite the differences in their setups, both algorithms exhibit the same non-asymptotic convergence behavior and dependence on the dimension. These schemes are summarised in the algorithms (\hyperref[algo1]{tIPLAu}) and (\hyperref[algo2]{tIPLAc}) presented below.
\begin{algorithm}[h]
    \caption{Tamed interacting  particle Langevin algorithm (tIPLAu)}\label{algo1}
    \begin{algorithmic}
        \Require N,$\lambda$,$\pi_{\text{init}}\in\mathcal{P}(\mathbb{R}^{d^{\theta}})\times\mathcal{P}((\mathbb{R}^{d^x})^N)$
        \State Draw $(\theta_0,\{X_0^{i,N}\}_{1\leq i\leq N})$ from $\pi_{\text{init}}$
        \For{n = $0$ to $n_T=\lfloor T/\lambda \rfloor$}
            \State $\theta_{n+1}^{\lambda,u}=\theta_{n}^{\lambda,u}-\dfrac{\lambda}{N^{p+1}}\sum_{i=1}^N h^{\theta}_{\lambda,u}(\theta_{n}^{\lambda,u},X_{n}^{i,\lambda,u})+\sqrt{\dfrac{2\lambda}{N^{p+1}}}\xi^{(0)}_{n+1}$\\
            \State $X_{n+1}^{i,\lambda,u}=X_{n}^{i,\lambda,u}-\dfrac{\lambda}{N^p}h^{x}_{\lambda,u}(\theta_{n}^{\lambda,u},X_{n}^{i,\lambda,u})+\sqrt{\dfrac{2\lambda}{N^p}}\xi^{(i)}_{n+1},\ \forall i\in{1,...,N}$
        \EndFor\\
        \Return $\theta_{n_T+1}$
    \end{algorithmic}
\end{algorithm}
It is apparent that the newly proposed algorithms, designated as \hyperref[algo1]{tIPLAu} and \hyperref[algo2]{tIPLAc}, incorporate foundational elements from \cite{Kuntz}, which employed deterministic dynamics in the $\theta$ component. Furthermore, their design benefits from the approach in \cite{Ipla} which introduced stochastic noise in the gradient flow of $\theta$, enabling explicit calculation of the invariant measure. The key addition of our approach lies in the application of a taming technique via the recalibration of the original drift terms, $h_x$ and $h_{\theta}$, into tamed counterparts, as described in \eqref{eq:05} and \eqref{eq:06}. This adjustment is crucial in stabilizing the algorithm and deriving non-asymptotic results.
\begin{algorithm}
    \caption{Tamed interacting  particle Langevin algorithm - coordinate-wise (tIPLAc)}\label{algo2}
    \begin{algorithmic}
        \Require N,$\lambda$,$\pi_{\text{init}}\in\mathcal{P}(\mathbb{R}^{d^{\theta}})\times\mathcal{P}((\mathbb{R}^{d^x})^N)$
        \State Draw $(\theta_0,\{X_0^{i,N}\}_{1\leq i\leq N})$ from $\pi_{\text{init}}$
        \For{n = $0$ to $n_T=\lfloor T/\lambda \rfloor$}
            \State $\theta_{n+1}^{\lambda,c}=\theta_{n}^{\lambda,c}-\dfrac{\lambda}{N}\sum_{i=1}^N h^{\theta}_{\lambda,c}(\theta_{n}^{\lambda,c},X_{n}^{i,\lambda,c})+\sqrt{\dfrac{2\lambda}{N}}\xi^{(0)}_{n+1}$\\
            \State $X_{n+1}^{i,\lambda,c}=X_{n}^{i,\lambda,c}-\lambda h^{x}_{\lambda,c}(\theta_{n}^{\lambda,c},X_{n}^{i,\lambda,c})+\sqrt{2\lambda}\xi^{(i)}_{n+1},\ \forall i\in{1,...,N}$
        \EndFor\\
        \Return $\theta_{n_T+1}$
    \end{algorithmic}
\end{algorithm}
\subsection{Essential quantities and proof strategy of main result}\label{essential}
In this subsection, we lay out the analytical framework and convergence properties of the sequences $(\theta_n^{\lambda,\cdot})_{n\geq 0}$ of either \hyperref[algo1]{tIPLAu} $(\theta_n^{\lambda,u})_{n\geq 0}$ or \hyperref[algo2]{tIPLAc} $(\theta_n^{\lambda,c})_{n\geq 0}$, resulting in their convergence to the maximiser $\theta^*$. A central concept of the technical analysis is the rescaling of the original dynamics of the following form \begin{align} \mathcal{Z}_t^{\cdot,N}=\left(\vartheta_t^{\cdot,N},N^{-1/2}\mathcal{X}_t^{1,\cdot,N},\ldots,N^{-1/2}\mathcal{X}_t^{N,\cdot,N}\right), \label{eq:001}\end{align} along with the corresponding  rescaling that is required for the algorithms and their continuous time interpolations, see more details in Appendix \hyperref[appendix1]{A}. This rescaled version of the particle system is essential in bridging the dynamics represented by either (\ref{eq:00})-(\ref{eq:01}) or (\ref{eq:02})-(\ref{eq:03}) with the pairs $\mathcal{V}_t^{i,\cdot,N}:=(\vartheta_t^{\cdot,N},\mathcal{X}_t^{i,\cdot,N})$ generated by the algorithms. The essence of this rescaling lies in its ability to equate the moments of $\mathcal{Z}_t^{\cdot,N}$ with the averaged moments of the pairs $\mathcal{V}_t^{i,\cdot,N}$ across the $\mathcal{L}^2$ norm. This rescaling choice yields the critical property $$|\mathcal{Z}_t^{\cdot,N}|^2=\dfrac{1}{N}\sum_{i=1}^N|\mathcal{V}_t^{i,\cdot,N}|^2.$$It is imperative to note that this convenient rescaling does not detract from the analysis objectives, as our primary interest is in the convergence of the $\theta-$component. The introduction of particles serves primarily to facilitate sampling, hence their specific scaling does not impact the core convergence analysis.\\
\indent In the present analysis, the interacting  particle systems defined in equations (\ref{eq:00})-(\ref{eq:01}) and (\ref{eq:02})-(\ref{eq:03}) are deliberately constructed to target the same invariant measure $\pi_*^N$ characterized by the density $Z^{-1}\exp(-\sum_{i=1}^N U(\theta,x^i))$. This aspect is critical as it assigns the number of particles $N$ a role akin to the inverse temperature parameter encountered in simulated annealing algorithms, thereby regulating the concentration of the $\theta-$marginal of the invariant measure  $\pi_{\Theta}^N$, towards the maximiser $\theta^*$. This behavior, wherein $\pi_{\Theta}^N\to\delta_{\theta^*}$ as $N\to\infty$, is established in Proposition \hyperlink{Prop2}{2}.\\
\indent Moreover, the convergence rate of the Langevin diffusion to its invariant measure is a standard result, one  could consult \cite{Durmus2} and the references within. This is further generalized to particle systems in the research conducted by \cite{Ipla}. By integrating their findings into our framework, we arrive at Proposition \hyperlink{Prop3}{3} and \hyperlink{Prop7}{4}. Equipped with these guarantees, we obtain the convergence of the underlying dynamics to their invariant measure, and subsequently our analysis focuses to the discretization errors inherent in the proposed schemes.\\
\indent One expects a dimensional scaling of at least $\mathcal{O}(d^{1/2})$ for the Wasserstein-2 numerical error of Langevin based algorithms, for example \cite{Durmus1}. In the context of the E optimization where $N$ particles are generated, the true dimensionality of the problem is $d^{\theta}+N\cdot d^{x}$, which can be naively interpreted as the need to increase the number of iterations to achieve convergence as $N$ increases. However, this dependency on $N$ is effectively mitigated by making use of the implied symmetry within the dynamics of the interacting  particle system by considering only the $\theta$-marginal in the analysis. Such an approach enables us to attain the classical Euler-Maruyama convergence rate for the numerical solutions of stochastic differential equations, as demonstrated in Proposition \hyperlink{Prop5}{5} and \hyperlink{Prop6}{6}.\\
\indent In closing, one decomposes the global $\mathcal{L}^2$-error, denoted as $\left(\mathbb{E}\left[|{\theta}^{*}-{\theta}^{\lambda,\cdot}_n|^2\right]\right)^{1/2}$, into three distinct components $$\mathbb{E}^{1/2}\left[|{\theta}^{*}-{\theta}^{\lambda,\cdot}_n|^2\right]\leq W_2(\delta_{\theta^*},\pi_{\Theta}^N)+W_2(\pi_{\Theta}^N,\mathcal{L}(\vartheta^{\cdot,N}_{n\lambda}))+W_2(\mathcal{L}(\vartheta^{\cdot,N}_{n\lambda}),\mathcal{L}(\theta_n^{\lambda,\cdot})),$$ where ${\theta}^{\lambda,\cdot}_n$ stands for the iterates of either of the algorithms, \hyperref[algo1]{tIPLAu} and \hyperref[algo2]{tIPLAc} alike. Moreover, $W_2(\delta_{\theta^*},\pi_{\Theta}^N)$ quantifies the deviation of the invariant measure from the maximiser, $W_2(\pi_{\Theta}^N,\mathcal{L}(\vartheta^{\cdot,N}_{n\lambda}))$ captures the discrepancy between the law of the dynamics and their invariant measure, and $W_2(\mathcal{L}(\vartheta^{\cdot,N}_{n\lambda}),\mathcal{L}(\theta_n^{\lambda,\cdot}))$ encompasses the error induced by the discretization scheme, namely, the divergence between the law of the iterations of our algorithms and their continuous counterparts.
\section{Main Assumptions and Results}
 In this section, we provide the assumptions regarding the potential $U(u)$ and its gradient  $h(u)=\nabla U(u)$ that define the framework, under which the main results for the non-asymptotic behavior of the newly proposed algorithms, \hyperref[algo1]{tIPLAu} and \hyperref[algo2]{tIPLAc}, are derived.
\par \noindent Let $d:=d^{\theta}+d^x$, and $U$ be a $C^1\left(\mathbb{R}^d\right)$ function. We require that $h(v):=h(\theta,x)$ is a locally Lipschitz function in $\theta$ and $x$.\\\hypertarget{Assum1}{} \\
\noindent \textbf{A1}. There exist $L>0$ and $\ell>0$ such that
$$|h(v)-h(v')|\leq L\left(1+|v|^{\ell}+|v'|^{\ell}\right)|v-v'|,\ \forall v,v'\in\mathbb{R}^d.$$
\noindent Additionally we require $U$ to be $\mu-$strongly convex.\\ \hypertarget{Assum2}{} \\
\noindent \textbf{A2}. There exists $\mu> 0$ such that
$$\left\langle v-v',h(v)-h(v')\right\rangle\geq \mu |v-v'|^2,\ \forall v,v'\in\mathbb{R}^d.$$
Assumption \hyperlink{Assum1}{A1} is a significant relaxation of the global Lipschitz condition which is widely used in the literature. It allows the gradient $h$ to grow polynomially fast at infinity. Assumption \hyperlink{Assum2}{A2} guarantees that $U$ has a unique minimiser and also can be seen as a monotonicity type condition, which is satisfied by the drift coefficients of the corresponding SDEs.\\
\indent We also present a growth condition for the gradient $\nabla U$, which is essential in the case where the coordinate-wise taming function \eqref{eq:06} is used.\\ \hypertarget{Assum3a}{}\\
\noindent \textbf{A3-i}. For each $i$ in $\{1,\ldots,d\}$ there exists $\mu>0$ such that $$h^{(i)}(v)v^{(i)}\geq \dfrac{\mu}{2} |v^{(i)}|^2-\dfrac{1}{2\mu}|h^{(i)}_0|^2.$$ In the special case where $d^x=d^\theta$ \hyperlink{Assum3a}{A3-i} can be slightly relaxed to the following form.\\\hypertarget{Assum3b}{}\\
\noindent \textbf{A3-ii}.
\noindent Let $d^x=d^{\theta}=d$, then for each $i$ in $\{1,\ldots,d\}$ there exist $\mu>0$, $\rho\geq 0$ satisfying the constrain $4\rho<\mu$, such that $$h^{{\theta}_i}(v){\theta}_i\geq \dfrac{\mu}{2} |{\theta}_i|^2-\rho|{x}_i|^2-\dfrac{1}{2\mu}|h^{\theta_i}_0|^2 \text{ and } h^{x_i}(v)x_i\geq \dfrac{\mu}{2} |x_i|^2-\rho|{\theta}_i|^2-\dfrac{1}{2\mu}|h^{x_i}_0|^2.$$
\noindent Assumption \hyperlink{Assum3a}{A3} introduces a coordinate-wise dissipativity type condition, which plays a critical role in establishing moment bounds for \hyperref[algo2]{tIPLAc}. Its necessity stems from the fact that Remark \hyperlink{Remark2}{2} does not guarantee that the taming function \eqref{eq:06} preserves the dissipativity of the drift $h(v)=\nabla U(v)$.\\ 
\indent We emphasize that Assumption \hyperlink{Assum3a}{A3} is primarily influenced by the structural properties of the taming function. This assumption could potentially be relaxed to require only argument-wise dissipativity, as described by the following mathematical formulation: $$h^{\theta}(v)\theta\geq \dfrac{\mu}{2} |\theta|^2-\rho|x|^2-\dfrac{1}{2\mu}|h^{\theta}_0|^2 \text{ and } h^{x}(v)x\geq \dfrac{\mu}{2} |x|^2-\rho|\theta|^2-\dfrac{1}{2\mu}|h^{x}_0|^2.$$ This relaxation is applicable when considering the taming function $$h_{\lambda,a}^w(v)=\dfrac{h^{w}(v)-\mu w}{1+{\lambda}^{1/2}|h^{w}(v)-\mu w|}+\mu w\ \text{for } w=\theta,\ x,$$ which represents an intermediate scenario bridging the uniform and coordinate-wise approaches. Such consideration leads to an alternative formulation of the tIPLA algorithm. However, this variant of the tIPLA algorithm is not of primary interest due to its inherent limitations. In particular, it fails to adequately discriminate between the components of the gradient that exhibit explosive behavior and those that do not. While simultaneously, this alternative formulation imposes more stringent requirements than those necessitated by the uniform case. Lastly, we proceed with an assumption on starting values for the tIPLA algorithms.\\ \hypertarget{Assum4} \\ \noindent \textbf{A4}. The initial condition $z_0^N=\left(\theta_0,N^{-1/2}x_0^1,\ldots,N^{-1/2}x_0^N\right)$ is such that $$\mathbb{E}\left[|z_0^N|^{2p_0}\right]< \infty,$$ where $p_0=2(\ell+1)$ and $\ell$ is the order of polynomial growth of $\nabla U(\theta,x)$ as given in Assumption \hyperlink{Assum1}{A1}.\hypertarget{Remark1}{}
\begin{remark} Assumption \hyperlink{Assum1}{A1} imposes a restriction on the growth of $h(v)=\nabla U(v)$, that is \begin{align}h(v)\leq K(1+|v|^{\ell+1})\label{eqR1},\ \forall v\in \mathbb{R}^{d},\end{align} where $K=2L+|h_0|$.\end{remark} \hypertarget{Remark2}{}
\begin{remark} From Assumption \hyperlink{Assum2}{A2} (strong convexity) one deduces that it implies dissipativity \begin{align}\langle v,h(v)\rangle\geq \dfrac{\mu}{2}|v|^2-b\label{eqR2},\ \forall v\in \mathbb{R}^{d},\end{align} where $b=\dfrac{1}{2\mu}|h_0|^2$.\end{remark}
\subsection{Taming functions and inherited properties}
In this subsection, we establish important properties of the proposed taming functions.\\ \hypertarget{Property1}{}
\par \noindent \textbf{Property 1}. For all $\lambda>0$ and  $v\in\mathbb{R}^{d}$, it holds that 
\begin{align*}
|h_{\lambda,u}(v)|&\leq \mu|v|+{\lambda}^{-1/2}N^{p/2},\\
|h_{\lambda,c}(v)|&\leq \mu|v|+(d^{\theta}+d^x){\lambda}^{-1/2}.
\end{align*}
\noindent The gradient $h$ is sufficiently close to its tamed counterparts $h_{\lambda,u}, h_{\lambda,c}$ for $\lambda>0$ small enough.\\ 
\hypertarget{Property2}{} 
\par \noindent \textbf{Property 2}. For all $\lambda>0$ and  $v\in\mathbb{R}^{d}$, it holds that
\begin{align*}
|h_{\lambda,u}(v)-h(v)|&\leq C_1{\lambda}^{1/2}N^{-p/2}(1+|v|^{2(\ell+1)}),\\
|h_{\lambda,c}(v)-h(v)|&\leq C_1{\lambda}^{1/2}(1+|v|^{2(\ell+1)}).    
\end{align*}
\noindent where $C_1=2^{2(\ell+3/2)}\max\{K^2,\mu^2\}$.\\
\par \noindent The taming functions $h_{\lambda,u}, h_{\lambda,c}$ inherit the dissipativity condition established in Remark \hyperlink{Remark2}{2}.\\ 
\hypertarget{Property3}{}
\par \noindent \textbf{Property 3}. For all $\lambda>0$ and  $v\in\mathbb{R}^{d}$, it holds that $$\langle v,h_{\lambda,u}(v)\rangle\geq \dfrac{\mu}{2}|v|^2-b,\ \text{and } \langle v,h_{\lambda,c}(v)\rangle\geq \dfrac{\mu}{2}|v|^2-b,$$
\noindent where $b$ is given in Remark \hyperlink{Remark2}{2}. We note that in the case of the coordinate-wise taming function to be dissipative, we require the additional smoothness provided by Assumption \hyperlink{Assum3a}{A3}.\\\newline In the context of a unified framework for Langevin based algorithm under superlinear gradients, as established in \cite{lim2021polygonal}, the properties 1-3 therein are recovered for $\delta=2$ and $\gamma=1/2$ which aligns with TH$\epsilon$O POULA \cite{lim2023langevin} and TUSLA \cite{lovas2023taming} algorithms.
\subsection{Preliminary Theoretical Statements}

 Both interacting SDEs defined by \eqref{eq:00}-\eqref{eq:01} and \eqref{eq:02}-\eqref{eq:03} respectively, admit a strong solution under \hyperlink{Assum1}{A1}-\hyperlink{Assum2}{A2}. This follows since each of the drift coefficients of the SDEs are locally Lipschitz functions and dissipative, in view of Remark  \hyperlink{Remark2}{2}, whereas the diffusion coefficients are constant. One could consult \cite{Mao} (Theorem 2.3.5) for more details.\label{statements} Moreover, both systems exhibit the same invariant measure, as one would expect from a standard Langevin diffusion.\\\raggedright
 
\begin{proposition}\hypertarget{Prop1}{} Let \hyperlink{Assum1}{A1} and \hyperlink{Assum2}{A2} hold. Then, the measure $\pi_{*}^N$ characterized by the density $Z^{-1}\exp{\left(-\sum_{i=1}^N U(\theta,x^i)\right)}$, with $Z$ being the normalizing constant, is the invariant measure for both interacting particles systems \eqref{eq:00}-\eqref{eq:01} and \eqref{eq:02}-\eqref{eq:03}.\end{proposition}
\begin{proof}
    The existence of an invariant measure is established by invoking the Krylov-Bogoliubov Theorem, as presented in Theorem 7.1 \cite{Prato}. This theorem guarantees the existence of an invariant measure for a Markov process generated by its corresponding semigroup, given a tightness condition on the sequence of probability measured associated with that process. According to Da Prato's Proposition 7.10 \cite{Prato}, the establishment of uniform moment bounds, provided in Lemma \hyperref[lemma0]{1}, serves as a sufficient condition for tightness. The uniqueness of the invariant measure follows from Theorem 7.16 (ii) \cite{Prato}, as Hypothesis 7.13 is satisfied due to \hyperlink{Assum1}{A1} and \hyperlink{Assum2}{A2}. Furthermore, one can verify that $\pi_{*}^N$ is indeed the invariant measure in discussion by checking that for any $\phi\in C^{\infty}_c(\mathbb{R}^d)$ one has $\int_{\mathbb{R}^d}\mathcal{L}\phi \ d\pi_{*}^N=0$, where $\mathcal{L}$ is the infinitesimal generator of the process. See the proof of Proposition 2 in \cite{Ipla} for more details.
\end{proof}
\noindent Henceforth we consider exclusively the $\theta-$marginal of the invariant measure. We show that indeed the number of particles $N$, plays the role of the inverse temperature parameter in temperature annealing algorithms in the sense that it provides control over the marginal's concentration around the maximiser.\\ 
\par \noindent \begin{proposition}\hypertarget{Prop2}{} Let $\pi_{\Theta}^N$ denote the $\theta-$marginal of the invariant measure and $\theta^*$ be the maximiser of $k(\theta)$. Then, under \hyperlink{Assum2}{A2}, for any $N\in\mathbb{N}$, the following bound holds
$$W_2(\pi_{\Theta}^N,\delta_{\theta^*})\leq \sqrt{\dfrac{d^{\theta}}{\mu N}}.$$\end{proposition}
\begin{proof}
    Follows from Proposition 3 in \cite{Ipla}, with the  observation that strong log-concavity is preserved under marginalization, see Theorem 3.8 in \cite{convex}.
\end{proof}
\noindent Next, we need the following ergodicity result to obtain the convergence of the law of the IPS to the invariant measure.\\ 
\par \noindent \begin{proposition}\hypertarget{Prop3}{} Let \hyperlink{Assum1}{A1}, \hyperlink{Assum2}{A2} and \hyperlink{Assum4}{A4} hold and consider the first component, namely $\vartheta^{u,N}_t$, of the continuous dynamics \eqref{eq:02}-\eqref{eq:03}. Then, for any $N\in\mathbb{N}$,
$$W_2\left(\mathcal{L}(\vartheta^{u,N}_t),\pi_{\Theta}^N\right)\leq \exp{(-\mu t/N^{p})}\left(\left(\mathbb{E}\left[|z_0^N-z^*|^2\right]\right)^{1/2}+\left(\dfrac{d^x N+d^{\theta}}{\mu N^{p+1}}\right)^{1/2}\right).$$\end{proposition}
\begin{proof}
    The proof is postponed to Appendix \hyperref[Proof3]{A}. 
\end{proof}
\par \noindent \begin{proposition}\hypertarget{Prop7}{} Let \hyperlink{Assum1}{A1}-\hyperlink{Assum4}{A4} hold and consider the first component, namely $\vartheta^{c,N}_t$, of the continuous dynamics \eqref{eq:00}-\eqref{eq:01}. Then, for any $N\in\mathbb{N}$,
$$W_2\left(\mathcal{L}(\vartheta^{c,N}_t),\pi_{\Theta}^N\right)\leq \exp{(-\mu t)}\left(\left(\mathbb{E}\left[|z_0^N-z^*|^2\right]\right)^{1/2}+\left(\dfrac{d^x N+d^{\theta}}{\mu N}\right)^{1/2}\right).$$\end{proposition}
\begin{proof}
    Follows closely from the proof of Proposition 4 in \cite{Ipla}. One can easily verify that although the corresponding A1 in \cite{Ipla} invokes global Lipschitz continuity rather than local, its need there is solely to ensure the existence of a unique global solution, a fact already established at the start of Subsection \hyperref[statements] {3.2}.
\end{proof}
\subsection{Discretisation Error Estimates}
\noindent We consider the continuous interpolation of the \hyperref[algo1]{tIPLAu} algorithm, which approximates the time-scaled version of the continuous dynamics \eqref{eq:02}-\eqref{eq:03}, denoted by $\overline{{Z}}_t^{\lambda,u}$. Notably, the law of the discretized process and its interpolation coincide at grid-points, i.e. $\mathcal{L}({Z}_n^{\lambda,u})=\mathcal{L}(\overline{{Z}}_n^{\lambda,u})$. The continuous-time interpolation of \hyperref[algo1]{tIPLAu} is defined as follows
\begin{align}
\overline{\theta}_0^{\lambda,u}=&\ \theta_0,\ d\overline{\theta}_t^{\lambda,u}=-\dfrac{\lambda}{N^{p+1}}\sum_{i=1}^N h^{\theta}_{\lambda,u}(\theta_{\lfloor{t}\rfloor}^{\lambda,u},X^{i,\lambda,u}_{\lfloor{t}\rfloor})dt+\sqrt{\dfrac{{2\lambda}}{N^{p+1}}}dB_t^{0,\lambda},\label{eq:002}\\
\overline{X}_0^{i,\lambda,u}=&\ x^i_0,\ d\overline{X}_t^{i,\lambda,u}=-\dfrac{\lambda}{N^p}h^{x}_{\lambda,u}(\theta_{\lfloor{t}\rfloor}^{\lambda,u},X^{i,\lambda,u}_{\lfloor{t}\rfloor})dt+\sqrt{\dfrac{{2\lambda }}{N^p}}dB_t^{i,\lambda},\ \text{ for } i=1,..,N,\label{eq:003}
\end{align}
whereas the time-changed SDEs \eqref{eq:02}-\eqref{eq:03} are given by
\begin{gather}
d{\vartheta}_{\lambda t}^{u,N}=-\dfrac{\lambda}{N^{p+1}}\sum_{i=1}^N h^{\theta}({\vartheta}^{u,N}_{\lambda t},\mathcal{X}_{\lambda t}^{i,u,N})dt+\sqrt{\dfrac{2\lambda }{N^{p+1}}}dB_t^{0,\lambda},\label{eq:004}\\
d\mathcal{X}_{\lambda t}^{i,u,N}=-\dfrac{\lambda}{N^p} h^{x}({\vartheta}^{u,N}_{\lambda t},\mathcal{X}_{\lambda t}^{i,u,N})dt+\sqrt{\dfrac{2\lambda}{N^p} }dB_t^{i,\lambda},\ \text{ for } i=1,..,N,\label{eq:005}
\end{gather}
where $B_t^{j,\lambda}:=B_{\lambda t}^{j}/\sqrt{\lambda},\ t\geq 0,\ 0\leq j\leq N,$ represents a Brownian motion under its completed natural filtration $\mathcal{F}^{j,\lambda}_t:=\mathcal{F}^j_{\lambda t}.$\\ 
\par \noindent \begin{proposition}\hypertarget{Prop5}{}\label{Proposition5} Let \hyperlink{Assum1}{A1}, \hyperlink{Assum2}{A2} and \hyperlink{Assum4}{A4} hold and consider the process ${\vartheta}_{\lambda t}^{u,N}$ as given by the dynamics \eqref{eq:004}-\eqref{eq:005}. Then, for every $\lambda_0< N^{2\ell+1}/(4\mu)$, there exists a constant $C>0$ independent of $N,n,\lambda, d^x \text{and } d^{\theta}$ such that for any $\lambda\in(0,\lambda_0)$,
$$\left(\mathbb{E}\left[|\overline{\theta}^{\lambda,u}_n-\vartheta^{u,N}_{n\lambda}|^2\right]\right)^{1/2}\leq \lambda^{1/2} C_{|z_0|,\mu,b,\ell,L}{(1+d^{\theta}/N+d^x)}^{\ell+1},$$
for all $n\in\mathbb{N}$.\end{proposition}
\begin{proof}
    The proof is postponed to Appendix \hyperref[Proof5]{A}.
\end{proof}
\noindent Similarly, we define the corresponding auxiliary processes for the iterates of \hyperref[algo2]{tIPLAc}. The continuous-time interpolation is given by
\begin{align}
\overline{\theta}_0^{\lambda,c}&=\theta_0,\ d\overline{\theta}_t^{\lambda,c}=-\dfrac{\lambda}{N}\sum_{i=1}^N h^{\theta}_{\lambda,c}(\theta_{\lfloor{t}\rfloor}^{\lambda,c},X^{i,\lambda,c}_{\lfloor{t}\rfloor})dt+\sqrt{\dfrac{{2\lambda}}{N}}dB_t^{0,\lambda},\label{eq:006}\\
\overline{X}_0^{i,\lambda,c}&=x^i_0,\ d\overline{X}_t^{i,\lambda,c}=-\lambda h^{x}_{\lambda,c}(\theta_{\lfloor{t}\rfloor}^{\lambda,c},X^{i,\lambda,c}_{\lfloor{t}\rfloor})dt+\sqrt{2\lambda }dB_t^{i,\lambda},\ \text{ for } i=1,\ldots,N,\label{eq:007}
\end{align}
and the time changed SDEs of \eqref{eq:00}-\eqref{eq:01} are
\begin{gather}
d{\vartheta}_{\lambda t}^{c,N}=-\dfrac{\lambda}{N}\sum_{i=1}^N h^{\theta}({\vartheta}^{c,N}_{\lambda t},\mathcal{X}_{\lambda t}^{i,c,N})dt+\sqrt{\dfrac{2\lambda }{N}}dB_t^{0,\lambda},\label{eq:008}\\
d\mathcal{X}_{\lambda t}^{i,c,N}=-{\lambda} h^{x}({\vartheta}^{c,N}_{\lambda t},\mathcal{X}_{\lambda t}^{i,c,N})dt+\sqrt{2\lambda}dB_t^{i,\lambda},\ \text{ for } i=1,\ldots,N,\label{eq:009}
\end{gather}where the Brownian motions are defined as in \eqref{eq:004}-\eqref{eq:005}.\\ 
\par \noindent \begin{proposition}\hypertarget{Prop6}{}\label{Proposition6} Let \hyperlink{Assum1}{A1}, \hyperlink{Assum2}{A2}, \hyperlink{Assum3a}{A3-i} and \hyperlink{Assum4}{A4} hold and consider the process ${\vartheta}_{\lambda t}^{c,N}$ as given by the dynamics \eqref{eq:008}-\eqref{eq:009}. Then, for every $\lambda_0<1/(4\mu)$, there exists a constant $C>0$ independent of $N,n,\lambda, d^x \text{and } d^{\theta}$ such that for any $\lambda\in(0,\lambda_0)$,
$$\left(\mathbb{E}\left[|\overline{\theta}^{\lambda,c}_n-\vartheta^{c,N}_{n\lambda}|^2\right]\right)^{1/2}\leq \lambda^{1/2} C_{|z_0|,\mu,b,\ell,L}{(1+d^{\theta}+d^x)}^{\ell+1},$$
for all $n\in\mathbb{N}$. If assumption \hyperlink{Assum3b}{A3-ii} is used in place of \hyperlink{Assum3a}{A3-i}, the proposition holds under the updated timestep constraint, $\lambda_0<1/(8\mu)$.\end{proposition}
\begin{proof}
    The proof is postponed to Appendix \hyperref[Proof6]{A}.
\end{proof}
\subsection{Main results and global error}
\begin{theorem} \label{theorem:1}Consider the iterates $({\theta}^{\lambda,u}_{n})_{n\geq 0}$ as given in \hyperref[algo1]{tIPLAu}, and let \hyperlink{Assum1}{A1}, \hyperlink{Assum2}{A2} and \hyperlink{Assum4}{A4} hold. Then, for every $\lambda_0<N^{2\ell+1}/(4\mu)$, there exists a constant $C>0$ independent of $N,n,\lambda,d^x \text{and } d^{\theta}$ such that, for any $\lambda\in(0,\lambda_0)$,
\begin{align*}
\left(\mathbb{E}\left[|{\theta}^{*}-{\theta}^{\lambda,u}_{n}|^2\right]\right)^{1/2}&\leq \sqrt{\dfrac{d^{\theta}}{\mu N}}+\lambda^{1/2} C_{|z_0|,\mu,b,\ell,L}(1+d^{\theta}/N+d^x)^{\ell+1}\\&+\exp{(-\mu n\lambda/N^{2\ell+1})}\left(\left(\mathbb{E}\left[|z_0^N-z^*|^2\right]\right)^{1/2}+\left(\dfrac{d^x N+d^{\theta}}{\mu N^{2(\ell+1)}}\right)^{1/2}\right),
\end{align*}
for all $n\in\mathbb{N}$.\end{theorem}
\begin{proof}
    Combining the results from Propositions \hyperlink{Prop2}{2}, \hyperlink{Prop3}{3} and \hyperlink{Prop5}{5}, we decompose the expectation into three terms: one describing the concentration of $\pi_{\Theta}^N$ around $\theta^*$, one describing the convergence of the IPS to its invariant measure, and lastly one describing the error induced by the time discretisation
\begin{align}
\left(\mathbb{E}\left[|{\theta}^{*}-{\theta}^{\lambda,u}_{n}|^2\right]\right)^{1/2}&=W_2(\delta_{\theta^*},\mathcal{L}(\theta^{\lambda,u}_{n}))\nonumber\\
&\leq W_2(\delta_{\theta^*},\pi_{\Theta}^N)+W_2(\pi_{\Theta}^N,\mathcal{L}(\vartheta^{u,N}_{n\lambda}))+W_2(\mathcal{L}(\vartheta^{u,N}_{n\lambda}),\mathcal{L}(\theta_{n}^{\lambda,u})).\label{eq:31}
\end{align}
Substituting the bounds from the aforementioned Propositions yields the final inequality.
\end{proof}
\noindent \begin{theorem}\label{theorem:2} Consider the iterates $({\theta}^{\lambda,c}_{n})_{n\geq 0}$ as given in \hyperref[algo2]{tIPLAc}, and let \hyperlink{Assum1}{A1}, \hyperlink{Assum2}{A2}, \hyperlink{Assum3a}{A3-i} and \hyperlink{Assum4}{A4} hold. Then, for every $\lambda_0<1/(4\mu)$, there exists a constant $C>0$ independent of $N,n,\lambda,d^x \text{and } d^{\theta}$ such that, for any $\lambda\in(0,\lambda_0)$,
\begin{align*}
\left(\mathbb{E}\left[|{\theta}^{*}-{\theta}^{\lambda,c}_{n}|^2\right]\right)^{1/2}&\leq \sqrt{\dfrac{d^{\theta}}{\mu N}}+\exp{(-\mu n\lambda)}\left(\left(\mathbb{E}\left[|z_0^N-z^*|^2\right]\right)^{1/2}+\left(\dfrac{d^x N+d^{\theta}}{\mu N}\right)^{1/2}\right)\\
&+\lambda^{1/2} C_{|z_0|,\mu,b,\ell,L}(1+d^{\theta}+d^x)^{\ell+1},
\end{align*}
for all $n\in\mathbb{N}$. If assumption \hyperlink{Assum3b}{A3-ii} is used in place of \hyperlink{Assum3a}{A3-i}, the theorem holds under the updated timestep constraint, $\lambda_0<1/(8\mu)$.\end{theorem}
\begin{proof}
To prove the above inequality, we replace the bound for $W_2(\pi_{\Theta}^N,\mathcal{L}(\vartheta^{\cdot,N}_{n\lambda}))$ in \eqref{eq:31} given by Proposition  \hyperlink{Prop3}{3} with the bound given by Proposition \hyperlink{Prop7}{4} and respectively the bound for $W_2(\mathcal{L}(\vartheta^{\cdot,N}_{n\lambda}),\mathcal{L}(\theta_{n}^{\lambda,\cdot}))$ given by Proposition  \hyperlink{Prop5}{5} with the bound given by Proposition \hyperlink{Prop6}{6}.
\end{proof}
\noindent It is observed that the dependence on the dimension $d^{\theta}$ of the parameter space in Theorem \hyperref[theorem:2]{2} exhibits a slight deterioration compared to that presented in Theorem \hyperref[theorem:1]{1} and in Theorem 1 of \cite{Ipla}. Notably, the factor $1/N$ is absent in the third term of Theorem \hyperref[theorem:2]{2}. This is interpreted as a compromise entailed by the use of the coordinate-wise taming function \eqref{eq:06}, which requires the direct derivation of moment bounds for the pairs $\mathcal{V}_t^{i,\cdot,N}$, as discussed in Subsection \hyperref[essential]{2.3}. Consequently, this approach results in the loss of the symmetrical structure that is attained by considering the $N$-particle system $\mathcal{X}_t^{i,\cdot,N}$ in its entirety.\\
\noindent \begin{corollary}
Consider the iterates of \hyperref[algo1]{tIPLAu}, $({\theta}^{\lambda,u}_{n})_{n\geq 0}$ and let $\epsilon>0$. Then, for $N\geq9d^{\theta}/(\mu\epsilon^2)$ and $\lambda\leq\min\{\lambda_{0},\epsilon^2/(9C_0(1+\mu\epsilon^2/9+d^x)^{2\ell+2})\}$, one needs $n\geq (9/\mu)^{2\ell+2}((d^{\theta})^{2\ell+1}/\epsilon^{4\ell+4})C_0^2(1+\mu\epsilon^2/9+d^x)^{2\ell+2}\log(3(R_0+(d^x\epsilon^{4\ell+2}\mu^{2\ell}/(9d^{\theta})^{2\ell+1}+9\epsilon^{4\ell+4}\mu^{2\ell+1}/(9d^{\theta})^{2\ell+1})^{1/2})/\epsilon)$ iterations to achieve 
        \[W_2(\pi_\beta,\mathcal{L}({\theta}^{\lambda,u}_{n}))\leq \epsilon, \]
where $R_0=\left(\mathbb{E}\left[|z_0^N-z^*|^2\right]\right)^{1/2}$. Overall one has to choose $N=\mathcal{O}(1/\epsilon^{2})$, $\lambda=\mathcal{O}(\epsilon^2)$ and $n=\mathcal{O}(\epsilon^{-4(\ell+1)}\log(1/\epsilon))$. The corresponding dimensional dependencies are as follows, $N=\mathcal{O}(d^{\theta})$, $\lambda=\mathcal{O}((d^x)^{-2(\ell+1)})$ and $n=\mathcal{O}((d^{\theta}d^x)^{2\ell+1}\log((d^x)^{1/2}/(d^{\theta})^{\ell+1/2}))$.
\end{corollary}

$\phantom0$

\noindent \begin{corollary}
Consider the iterates of \hyperref[algo2]{tIPLAc}, $({\theta}^{\lambda,c}_{n})_{n\geq 0}$ and let $\epsilon>0$. Then, for $N\geq9d^{\theta}/(\mu\epsilon^2)$ and $\lambda\leq\min\{\lambda_{0},\epsilon^2/(9C_0(1+d^{\theta}+d^x)^{2\ell+2})\}$, one needs $n\geq (9/(\mu\epsilon^2))C_0^2(1+d^{\theta}+d^x)^{2\ell+2}\log(3(R_0+(d^x/\mu+\epsilon^2/9)^{1/2})/\epsilon)$ iterations to achieve 
        \[W_2(\pi_\beta,\mathcal{L}({\theta}^{\lambda,c}_{n}))\leq \epsilon, \]
where $R_0=\left(\mathbb{E}\left[|z_0^N-z^*|^2\right]\right)^{1/2}$. Overall one has to choose $N=\mathcal{O}(1/\epsilon^{2})$, $\lambda=\mathcal{O}(\epsilon^2)$ and $n=\mathcal{O}(\epsilon^{-2}\log(1/\epsilon))$.The corresponding dimensional dependencies are as follows, $N=\mathcal{O}(d^{\theta})$, $\lambda=\mathcal{O}((d^x+d^{\theta})^{-2(\ell+1)})$ and $n=\mathcal{O}((d^x+d^{\theta})^{2\ell+2}\log(d^x))$.
\end{corollary}
\subsection{Summary of contributions and discussion}\label{summary}
This article aims to expand the interacting particle-based gradient decent methods for marginal likelihood optimization beyond linear growth for the corresponding gradients. The contributions of our work can be summarized as follows:
\begin{itemize}
  \item We provide rigorous results for the treatment of interacting particle systems with superlinear drift coefficients.
  \item For stepsize $\lambda$, we achieve $\lambda^{1/2}$ rate in $W_2$ distance for both tIPLA-based algorithms to the target measure. To the best of our knowledge, this is the first result under such relaxed assumptions.
  \item We provide explicit dimensional dependencies in our constants, exhibiting polynomial scaling determined by the polynomial growth of the drift coefficients.
\end{itemize}
\begin{table*}[h]
\caption{Comparison of algorithmic complexity across existing literature}
\label{table}
\centering
\begin{tabular}{lccc}
\toprule
    &$W_2$
    &\textsc{Convexity}
	&\textsc{Growth}
    \\
\midrule
    PGD \cite{Kuntz} 
	&$\mathcal{O}\left(\epsilon^{-2}\log(1/\epsilon)\right)$
    &strongly convex
	&linear
	\\
	\\
    IPLA \cite{Ipla}
	&$\mathcal{O}\left(\epsilon^{-2}\log(1/\epsilon)\right)$
    &strongly convex
	&linear
	\\
	\\
    tIPLAu
    &$\mathcal{O}\left(\epsilon^{-4(\ell+1)}\log(1/\epsilon)\right)$
    &strongly convex
	&polynomial
	\\
	\\
    tIPlAc
	&$\mathcal{O}\left(\epsilon^{-2}\log(1/\epsilon)\right)$
    &strongly convex
	&polynomial
	\\
	\\
\bottomrule
\end{tabular}
\end{table*}
\noindent In Table \ref{table} we compare the performance of our algorithms with other IPS gradient decent variants that do not accommodate superlinear settings. We can see that our result for \hyperref[algo1]{tIPLAu} does not achieve the same complexity as its competitors. This is a natural drawback since we are not limited to the globally Lipschitz continuous case for gradients. However, \hyperref[algo2]{tIPLAc} compares favorably, attaining the same complexity as its competitors by allowing superlinear growth while imposing element-wise dissipativity. We highlight that, SOUL \cite{Soul} is omitted from Table \ref{table} due to different methodologies being employed: it operates sequentially with a single particle for latent posterior sampling, whereas our approach uses $N$ particles and is parallelizable; additionally, SOUL \cite{Soul} uses a decreasing stepsize approach to characterize almost sure convergence while in this paper a constant stepsize is used to obtain $\mathcal{L}^2$ convergence estimates. Furthermore, we do not include classical MCMC methods in our numerical comparison, as our focus is on gradient-based approaches.\setlength{\parindent}{15pt}
\par\indent The tIPLA algorithm class generalizes IPLA \cite{Ipla} by allowing a wider class of potentials with superlinear gradients to be considered. One observes that elements of the original proof structure appear in this work, while new techniques address challenges unique to the new setting. In particular, the main challenge in the superlinear context is designing a taming function that satisfies critical properties, specifically, linear growth (Property \hyperlink{Property1}{1}) and inheriting the dissipativity of the original gradient (Property \hyperlink{Property3}{3}). Both properties are essential for establishing higher moment bounds, which is key to our proof strategy.\setlength{\parindent}{15pt}
\par\indent
 The steps, to demonstrate both the concentration of the $\theta$-marginal of the invariant measure around the maximizer and the contraction of the IPS, parallel the corresponding proofs for IPLA. The primary differences lie in the technical modifications required to overcome the challenges introduced by aforementioned taming functions. To control the discretization error, we introduce the time-scaling factor $N^p$ in equations \eqref{eq:02}–\eqref{eq:03}. In the uniform taming case, this factor enables us to bound the sum of the higher moments of the paired processes $V_t^{i,\lambda,u}$ by the corresponding rescaled process $Z_t^{\lambda,u}$ (see equation \eqref{c3}). In contrast, for coordinate-wise taming, such a modification is unnecessary because higher moment bounds for $V_t^{i,\lambda,c}$ can be derived directly due to the smoothness Assumption \hyperlink{Assum3a}{3}.
\section{Numerical Experiments}
This section presents the performance evaluation of \hyperref[algo1]{tIPLAu} and \hyperref[algo2]{tIPLAc} by applying them on a set optimization problems. We also include competitive algorithms such as IPLA \cite{Ipla}, PGD \cite{Kuntz}, and SOUL \cite{Soul} for comparison. The results highlight the ability of our methods to operate under more relaxed assumptions in the superlinear setting, where other algorithms face limitations.
\subsection{Bayesian logistic regression with thin tailed priors}
Let $y\in\mathbb{R}^{d^y}$ represent the observable variable, which consists of a set of binary responses $\{y_j\}_{j=1}^{d^y}\in\{0,1\}$, and $x\in\mathbb{R}^{d^x}$, $\theta\in\mathbb{R}^{d^{\theta}}$ the latent variable and the target parameter to be estimated respectively with $d^x=d^{\theta}$. We consider the Bayesian logistic regression model where the conditional distribution of $y$ given $x$ is defined as $$p_{\theta}(y|x)=\prod_{j=1}^{d^y}s(u_j^Tx)^{y_j}(1-s(u_j^Tx))^{1-y_j},$$ where $s(t)=\exp{(t)}/(1+\exp{(t)})$ and $\{u_j\}^{d^y}_{j=1}\in\mathbb{R}^{d^x}$ are fixed covariates corresponding to the observable data. Instead of employing a standard Gaussian prior, we model the prior distribution using a generalized Gaussian with thinner tails, expressed as \begin{align}p_{\theta}(x)=C_1\exp\left(-\dfrac{1}{\sigma^4}\sum_{j=1}^{d^x}|x_j-\theta_j|^4-\dfrac{1}{\sigma^2}\sum_{j=1}^{d^x}|x_j-\theta_j|^2)\right).\label{gengauss}\end{align} Thus, the marginal likelihood can be written as \begin{align*}
   p_\theta(y)&=\int_{\mathbb{R}^{d^x}}p_{\theta}(y|x)p_{\theta}(x)dx=\int_{\mathbb{R}^{d^x}}\exp{(-U_y(\theta,x))}dx,
    \end{align*}
where\begin{align}
    U_y(\theta,x)&=C_2+\dfrac{1}{\sigma^4}\sum_{j=1}^{d^x}|x_j-\theta_j|^4+\dfrac{1}{\sigma^2}\sum_{j=1}^{d^x}|x_j-\theta_j|^2\nonumber-\sum_{j=1}^{d^y}\left(y_j\ln(s(u_j^Tx))+(1-y_j)\ln(s(-u_j^Tx))\right).\label{regU}
\end{align}
Here the constant $C_2$ consists of terms not dependent on either $x$ or $\theta$. Assuming that $\sigma^2$ is a fixed, known constant, we verify that the potential $ U_y(\theta,x)$ satisfies our main assumptions \hyperlink{Assum1}{A1} and \hyperlink{Assum2}{A2}, in view of Remark \ref{remark3} and Remark \ref{remark4}.\newline \par \noindent \textbf{Experiment details}. Following a similar setup with \cite{Ipla}, we choose $d^x = d^{\theta} = 3$ and $d^y = 900$ and set $\theta^* = [2, 5, -1]$ to test the convergence of the parameter. We run \hyperref[algo1]{tIPLAu} for $M = 200\text{k}$ steps, using $N = 10, 100, 1000$ particles and a stepsize of $\lambda = 0.0001$. We set $\sigma^2 = 0.1$ and simulate the covariates $v_j$ from the uniform distribution $U(-1,1)$ over $\mathbb{R}^{d^x}$ . The parameter ${\theta}_0$ is initialized from a deterministic value $[100, -100, 0]$, while the particles $X^{i,N}_0$ are drawn from the generalized Gaussian \eqref{gengauss} with mean $\theta_0$ and variance $\sigma^2=0.1$.\\
\indent As shown in Fig. \ref{fig1}, we observe that as $N$ increases, the convergence time to the true value decreases, with a pronounced effect when $N$ exceeds 1000. However, as seen in Fig.\ref{fig2}, which displays the last 10,000 iterations of each simulation, there is no significant impact on the residual bias of the algorithm once convergence has been achieved. This suggests that the source of the bias in this case comes from the discretization error, which remains largely the same with respect to $N$ and is primarily dependent on the problem-specific constants and the chosen stepsize. Moreover, it is evident that as $N$ increases, the Markov chains generated by the algorithm become less noisy, as one would expect from Langevin based algorithms when the inverse temperature parameter is increased. 
\begin{figure}[ht]
    \centering
    \begin{minipage}{0.33\textwidth}
        \centering
        \includegraphics[width=\textwidth]{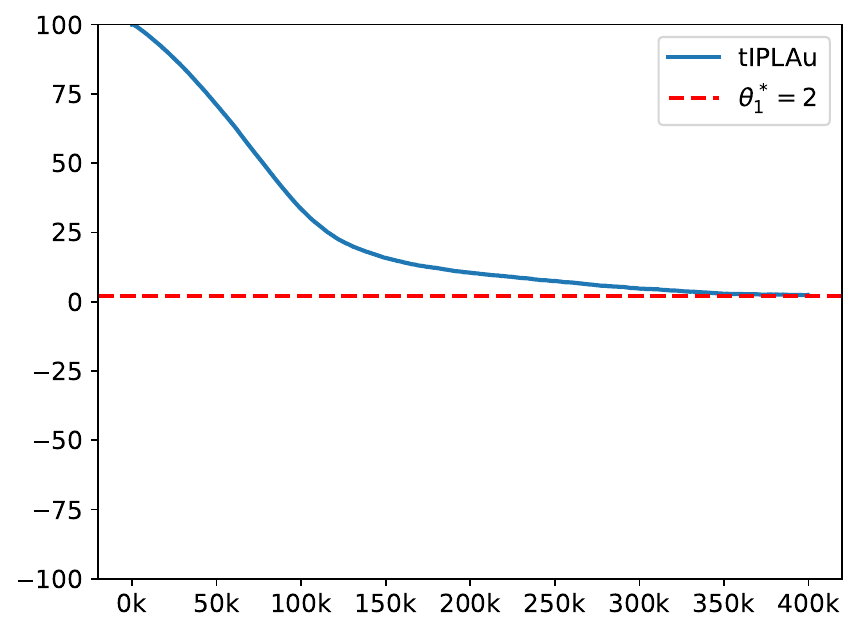}
    \end{minipage}\hfill
    \begin{minipage}{0.33\textwidth}
        \centering
        \includegraphics[width=\textwidth]{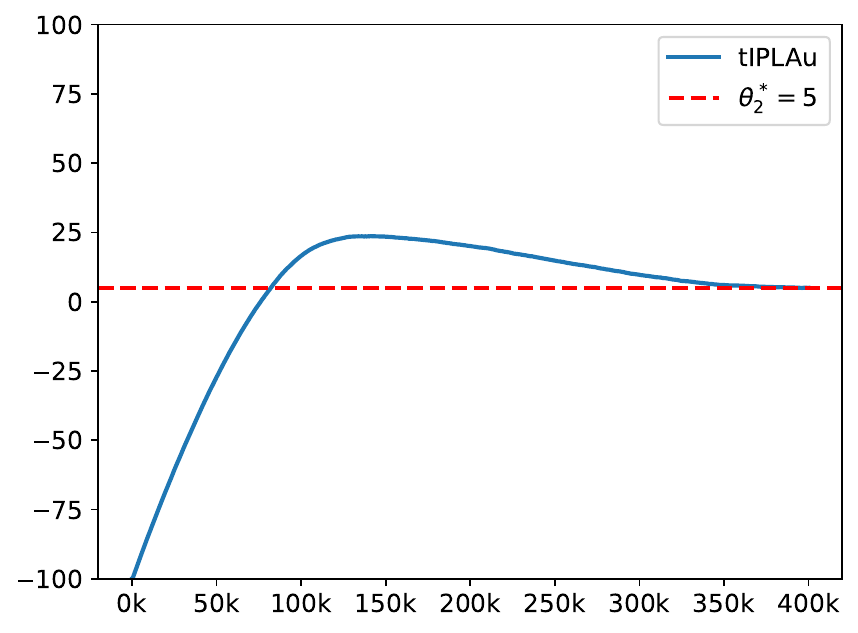}
    \end{minipage}\hfill
    \begin{minipage}{0.33\textwidth}
        \centering
        \includegraphics[width=\textwidth]{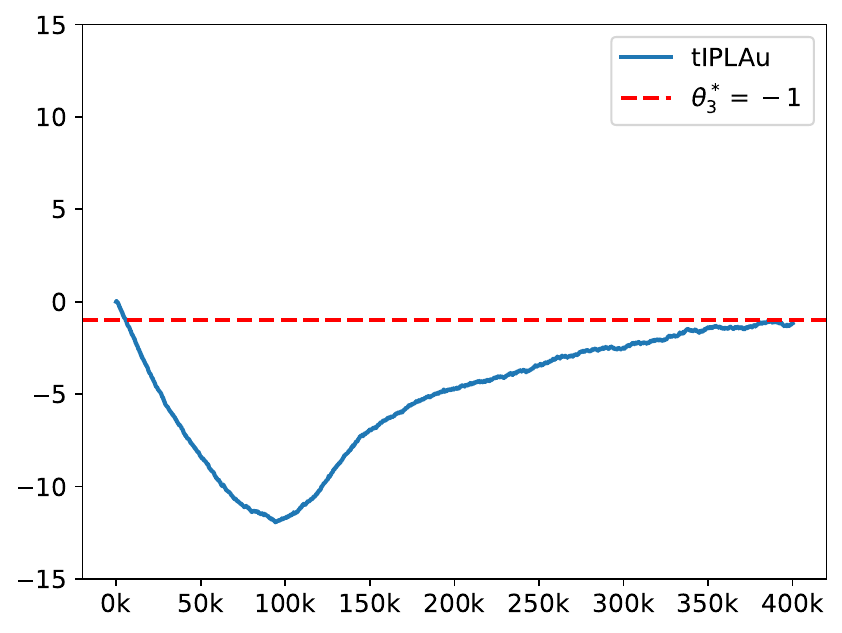}
    \end{minipage}
    \vfill
    \begin{minipage}{0.33\textwidth}
        \centering
        \includegraphics[width=\textwidth]{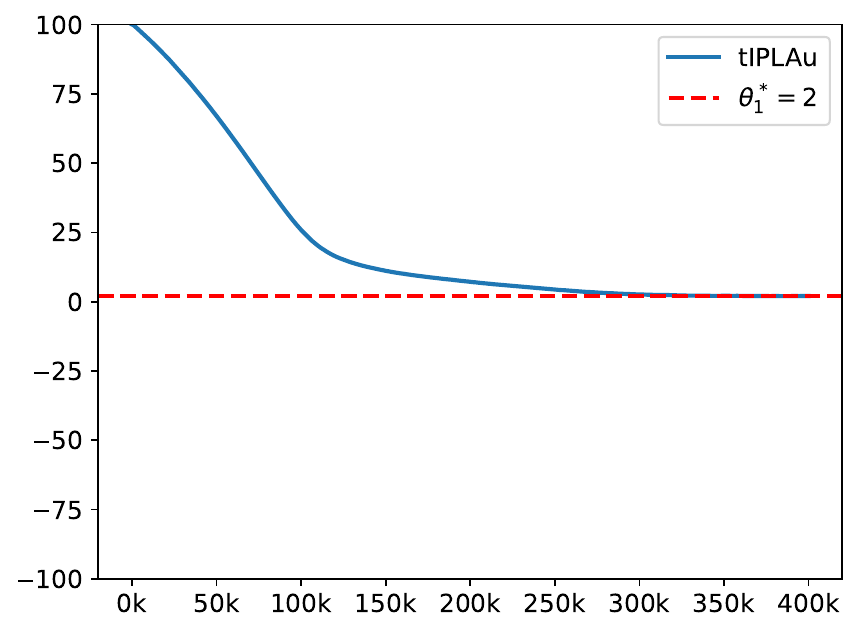}
    \end{minipage}\hfill
    \begin{minipage}{0.33\textwidth}
        \centering
        \includegraphics[width=\textwidth]{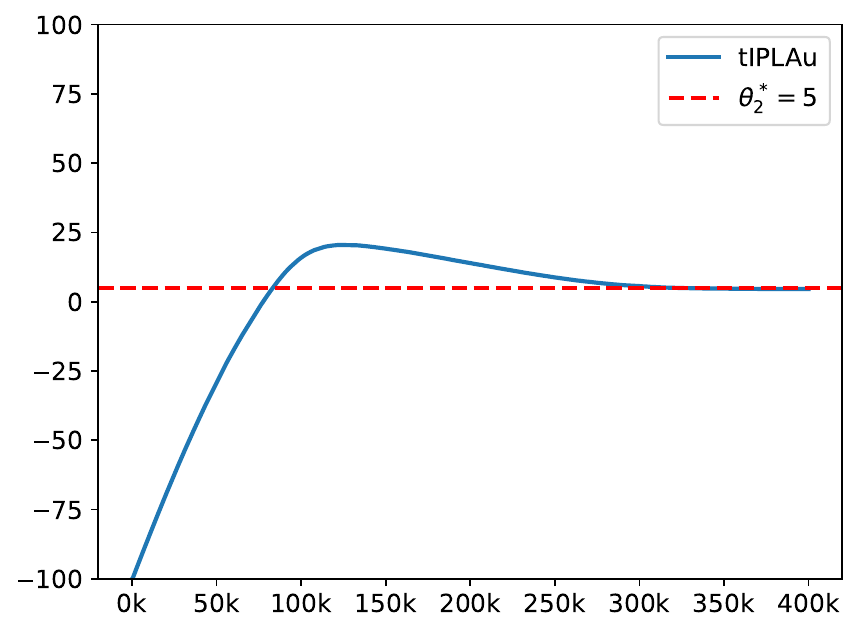}
    \end{minipage}\hfill
    \begin{minipage}{0.33\textwidth}
        \centering
        \includegraphics[width=\textwidth]{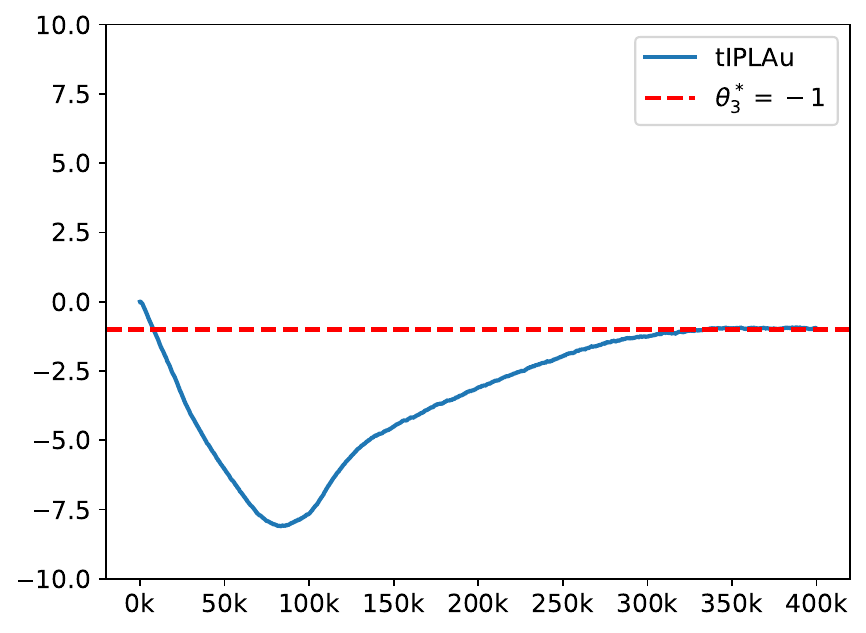}
    \end{minipage}
    \vfill
    \begin{minipage}{0.33\textwidth}
        \centering
        \includegraphics[width=\textwidth]{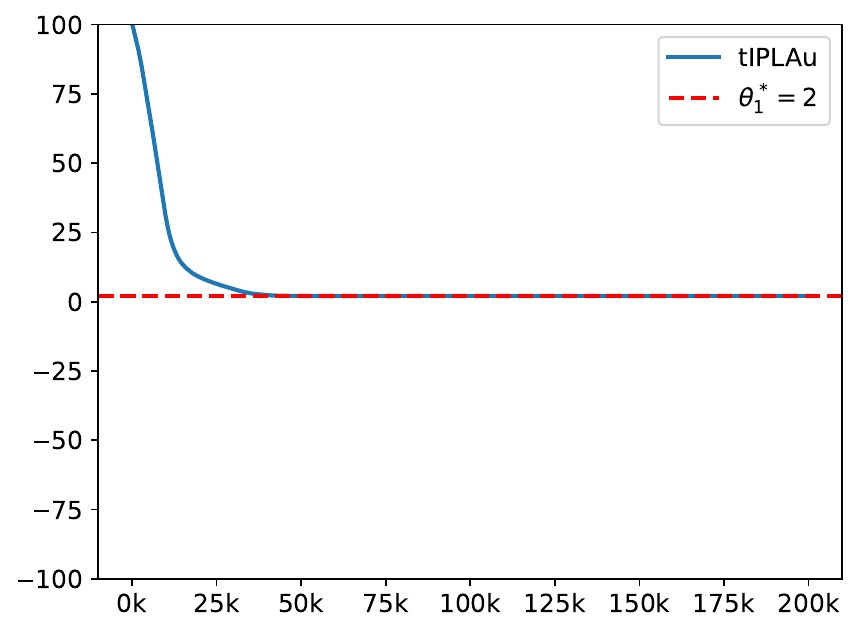}
    \end{minipage}\hfill
    \begin{minipage}{0.33\textwidth}
        \centering
        \includegraphics[width=\textwidth]{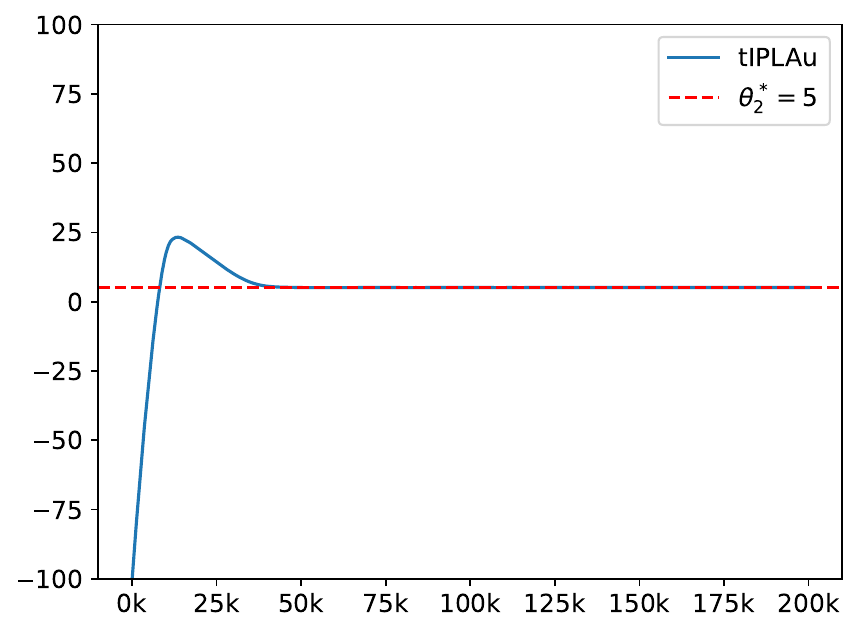}
    \end{minipage}\hfill
    \begin{minipage}{0.33\textwidth}
        \centering
        \includegraphics[width=\textwidth]{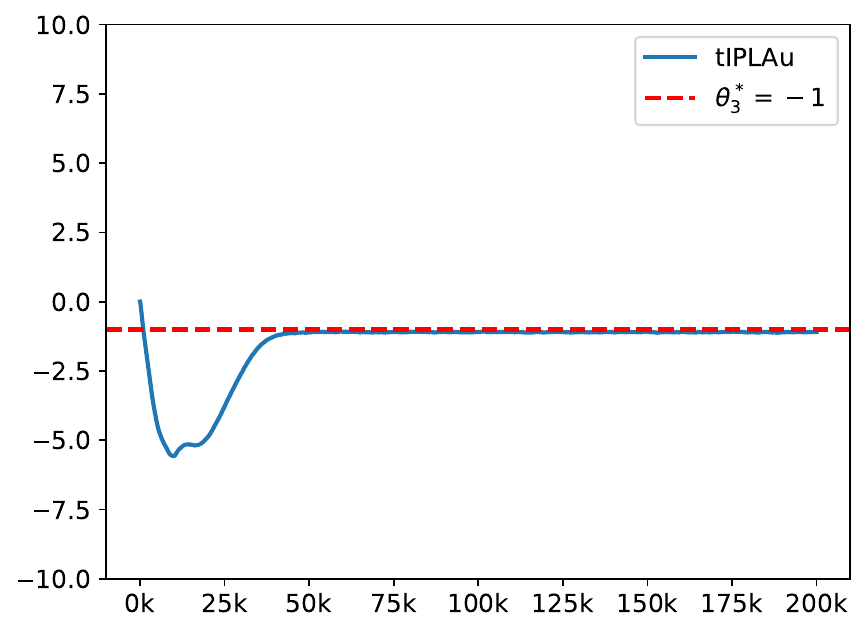}
    \end{minipage}
    \caption{The performance of \hyperref[algo1]{tIPLAu} on the synthetic logistic regression problem. Each column corresponds to one of the components of the true parameter $\theta^*$ = [2, 5, -1]. The top row corresponds to $N=10$, the middle row to $N=100$, and the bottom row to $N=1000$}
    \label{fig1}
\end{figure}
\begin{figure}[ht]
    \centering
    \begin{minipage}{0.32\textwidth}
        \centering
        \includegraphics[width=\textwidth]{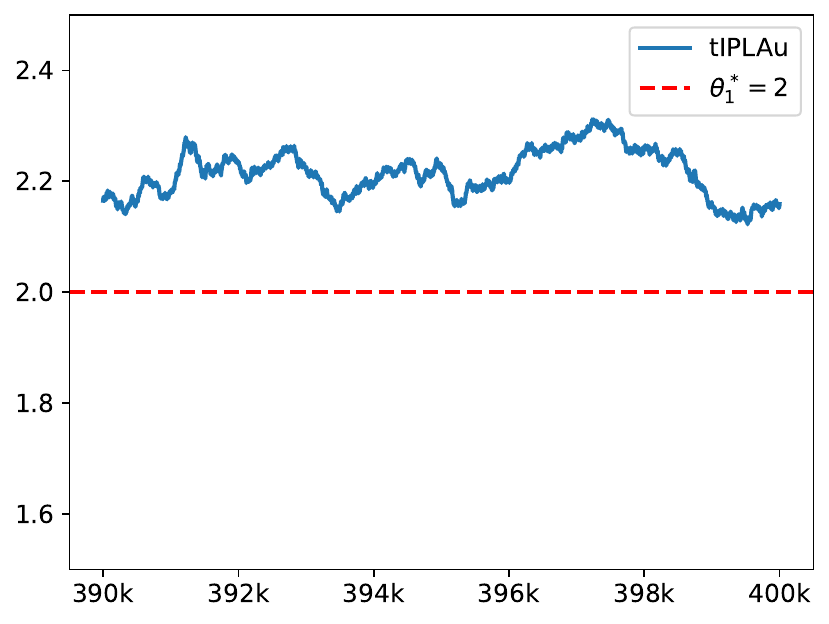}
    \end{minipage}\hfill
    \begin{minipage}{0.32\textwidth}
        \centering
        \includegraphics[width=\textwidth]{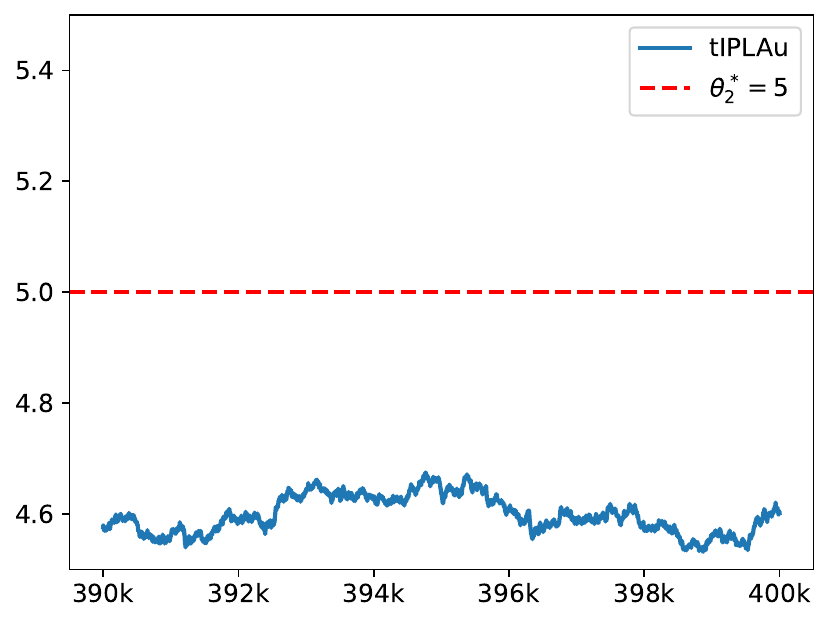}
    \end{minipage}\hfill
    \begin{minipage}{0.32\textwidth}
        \centering
        \includegraphics[width=\textwidth]{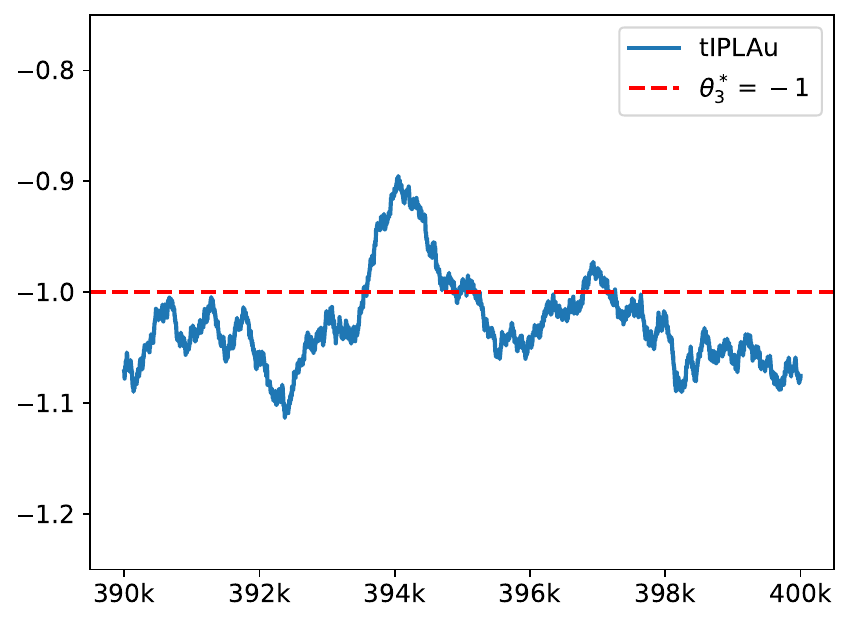}
    \end{minipage}
    \vfill
    \begin{minipage}{0.33\textwidth}
        \centering
        \includegraphics[width=\textwidth]{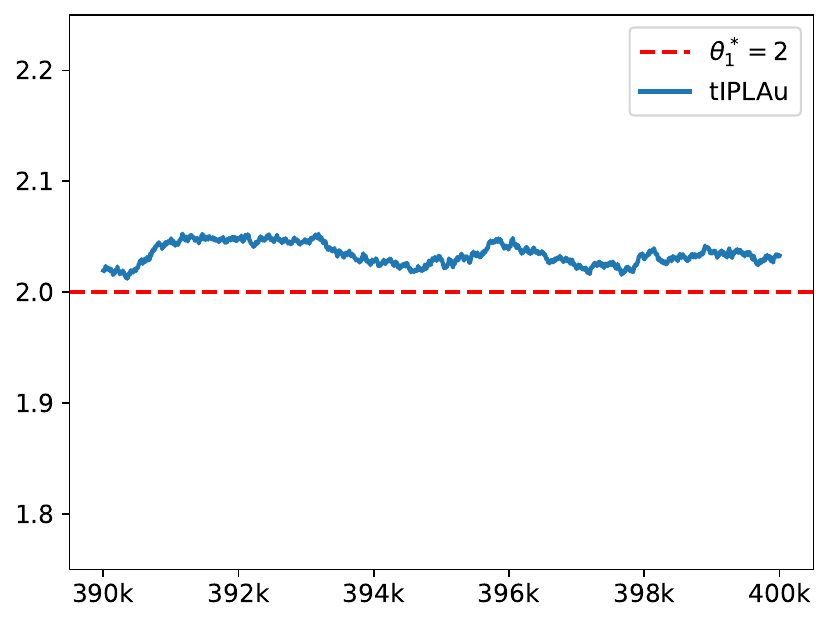}
    \end{minipage}\hfill
    \begin{minipage}{0.33\textwidth}
        \centering
        \includegraphics[width=\textwidth]{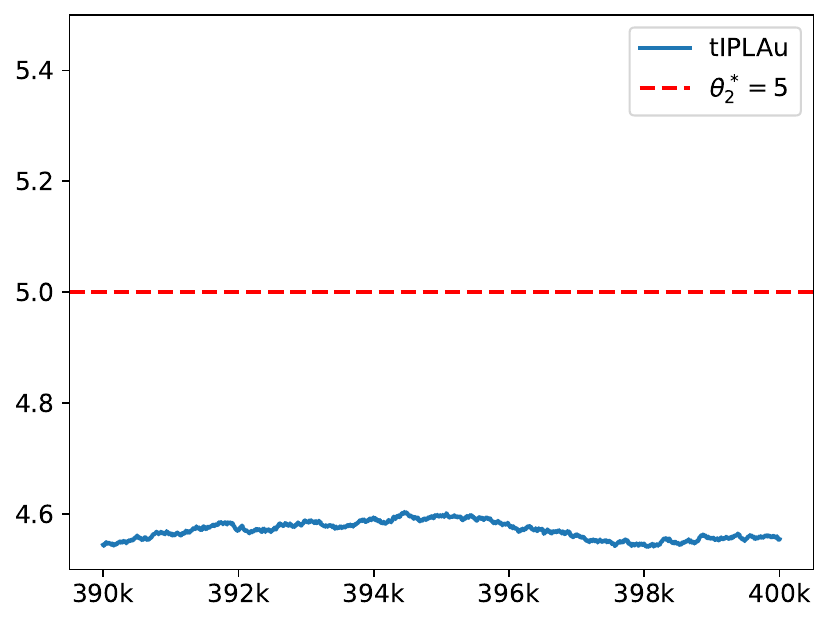}
    \end{minipage}\hfill
    \begin{minipage}{0.33\textwidth}
        \centering
        \includegraphics[width=\textwidth]{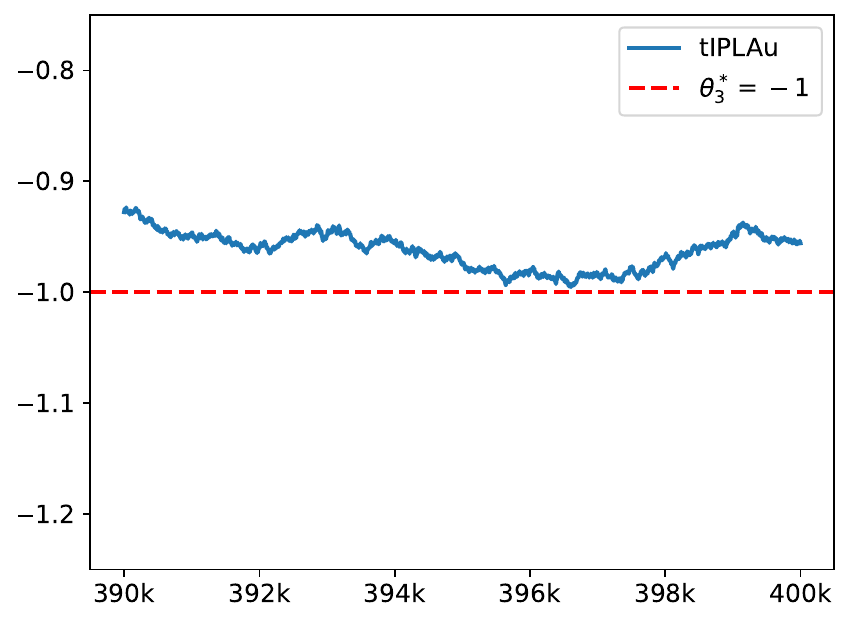}
    \end{minipage}
    \vfill
    \begin{minipage}{0.33\textwidth}
        \centering
        \includegraphics[width=\textwidth]{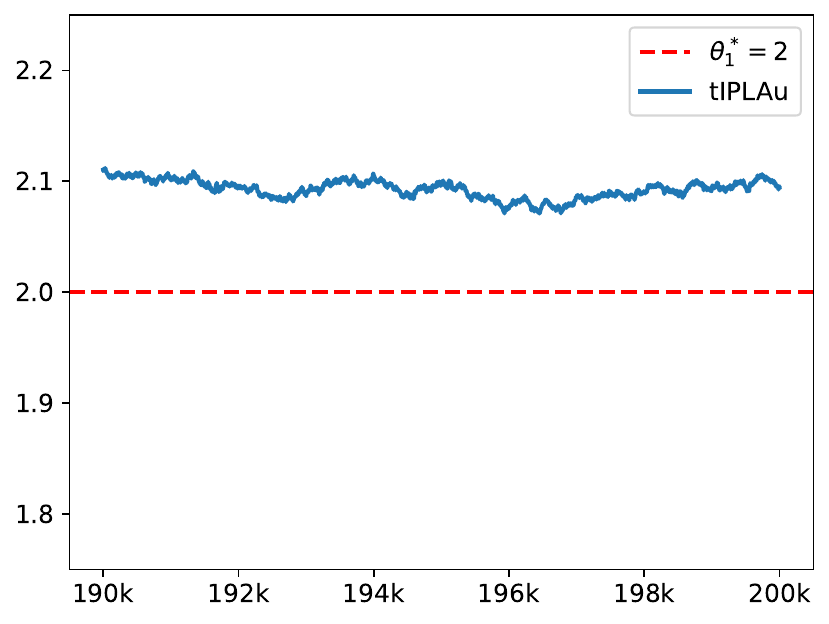}
    \end{minipage}\hfill
    \begin{minipage}{0.33\textwidth}
        \centering
        \includegraphics[width=\textwidth]{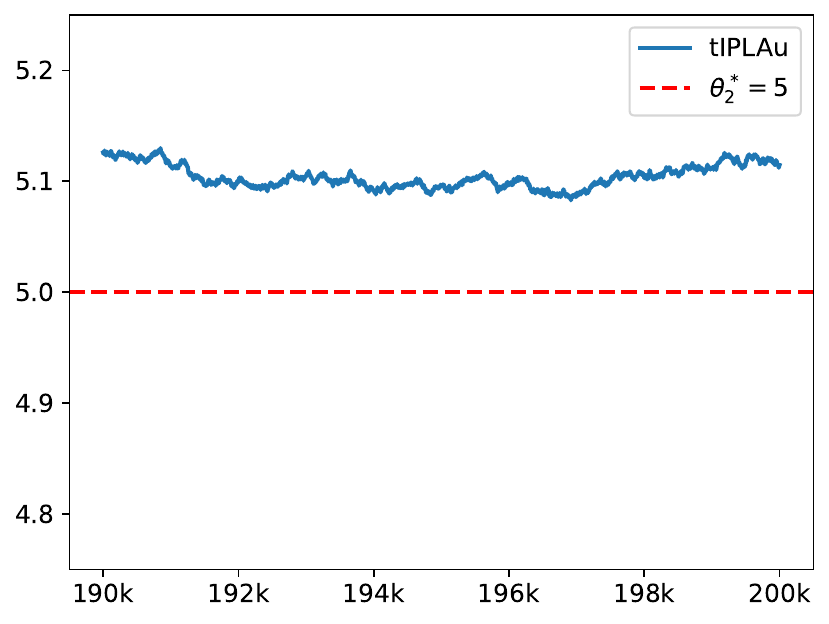}
    \end{minipage}\hfill
    \begin{minipage}{0.33\textwidth}
        \centering
        \includegraphics[width=\textwidth]{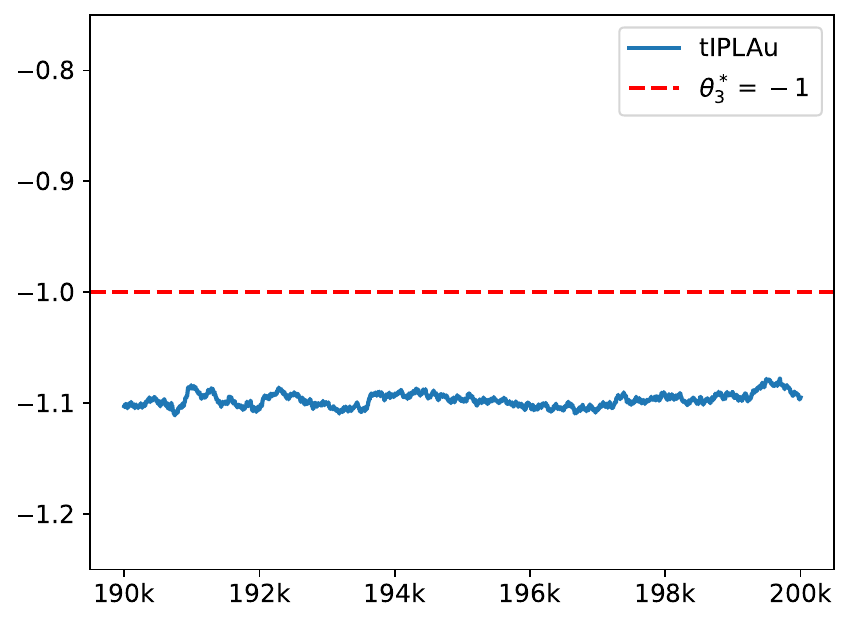}
    \end{minipage}
    \caption{The last 10k iterates from \hyperref[algo1]{tIPLAu}'s performance as shown in Fig 1. Each column corresponds to one of the component of the true parameter $\theta^*$ = [2, 5, -1]. The top row corresponds to $N=10$, the middle row to $N=100$, and the bottom row to $N=1000$}
    \label{fig2}
\end{figure}
\subsection{Higher order - Toy example}\label{toy}
To demonstrate the applicability and versatility of our framework we consider the following highly superlinear toy example where an intractable integral has to be maximized with respect to the parameter $\theta$. Let for any $m\geq1$ the objective function to be defined as
\begin{align*} k(\theta)=\int_{\mathbb{R}^{d^x}}\exp\left(-|x|^{4m}-(|x|^{2m}+1)(|\theta|^{2m}+1)-|\theta|^{4m}\right.\\ \left.-|x|^4-(|x|^2+1)(|\theta|^2+1)-|\theta|^4\right)dx.\end{align*}
In this case the maximiser $\theta^*=\text{argmax}\{k(\theta)\}$ can easily be derived to be $\theta^*=0$ which helps us verify our experimental foundings. \noindent One simply observes that \begin{gather*}\exp\left(-|x|^{4m}-(|x|^{2m}+1)(|\theta|^{2m}+1)-|\theta|^{4m}-|x|^4-(|x|^2+1)(|\theta|^2+1)-|\theta|^4\right)\\\leq \exp(\left(-|x|^{4m}-(|x|^{2m}+1)-|x|^4-(|x|^2+1)\right),\end{gather*}
with the equality holding if and only if $\theta=0$. Hence by taking the definite integral with respect to $dx$ over the whole $\mathbb{R}^{d^x}$ we have
$$k(\theta)\leq \int_{\mathbb{R}^{d^x}}\exp{\left(-|x|^{4m}-(|x|^{2m}+1)-|x|^4-(|x|^2+1)\right)}dx:=k(\theta^*).$$
Now the potential $U(\theta,x)$ can be written as
\begin{align}U(\theta,x)=|x|^{4m}+(|x|^{2m}+1)(|\theta|^{2m}+1)+|\theta|^{4m}+|x|^4+(|x|^2+1)(|\theta|^2+1)+|\theta|^4,\label{toyU}\end{align}
with the following gradient structure
\begin{gather*} 
\nabla_{x}U=4m|x|^{4m-2}x+2m(|\theta|^{2m}+1)|x|^{2m-2}x+4|x|^2x+2(|\theta|^2+1)x,\\
\nabla_{\theta}U=4m|\theta|^{4m-2}\theta+2m(|x|^{2m}+1)|\theta|^{2m-2}\theta+4|\theta|^2\theta+2(|x|^2+1)\theta,
\end{gather*}
\noindent In light of Remark \ref{remark5} and Remark \ref{remark6} we verify that both \hyperlink{Assum1}{A1} and \hyperlink{Assum2}{A2} hold, in particular we have
\begin{gather*}
    |\nabla U(v)-\nabla U(v')|\leq3m2^{4m}\left(1+|v|^{4m-2}+|v'|^{4m-2}\right)|v-v'|,\\
    \left\langle v-v',\nabla U(v)-\nabla U(v')\right\rangle\geq 2|v-v'|^2,
\end{gather*}
where $v=(\theta,x)$. To check that \hyperlink{Assum3a}{A3-i} holds, we easily observe that
\begin{align*}
    \nabla_{\theta_i}U\theta_i\geq 2\theta^2_i \text{ and } \nabla_{x_i}U x_i\geq 2x^2_i.
\end{align*}
\begin{figure}[ht]
    \centering
    \begin{minipage}{0.5\textwidth}
        \centering
        \includegraphics[width=\textwidth]{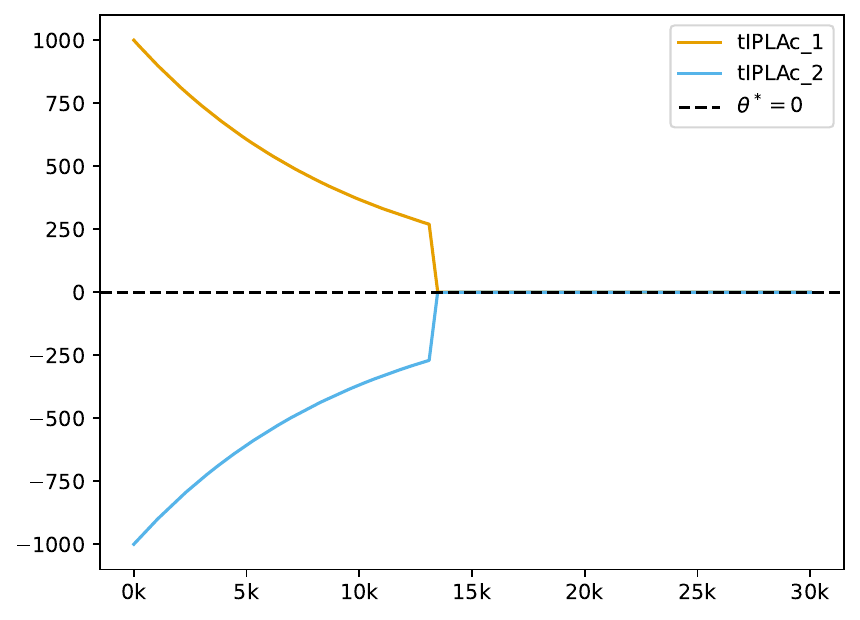}
    \end{minipage}\hfill
    \begin{minipage}{0.5\textwidth}
        \centering
        \includegraphics[width=\textwidth]{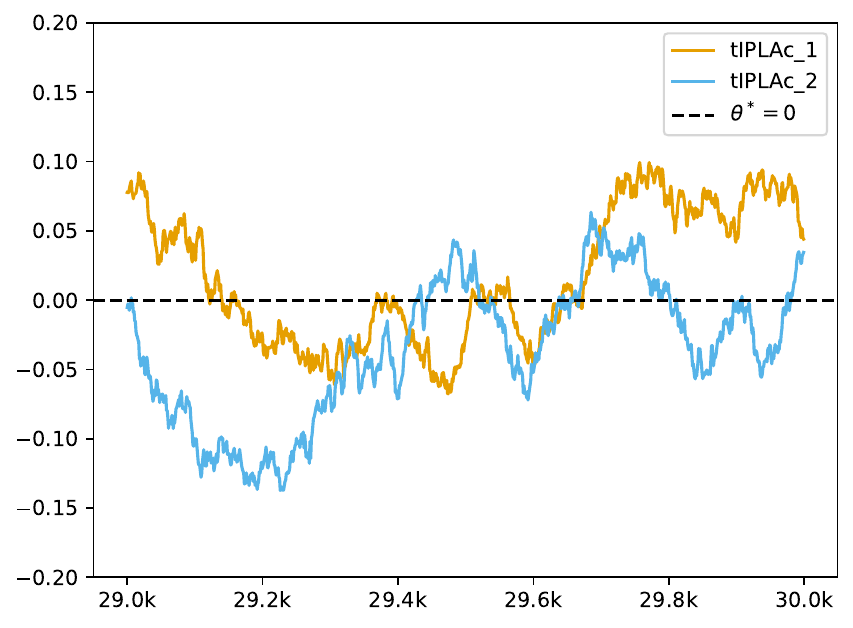}
    \end{minipage}
    \vfill  
    \begin{minipage}{0.5\textwidth}
        \centering
        \includegraphics[width=\textwidth]{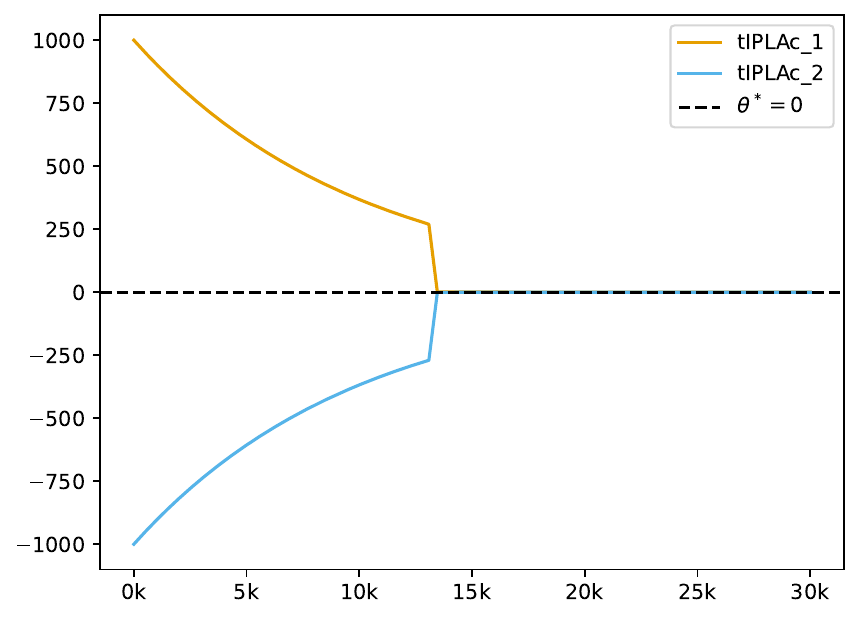}
    \end{minipage}%
    \hfill
    \begin{minipage}{0.5\textwidth}
        \centering
        \includegraphics[width=\textwidth]{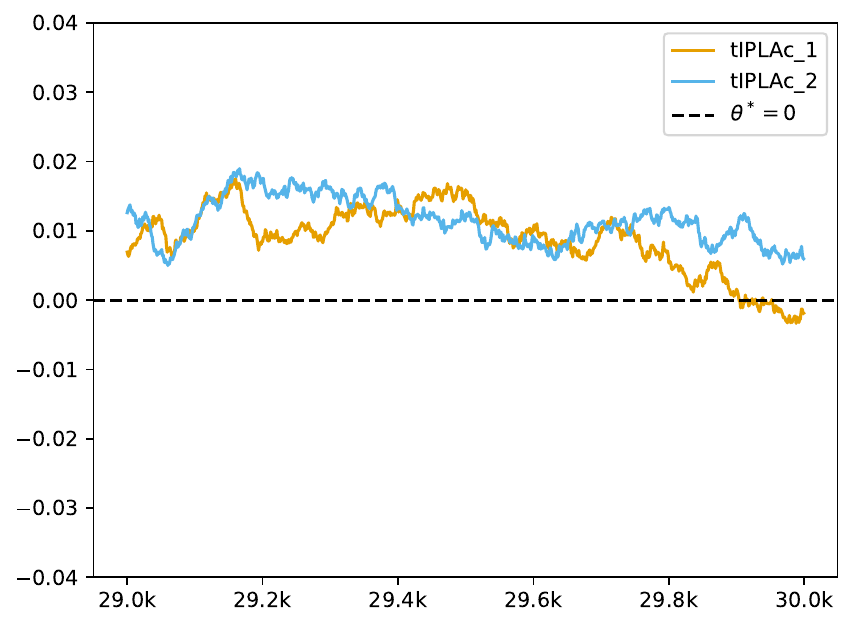}
    \end{minipage}
    \vfill  
    \begin{minipage}{0.5\textwidth}
        \centering
        \includegraphics[width=\textwidth]{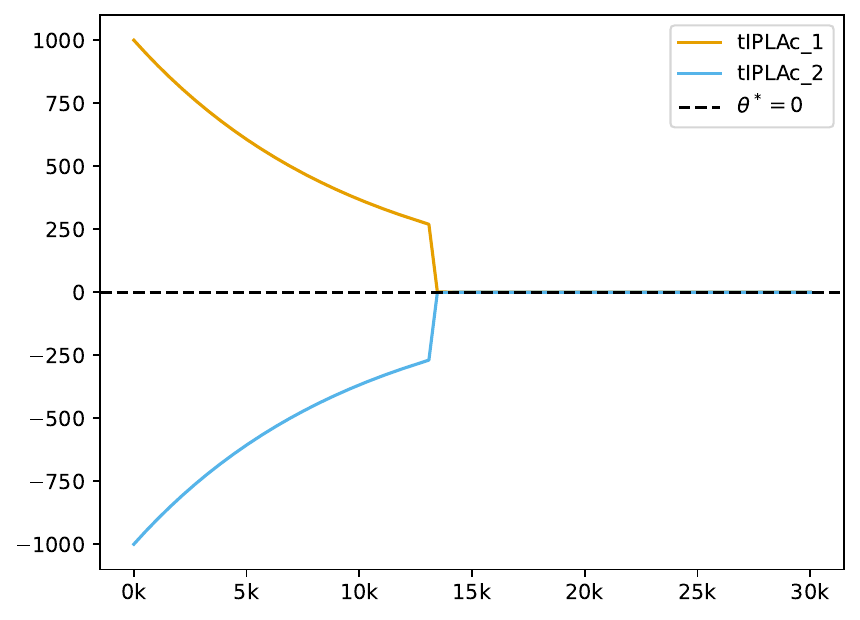}
    \end{minipage}%
    \hfill
    \begin{minipage}{0.5\textwidth}
        \centering
        \includegraphics[width=\textwidth]{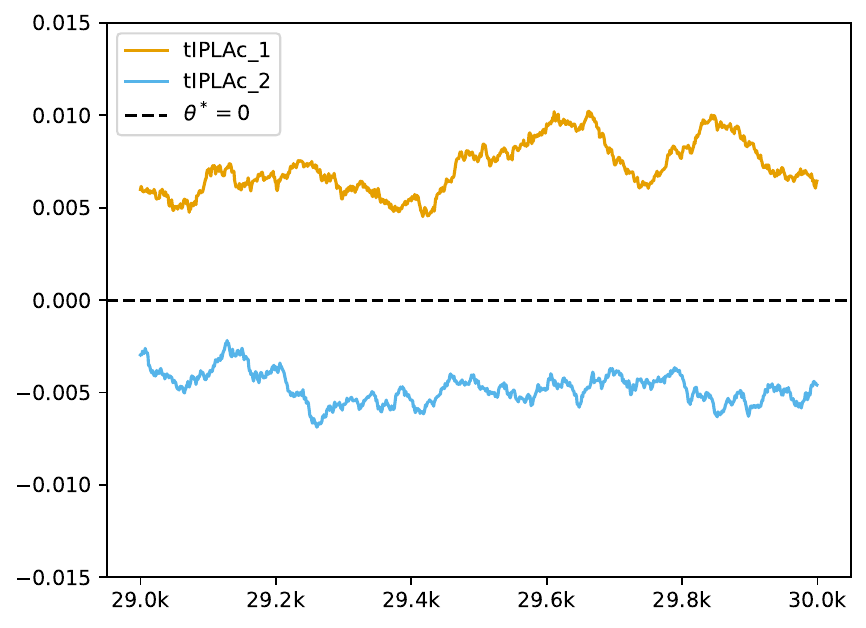}
    \end{minipage}
    \caption{The performance of \hyperref[algo2]{tIPLAc} on the superlinear toy problem for $m=15$. The top row corresponds to $N=100$, the middle row to $N=1000$, and the bottom row to $N=10,000$. The left column displays the full iterates of the algorithm, while the right column focuses on the last 1000 iterations to examine any residual bias in the limiting behavior}
    \label{fig3}
\end{figure}
\clearpage
\noindent \textbf{Experiment details}.
\noindent We set $m=15$ to a achieve a highly superlinear regime, where terms of the $60^{\text{th}}$ order appear. For the experiment, we choose $d^x=100$ and $d^{\theta} = 2$, with the true parameter $\theta^*$ known to be the zero vector $(0,0)$ to test parameter convergence. We run \hyperref[algo2]{tIPLAc}, for $M = 30\text{k}$ iterations, using $N = 100, 1000, 10000$ particles and a stepsize of $\lambda = 0.0001$. The parameter ${\theta}_0$ is initialized on purpose from a very high deterministic value $[1000, -1000]$, while the particles $X^{i,N}_0$ are drawn from a Gaussian distribution with randomized mean over the interval $(-100,100)$ for each coordinate, and covariance matrix $R=10I_{d^x}$.\\
\indent As shown in Fig. \ref{fig3}, we observe that as $N$ increases, the convergence time to the true value remains unchanged, contrary to the results from the \hyperref[algo1]{tIPLAu} algorithm. This is explained by the absence of $N$ in the exponent of the second error term in Theorem \ref{theorem:2}. Furthermore, focusing on the second column, which displays the last 1000 iterations of each simulation, we notice a significant decrease of the residual bias of the algorithm once convergence has been achieved. This aligns with our theoretical finding, where we observe an one order of magnitude reduction in bias corresponding to a two order of magnitude increase in $N$. Additionally, it is evident that as $N$ increases, the iterates of the algorithm generate less noisy Markov chains, confirming that $N$ acts as the inverse temperature parameter in our setting (Fig. \ref{fig4}).\\
\begin{figure}[ht]
    \centering
    \includegraphics[width=0.5\textwidth]{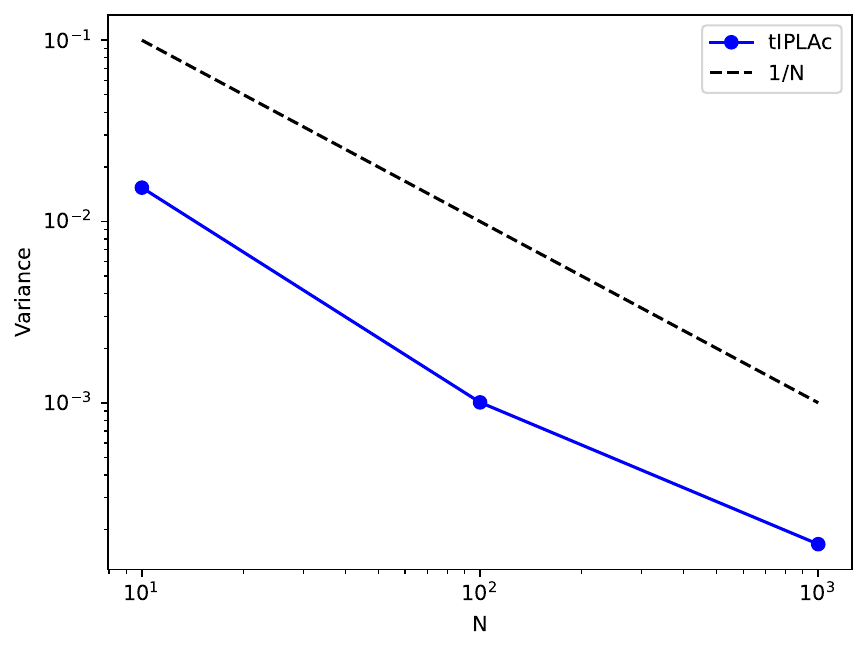} 
    \caption{The convergence rate of the variance of the parameter estimates produced by tIPLAc over 100 Monte Carlo runs for each $N\in(10,100,100)$. We verify that the $\mathcal{O}\left(1/N\right)$ convergence rate holds for the second moments as suggested by our theoretical results}
    \label{fig4}
\end{figure}
 Furthermore, we deploy \hyperref[algo2]{tIPLAc} in par with the other aforementioned interacting particles based algorithms, that is IPLA \cite{Ipla}, PGD \cite{Kuntz} and SOUL \cite{Soul} in a moderate superlinear setting with $m=1$ and $N=2$ number of particles. Once again we choose $d^x=100$ and $d^{\theta}=2$ , with the true parameter $\theta^*$ known to be the zero vector $(0,0)$ to test parameter convergence. We run all all the algorithms for $M=3000$ iterations and a stepsize of $\lambda=0.0001.$ The parameter ${\theta}_0$ is initialized from a deterministic value $[3.28345,-3.28345]$, while the particles $X^{i,N}_0$ are drawn from a Gaussian distribution with randomized mean over the interval $(-0.1,0.1)$ for each coordinate, and covariance matrix $R=0.1I_{d^x}$.\\
 \indent As illustrated in Fig \ref{fig5}, \hyperref[algo2]{tIPLAc} demonstrates rapid convergence to the true parameter. In contrast, the non-tamed algorithms exhibit oscillatory behavior, ultimately diverging to infinity. This divergence arises due to the lack of theoretical guarantees in handling superlinearities, which are intrinsic to the tIPLA algorithms. These observations align with existing literature on unadjusted Langevin-based algorithms, where it is well-established that, in a superlinear setting, the iterates tend to diverge given sufficiently poor initial conditions. In this simple case, although the initial condition is only an order of magnitude away from the ground truth, divergence is still observed. Specifically, SOUL diverges as early as the third iteration, while IPLA and PGD exhibit divergence by the fifth iteration.\\
\indent For the subsequent simulations, we set $m=6$ to a achieve a mild superlinear regime with $12^{\text{th}}$ order terms. We compare \hyperref[algo1]{tIPLAu} and \hyperref[algo2]{tIPLAc} under high dimensional settings. In particular we examine two cases per algorithm, namely $d^{\theta}=d^x=100$ and $d^{\theta}=d^x=1000$, with the true parameter $\theta^*$ known to be the zero vector to test parameter convergence. For \hyperref[algo1]{tIPLAu}, we run $M = 50\text{k}$ iterations, using $N = 50$ particles and a stepsize of $\lambda = 0.0001$. Respectively \hyperref[algo2]{tIPLAc} is run, for $M = 20\text{k}$ iterations, using $N = 50$ particles and a stepsize of $\lambda = 0.0001$. The parameter $\theta$ is a $d^{\theta}$-dimensional vector where each coordinate is independently initialized to either $1000$ or $-1000$, while the particles $X^{i,N}_0$ are drawn from a Gaussian distribution with randomized mean over the interval $(-10,10)$ for each coordinate, and covariance matrix $R=10I_{d^x}$.\\
As illustrated in Fig. \ref{fig6}, focusing on the first two rows, \hyperref[algo2]{tIPLAc} converges to the true parameter significantly faster than \hyperref[algo1]{tIPLAu} in both settings. This aligns with our theoretical results, as the convergence rate of  \hyperref[algo1]{tIPLAu} deteriorates with increasing $N$ due to time-scaling effects, whereas \hyperref[algo2]{tIPLAc} remains unaffected. Additionally, both algorithms exhibit slower convergence when the problem's dimensionality increases by an order of magnitude, consistent with the established Corollaries.\\
Examining the last two rows, which display the final 1500 iterations, we observe that the residual bias of  \hyperref[algo1]{tIPLAu} is more sensitive to dimensionality, whereas \hyperref[algo2]{tIPLAc} remains robust. Moreover, \hyperref[algo2]{tIPLAc} consistently yields lower residual bias than  \hyperref[algo1]{tIPLAu}. It is expected for coordinate-wise taming algorithms such as \hyperref[algo2]{tIPLAc} to outperform their uniform counterparts, as the latter indiscriminately modify all coordinates, failing to distinguish between stable and exploding ones. When the additional smoothness Assumption \hyperlink{Assum3a}{3} holds, \hyperref[algo2]{tIPLAc} is preferable to \hyperref[algo1]{tIPLAu}.
\begin{figure}[ht]
    \centering
    \begin{minipage}{0.5\textwidth}
        \centering
        \includegraphics[width=\textwidth]{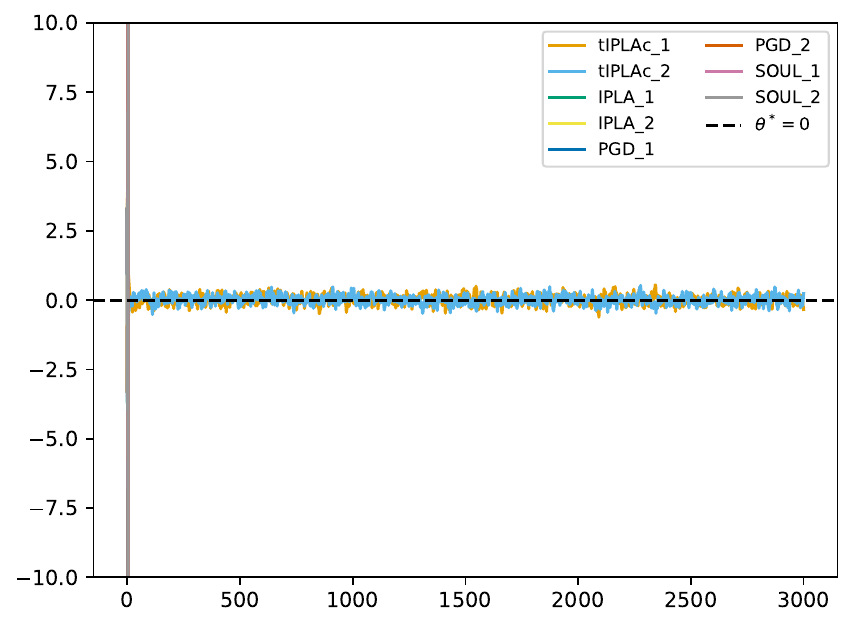}
    \end{minipage}\hfill
    \begin{minipage}{0.5\textwidth}
        \centering
        \includegraphics[width=\textwidth]{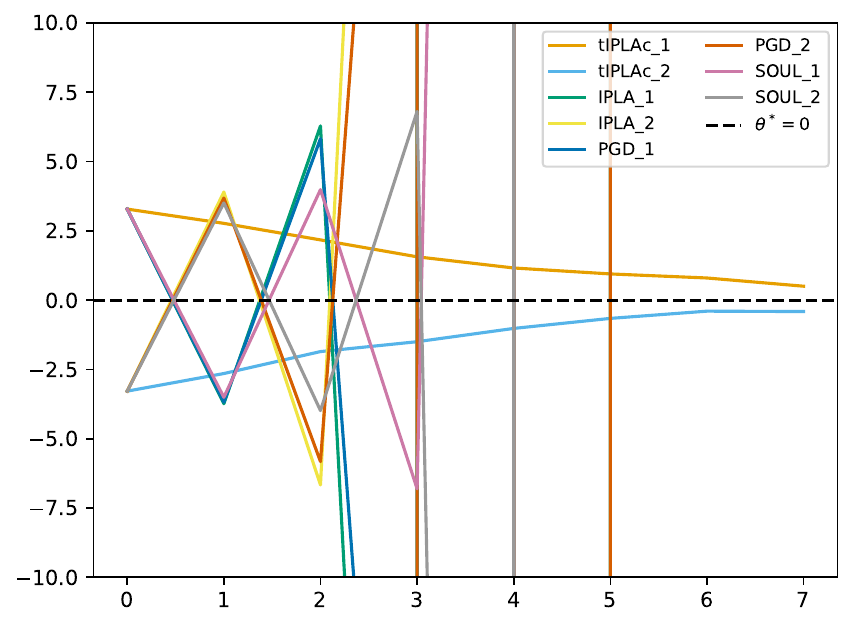}
    \end{minipage}
    \caption{The performance of \hyperref[algo2]{tIPLAc}, IPLA, PGD and SOUL on the superlinear toy problem for $m=1$ and $N=100$. The left column displays the full iterates of the algorithms, while the right column focuses on the first 8 iterations where the explosive behavior of the non-tamed algorithms is observed}
    \label{fig5}
\end{figure}
\begin{figure}[ht]
    \centering
    \begin{minipage}{0.5\textwidth}
        \centering
        \includegraphics[width=\textwidth,
  height=0.225\textheight,
  keepaspectratio=false]{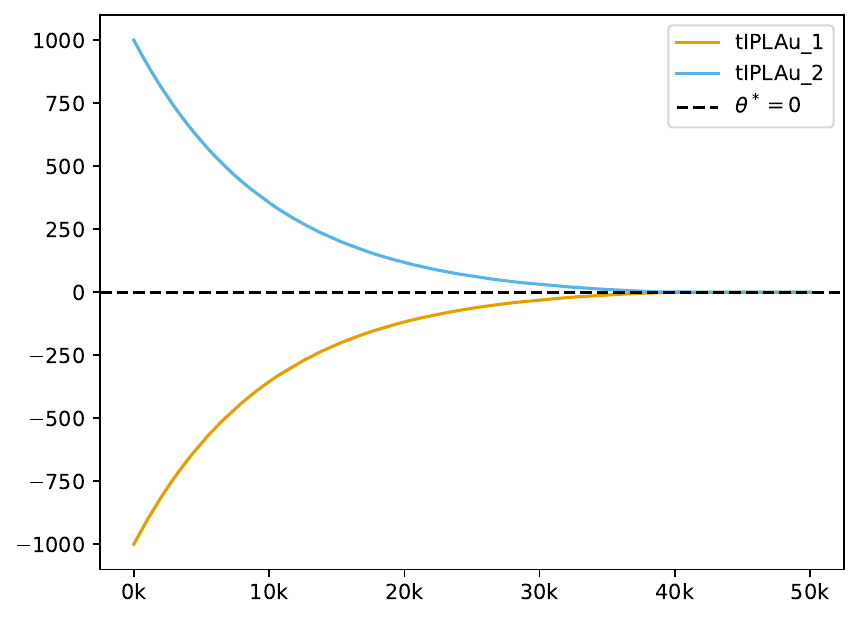}
    \end{minipage}\hfill
    \begin{minipage}{0.5\textwidth}
        \centering
        \includegraphics[width=\textwidth,
  height=0.225\textheight,
  keepaspectratio=false]{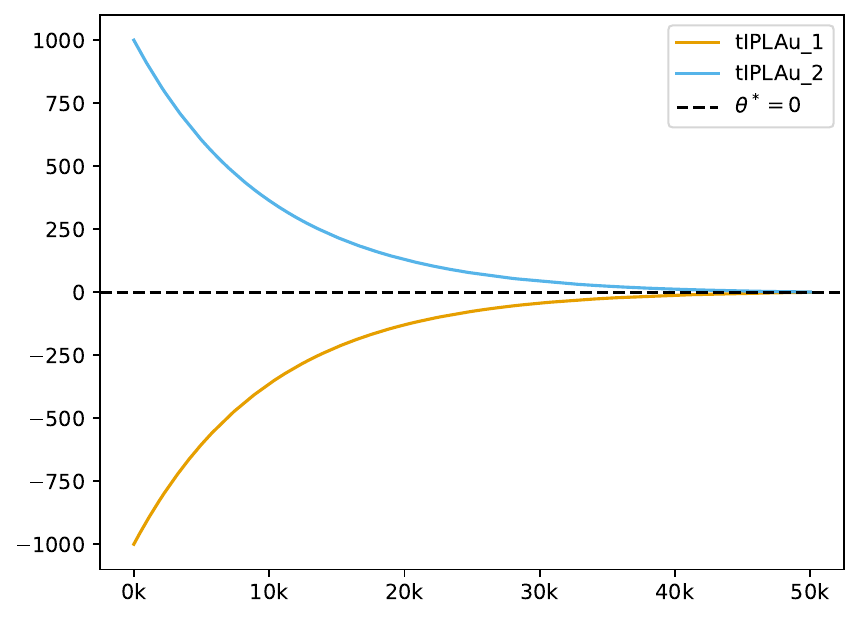}
    \end{minipage}
    \vfill
    \centering
    \begin{minipage}{0.5\textwidth}
        \centering
        \includegraphics[width=\textwidth,
  height=0.225\textheight,
  keepaspectratio=false]{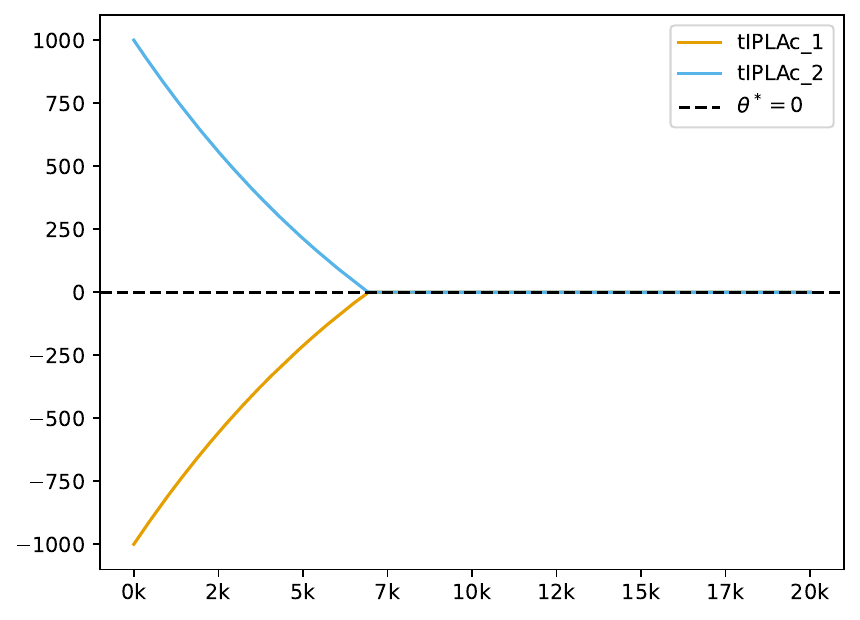}
    \end{minipage}\hfill
    \begin{minipage}{0.5\textwidth}
        \centering
        \includegraphics[width=\textwidth,
  height=0.225\textheight,
  keepaspectratio=false]{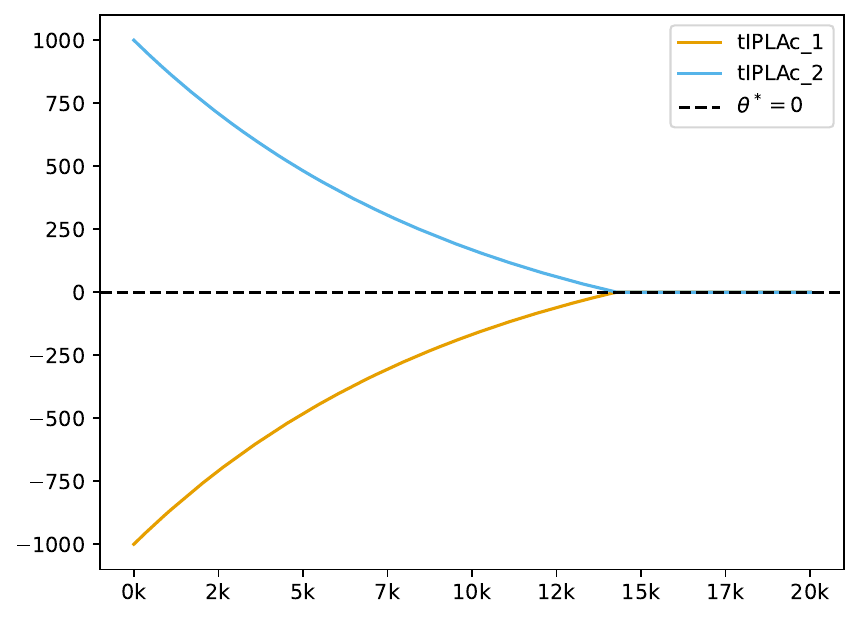}
    \end{minipage}
    \vfill
    \centering
    \begin{minipage}{0.5\textwidth}
        \centering
        \includegraphics[width=\textwidth,
  height=0.225\textheight,
  keepaspectratio=false]{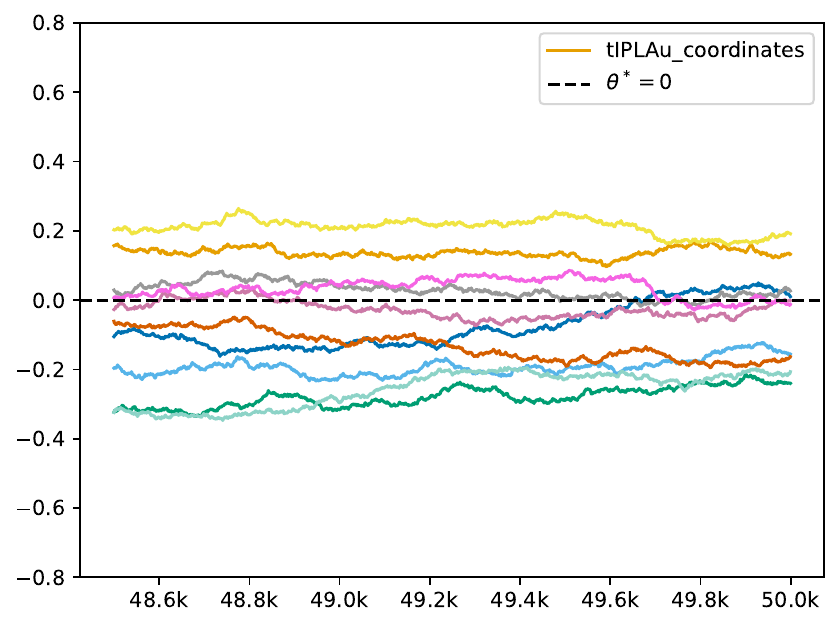}
    \end{minipage}\hfill
    \begin{minipage}{0.5\textwidth}
        \centering
        \includegraphics[width=\textwidth,
  height=0.225\textheight,
  keepaspectratio=false]{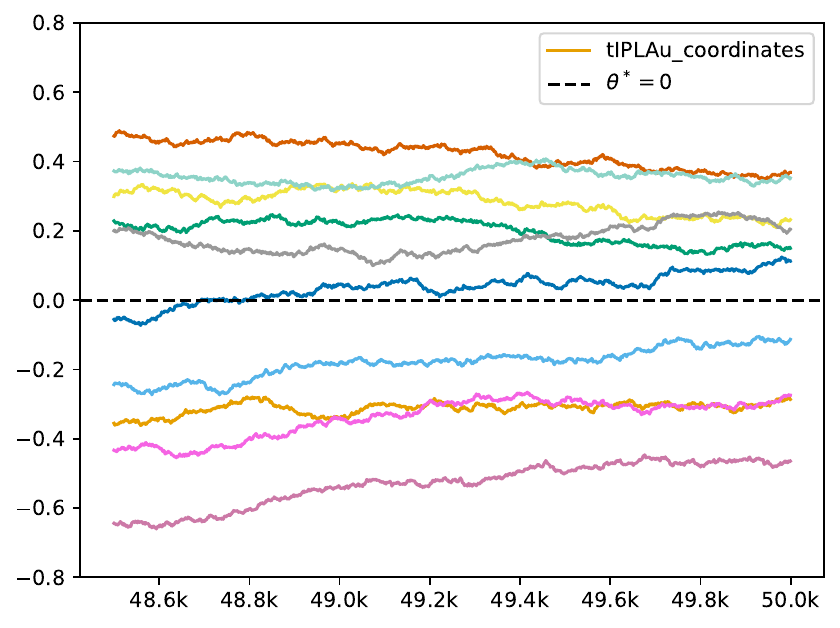}
    \end{minipage}
    \vfill
    \centering
    \begin{minipage}{0.5\textwidth}
        \centering
        \includegraphics[width=\textwidth,
  height=0.225\textheight,
  keepaspectratio=false]{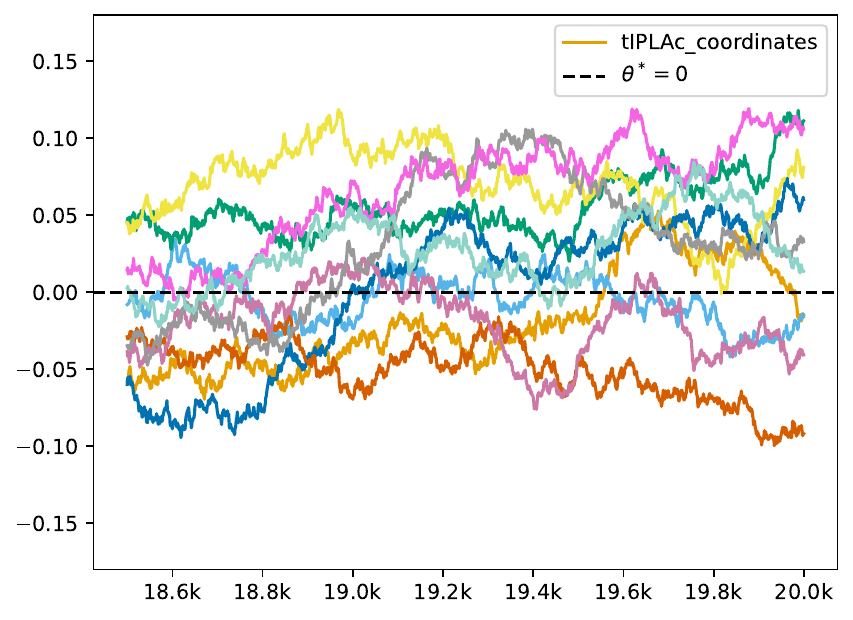}
    \end{minipage}\hfill
    \begin{minipage}{0.5\textwidth}
        \centering
        \includegraphics[width=\textwidth,
  height=0.225\textheight,
  keepaspectratio=false]{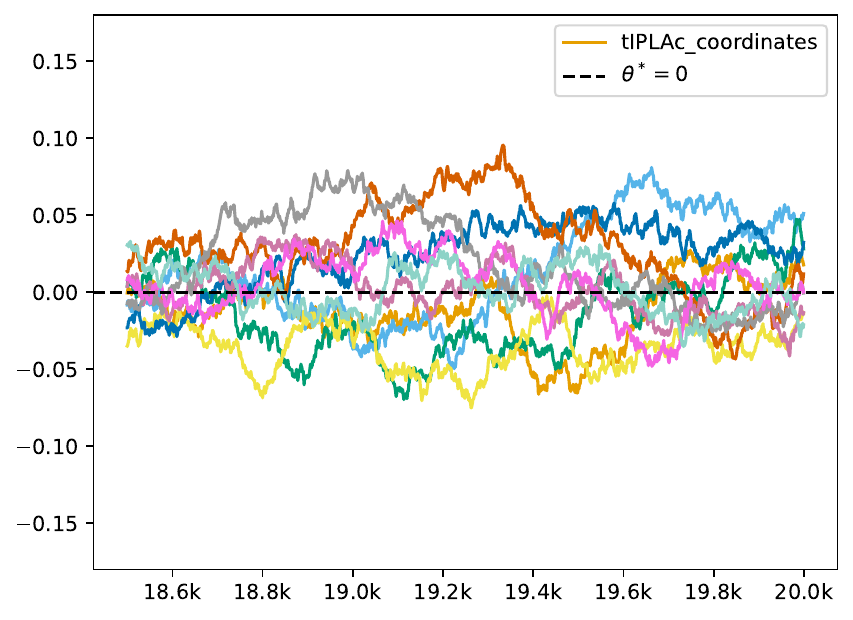}
    \end{minipage}
    \caption{Performance comparison between \hyperref[algo1]{tIPLAu} and \hyperref[algo2]{tIPLAc} on the superlinear toy problem. The left column corresponds to $d^x=d^{\theta}=100$, and the right column to $d^x=d^{\theta}=1000$. The first two rows display the full iterates of the algorithm, while the last two rows focus on the last 1500 iterations to examine any residual bias in the limiting behavior}
    \label{fig6}
\end{figure}
\subsection{Mixed term - Toy example}
To demonstrate the applicability of \hyperref[algo2]{tIPLAc} under assumption \hyperlink{Assum3b}{A3-ii}, we consider the example from Section \hyperref[toy]{4.2} for $m=1$ but with an additional cross term $x\theta$,
$$k(\theta)=\int_{\mathbb{R}^{d^x}}\exp{\left(-x\theta-|x|^4-(|x|^2+1)(|\theta|^2+1)-|\theta|^4\right)}dx.$$
Hence the potential $U(\theta,x)=x\theta+|x|^4+(|x|^2+1)(|\theta|^2+1)+|\theta|^4$ has the following gradient structure
\begin{align*}
    \nabla_x U=4|x|^2x+2(|\theta|^2+1)x+\theta \text{ and } \nabla_{\theta} U=4|\theta|^2\theta+2(|x|^2+1)\theta +x.
\end{align*}
One immediately checks that \hyperlink{Assum1}{A1} holds, and in view of Remark \ref{remark7}, \hyperlink{Assum2}{A2} is obtained as well, that is
\begin{gather*}
    |\nabla U(v)-\nabla U(v')|\leq52\left(1+|v|^{3}+|v'|^{3}\right)|v-v'|,\\
    \left\langle v-v',\nabla U(v)-\nabla U(v')\right\rangle\geq |v-v'|^2,
\end{gather*}
where $v=(\theta,x)$. However the gradient of the potential $U(\theta,x)$ fails to satisfy \hyperlink{Assum3a}{A3-i}, as shown below \begin{figure}[ht]
    \centering
    \begin{minipage}{0.5\textwidth}
        \centering
        \includegraphics[width=\textwidth,
  height=0.25\textheight,
  keepaspectratio=false]{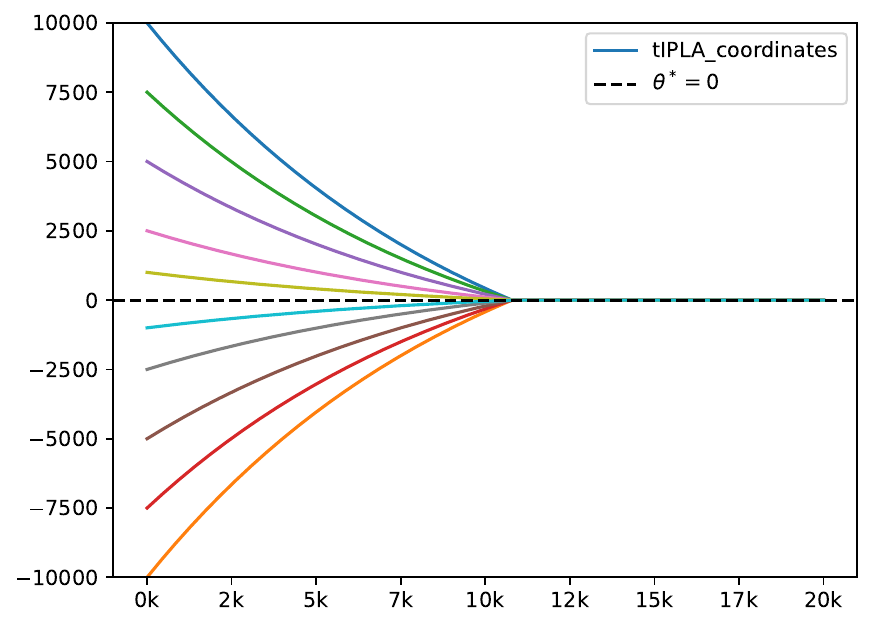}
    \end{minipage}\hfill
    \begin{minipage}{0.5\textwidth}
        \centering
        \includegraphics[width=\textwidth,
  height=0.25\textheight,
  keepaspectratio=false]{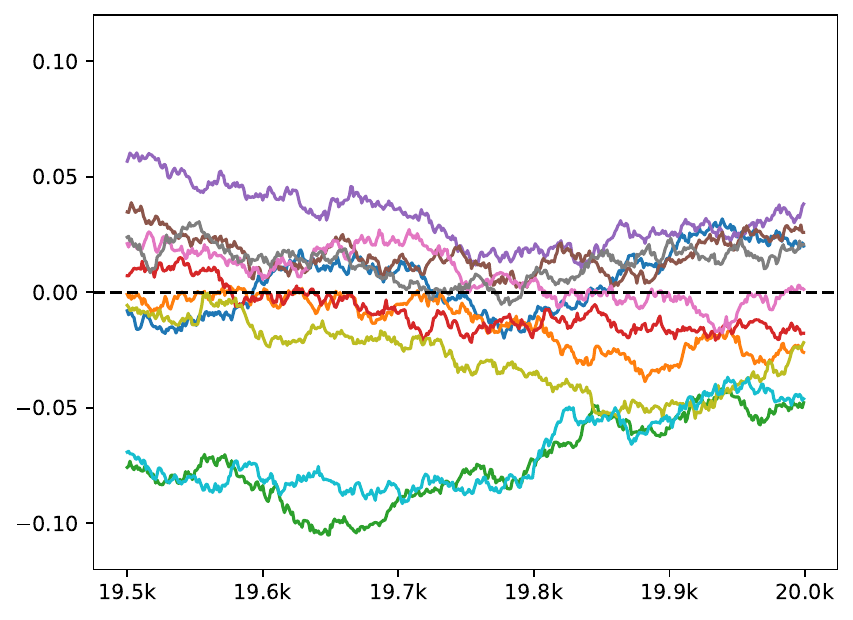}
    \end{minipage}
    \vfill
    \centering
    \begin{minipage}{0.5\textwidth}
        \centering
        \includegraphics[width=\textwidth,
  height=0.25\textheight,
  keepaspectratio=false]{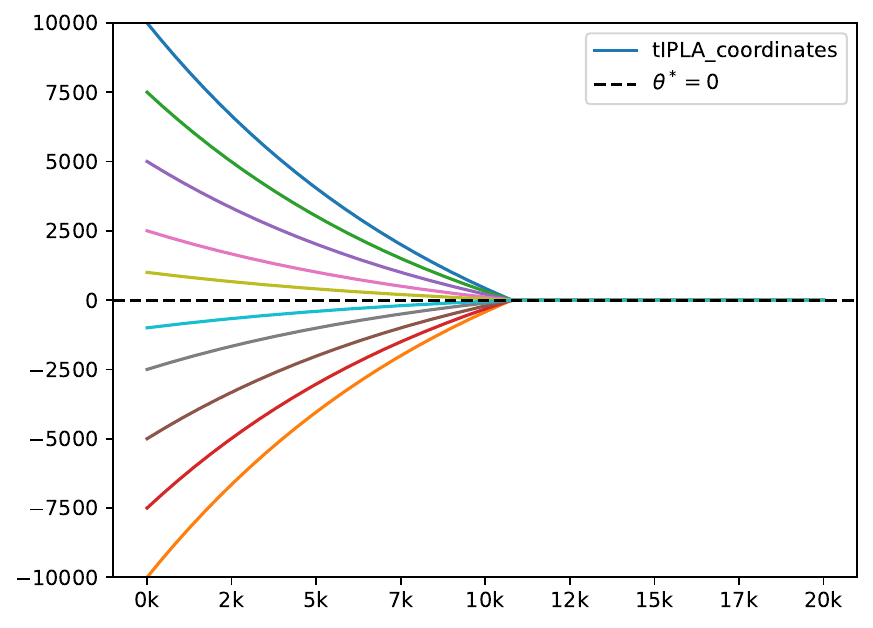}
    \end{minipage}\hfill
    \begin{minipage}{0.5\textwidth}
        \centering
        \includegraphics[width=\textwidth,
  height=0.25\textheight,
  keepaspectratio=false]{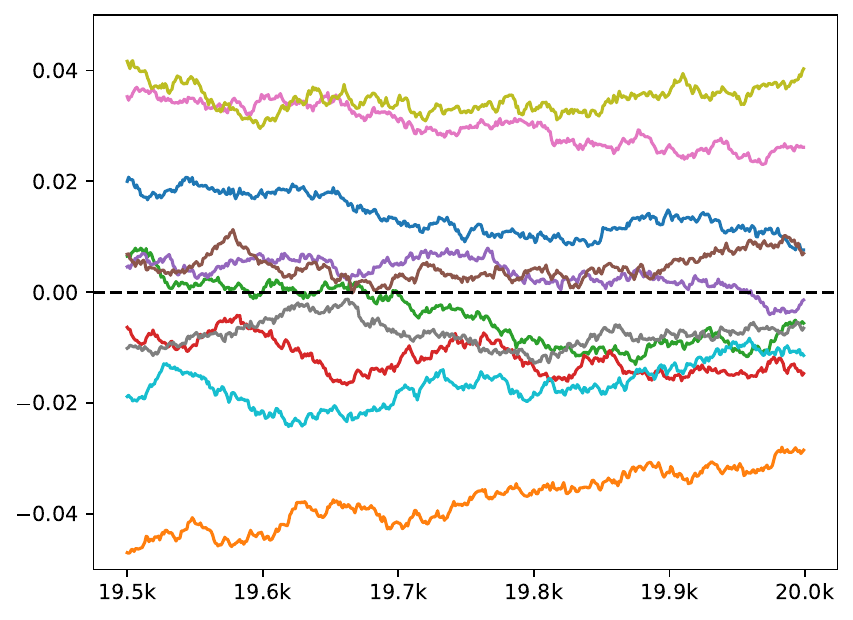}
    \end{minipage}
    \vfill
    \centering
    \begin{minipage}{0.5\textwidth}
        \centering
        \includegraphics[width=\textwidth,
  height=0.25\textheight,
  keepaspectratio=false]{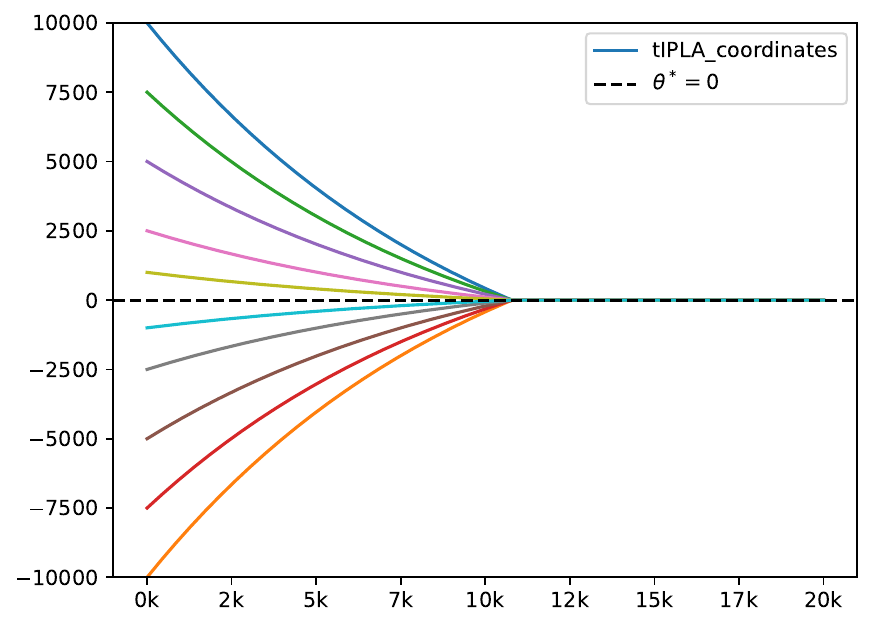}
    \end{minipage}\hfill
    \begin{minipage}{0.5\textwidth}
        \centering
        \includegraphics[width=\textwidth,
  height=0.25\textheight,
  keepaspectratio=false]{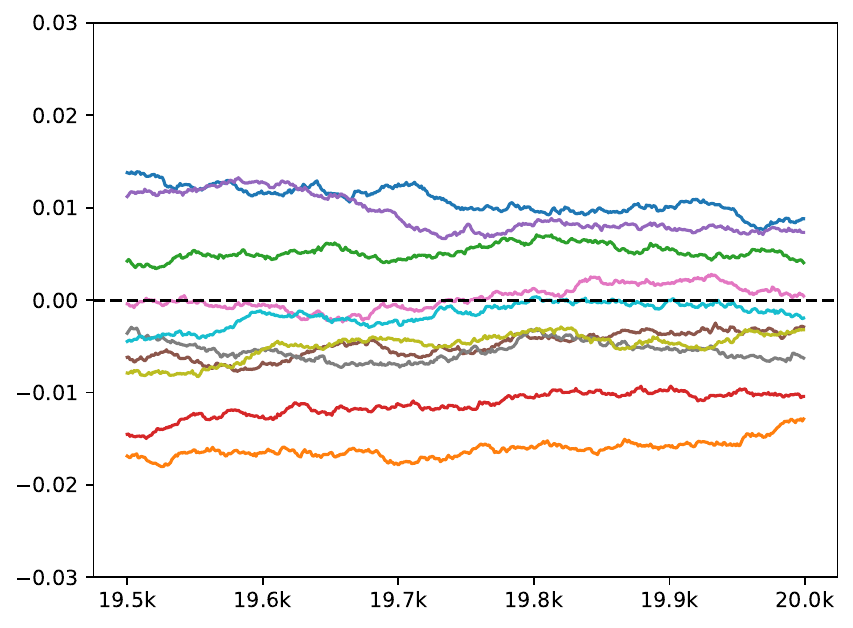}
    \end{minipage}
    \caption{The performance of \hyperref[algo2]{tIPLAc} on the superlinear toy problem under assumption A3-ii. The top row corresponds to $N=100$, the middle row to $N=1000$, and the bottom row to $N=10,000$. The left column displays the full iterates of the algorithm, while the right column focuses on the last 500 iterations to examine any residual bias in the limiting behavior}
    \label{fig7}
\end{figure}
\begin{align*}
    \left(\nabla_x U\right)_i x_i=4|x|^2x^2_i+2(|\theta|^2+1)x^2_i+x_i\theta_i\geq 2x^2_i+x_i\theta_i,
\end{align*}
since the product $x_i\theta_i$ cannot be controlled due to the fact that $\theta$ takes values in an unbounded domain. Nevertheless, we show that \hyperlink{Assum3b}{A3-ii} holds with parameters $\mu=3,\ \rho=1/2 \text{ and } b=0$. Indeed
\begin{align*}
    \left(\nabla_x U\right)_i x_i\geq 2x^2_i-\dfrac{x_i^2}{2}-\dfrac{\theta_i^2}{2}=\dfrac{3}{2}x_i^2-\dfrac{1}{2}\theta_i^2.
\end{align*}
\noindent \textbf{Experiment details}.
We choose $d^x=d^{\theta} = 10$, with the true parameter $\theta^*$ being unknown. We run \hyperref[algo2]{tIPLAc}, for $M = 20\text{k}$ iterations, using $N = 100, 1000, 10000$ particles and a stepsize of $\lambda = 0.0001$. The parameter ${\theta}_0$ is initialized from a deterministic value $[10,000, -10,000, 7500,-7500,5000,-5000,2500,-2500,1000,-1000]$, while the particles $X^{i,N}_0$ are drawn from a Gaussian distribution with randomized mean over the interval $(-100,100)$ for each coordinate, and covariance matrix $R=10I_{d^x}$.\\
\indent As shown in Fig. \ref{fig7}, we observe that as $N$ increases, the convergence time to the true value remains unchanged, which agrees with our observation in Fig. \ref{fig3}. Furthermore, focusing on the second column, which displays the last 500 iterations of each simulation, we notice again that the residual bias of the algorithm, once convergence has been achieved, decreases significantly as the number of particles increases.
\clearpage
\section{References} \vspace{-2.0em} \renewcommand\refname{} 
\appendix
\section{Proofs of Section 3}
\subsection{Uniform moments bounds}
\begin{lemma}\label{lemma0} Consider either the dynamics given by \eqref{eq:00}-\eqref{eq:01} or \eqref{eq:02}-\eqref{eq:03} and let Assumptions \hyperlink{Assum1}{A1}-\hyperlink{Assum2}{A2} hold, then there exists a constant $C>0$ such that $$\sup_{t\geq 0}\mathbb{E}\left[|\mathcal{Z}_{t}^{\cdot,N}|^2\right]\leq C.$$\end{lemma}
\begin{proof}
    Let us first consider the rescaled dynamics as given by \eqref{eq:001} for the original system of equations \eqref{eq:00}-\eqref{eq:01}. The square norm of these dynamics is expressed as
\begin{align*}
    |\mathcal{Z}_{t}^{c,N}|^2=|\vartheta_{t}^{c,N}|^2+\dfrac{1}{N}\sum_{i=1}^{N}|\mathcal{X}_{t}^{i,c,N}|^2.
\end{align*}
Using standard arguments involving stopping times, Gr\"{o}nwall's lemma and Fatou's lemma, we obtain the existence of a constant $c$, which depends on time, such that $\sup_{0\leq t\leq T}\mathbb{E}\left[|\mathcal{Z}_{t}^{c,N}|^2\right]\leq c$ for any $T>0$. Furthermore, by applying It\^{o}'s formula, we derive
\begin{align*}
    |\mathcal{Z}_{t}^{c,N}|^2&\leq |\vartheta_{0}^{c,N}|^2-2\int_0^t\langle \vartheta_{s}^{\cdot,N},\dfrac{1}{N}\sum_{i=1}^N h^{\theta}(\mathcal{V}_{s}^{i,c,N})\rangle ds+\dfrac{2d^{\theta}t}{N}+2\sqrt{\dfrac{2}{N}}\int_0^t \vartheta_{s}^{c,N}dB_{s}^{0,N}+2d^x t\\
    &+\dfrac{1}{N}\sum_{i=1}^N|\mathcal{X}_{0}^{i,c,N}|^2-\dfrac{2}{N}\sum_{i=1}^N\int_0^t\langle \mathcal{X}_s^{i,c,N},h^{x}(\mathcal{V}_{s}^{i,c,N})\rangle ds+\dfrac{2\sqrt{2}}{N}\sum_{i=1}^N\int_0^t \mathcal{X}_{s}^{i,c,N}dB_{s}^{i,N}ds\\
    &\leq |\mathcal{Z}_{0}^{c,N}|^2-2\int_0^t\dfrac{1}{N}\sum_{i=1}^N\langle \mathcal{V}_{s}^{i,c,N},h(\mathcal{V}_{s}^{i,c,N})\rangle ds+2(d^{\theta}/N+d^x)t\\
    &+2\sqrt{\dfrac{2}{N}}\int_0^t \vartheta_{s}^{c,N}dB_{s}^{0,N}+\dfrac{2\sqrt{2}}{N}\sum_{i=1}^N\int_0^t \mathcal{X}_{s}^{i,c,N}dB_{s}^{i,N}ds\\
    &\leq |\mathcal{Z}_{0}^{c,N}|^2-\mu\int_0^t |\mathcal{Z}_{s}^{c,N}|^2 ds+2bt+2(d^{\theta}/N+d^x)t\\
    &+2\sqrt{\dfrac{2}{N}}\int_0^t \vartheta_{s}^{c,N}dB_{s}^{0,N}+\dfrac{2\sqrt{2}}{N}\sum_{i=1}^N\int_0^t \mathcal{X}_{s}^{i,c,N}dB_{s}^{i,N}ds.
\end{align*}Taking the expectation on both sides, we obtain
\begin{gather*}
\mathbb{E}\left[|\mathcal{Z}_{t}^{c,N}|^2\right]\leq \mathbb{E}\left[|\mathcal{Z}_{0}^{c,N}|^2\right]-\mu\int_0^t \mathbb{E}\left[|\mathcal{Z}_{s}^{c,N}|^2\right]ds+2(b+d^{\theta}/N+d^x)t.\\
    \dfrac{d}{dt}\exp{(\mu t)}\mathbb{E}\left[|\mathcal{Z}_{t}^{c,N}|^2\right]\leq \exp{(\mu t)}C,
\end{gather*}
which, upon integration, leads to
\begin{align*}
    \mathbb{E}\left[|\mathcal{Z}_{t}^{c,N}|^2\right]\leq C/\mu\Rightarrow\sup_{t\geq 0}\mathbb{E}\left[|\mathcal{Z}_{t}^{c,N}|^2\right]\leq C,
\end{align*}
as C is a constant independent of time. The corresponding result regarding equations \eqref{eq:02}-\eqref{eq:03} follows by going through the same steps, with the calculations differing only up to a constant in the SDEs coefficients.
\end{proof}
\subsection{Uniformly tamed scheme }\label{appendix1}
\subsubsection{Key quantities for the proof of the main Lemmas.}
The following definitions refer to the rescaled dynamics \eqref{eq:002}-\eqref{eq:003} of the algorithm \hyperref[algo1]{tIPLAu} and its continuous time interpolations \eqref{eq:004}-\eqref{eq:005} \begin{align} Z^{\lambda,u}_{n}=\left(\theta_{n+1}^{\lambda,u},N^{-1/2}X_{n+1}^{1,\lambda,u},\ldots,N^{-1/2}X_{n+1}^{N,\lambda,u}\right),\label{rescale1}\end{align}
\begin{align} \overline{Z}^{\lambda,u}_t=\left(\overline{\theta}^{\lambda,u}_t,N^{-1/2}\overline{X}^{1,\lambda,u}_t,\ldots,N^{-1/2}\overline{X}^{N,\lambda,u}_t\right),\label{rescale31}\end{align}
\begin{align} \mathcal{Z}_{\lambda t}^{u,N}=\left(\vartheta_{\lambda t}^{u,N},N^{-1/2}\mathcal{X}_{\lambda t}^{1,u,N},\ldots,N^{-1/2}\mathcal{X}_{\lambda t}^{N,u,N}\right).\label{rescale21} \end{align}
\subsubsection{Moment and increment bounds}
\begin{lemma}\hypertarget{lemma1}{}\label{lemma1} Let \hyperlink{Assum1}{A1}, \hyperlink{Assum2}{A2} and \hyperlink{Assum4}{A4} hold. Then, for any $0\leq\lambda<N^p/(4\mu)$, it holds that
$$\mathbb{E}\left[|Z^{\lambda,u}_{n}|^2\right]\leq C_{|z_0|,\mu,b}(1+d^{\theta}/N+d^x),$$
for a constant $C>0$, independent of $N,n,\lambda, d^x \text{and } d^{\theta}$, given explicitly within the proof.\end{lemma}
\begin{proof} Consider the rescaled iterates as described in equation \eqref{rescale1}, which expands as
    \begin{gather*}
\left|Z_{n+1}^{\lambda,u}\right|^2=\left|\theta_{n+1}^{\lambda,u}\right|^2+\dfrac{1}{N}\sum_{i=1}^N \left|X_{n+1}^{i,\lambda,u}\right|^2\\ 
=\left|\theta_{n}^{\lambda,u}-\dfrac{\lambda}{N^{p+1}}\sum_{i=1}^N h^{\theta}_{\lambda,u}(\theta_{n}^{\lambda,u},X_{n}^{i,\lambda,u})\right|^2+\dfrac{2\lambda}{N^{p+1}}|\xi^{(0)}_{n+1}|^2\\
+2\sqrt{\dfrac{2\lambda}{N^{p+1}}}\left\langle \theta_{n}^{\lambda,u}-\dfrac{\lambda}{N^{p+1}}\sum_{i=1}^N h^{\theta}_{\lambda,u}(\theta_{n}^{\lambda,u},X_{n}^{i,\lambda,u}),\xi_{n+1}^{(0)}\right\rangle\\
+\dfrac{1}{N}\sum_{i=1}^N\left(\left|X_{n}^{i,\lambda,u}-\dfrac{\lambda}{N^p} h^{x}_{\lambda,u}(\theta_{n}^{\lambda,u},X_{n}^{i,\lambda,u})\right|^2+\dfrac{2\lambda}{N^p}|\xi^{(i)}_{n+1}|^2\right.\\
+\left.2\sqrt{\dfrac{2\lambda}{N^p}}\left\langle X_{n}^{i,\lambda,u}-\dfrac{\lambda}{N^p} h^{x}_{\lambda,u}(\theta_{n}^{\lambda,u},X_{n}^{i,\lambda,u}),\xi_{n+1}^{(i)}\right\rangle\right).
\end{gather*}
Taking the conditional expectation of both sides with respect to the filtration generated by $Z_{n}^{\lambda,u}$ the cross terms vanish to $0$ due to the independence between the $\xi^{(i)}_{n+1}$'s and $Z_{n}^{\lambda,u}$. This yields
\begin{align*}
\mathbb{E}\left[\left|Z_{n+1}^{\lambda,u}\right|^2|Z_{n}^{\lambda,u}\right]&=\mathbb{E}\left[\left|\theta_{n}^{\lambda,u}\right|^2|Z_{n}^{\lambda,u}\right]-\dfrac{2\lambda}{N^{p+1}}\sum_{i=1}^N\mathbb{E}\left[\langle \theta_n^{\lambda,u},h_{\lambda,u}^{\theta}(\theta_{n}^{\lambda,u},X_{n}^{i,\lambda,u})\rangle |Z_{n}^{\lambda,u}\right]\\&+\dfrac{\lambda^2}{N^{2(p+1)}}\mathbb{E}\left[\left|\sum_{i=1}^Nh_{\lambda,u}^{\theta}(\theta_{n}^{\lambda,u},X_{n}^{i,\lambda,u})\right|^2|Z_{n}^{\lambda,u}\right]+\dfrac{2\lambda d^{\theta}}{N^{p+1}}\\
&+\dfrac{1}{N}\sum_{i=1}^N\mathbb{E}\left[\left|X_{n}^{i,\lambda,u}\right|^2|Z_{n}^{\lambda,u}\right]\nonumber\\
&-\dfrac{2\lambda}{N^{p+1}}\sum_{i=1}^N\mathbb{E}\left[\langle X_n^{i,\lambda},h_{\lambda,u}^{x}(\theta_{n}^{\lambda,u},X_{n}^{i,\lambda,u})\rangle|Z_{n}^{\lambda,u}\right]\\
&+\dfrac{\lambda^2}{N^{2p+1}}\sum_{i=1}^N\mathbb{E}\left[|h_{\lambda,u}^x(\theta_{n}^{\lambda,u},X_{n}^{i,\lambda,u})|^2|Z_{n}^{\lambda,u}\right]+\dfrac{2\lambda d^x}{N^p}.
\end{align*}
Next, using the elementary inequality $(t_1+\ldots+t_m)^p\leq m^{p-1}(t_1^p+\ldots,t_m^p)$ along with the fact that all of the terms inside the conditional expectations are measurable, we simplify the expression to
\begin{align*}
\mathbb{E}\left[\left|Z_{n+1}^{\lambda,u}\right|^2|Z_{n}^{\lambda,u}\right]&\leq |\theta_{n}^{\lambda,u}|+\dfrac{1}{N}\sum_{i=1}^N|X_n^{i,\lambda,u}|^2+\dfrac{2\lambda d^{\theta}}{N^{p+1}}++\dfrac{2\lambda d^x}{N^p}\\&-\dfrac{2\lambda}{N^{p+1}}\sum_{i=1}^N\langle \theta_n^{\lambda},h_{\lambda,u}^{\theta}(\theta_{n}^{\lambda,u},X_{n}^{i,\lambda,u})\rangle\\
&-\dfrac{2\lambda}{N^{p+1}}\sum_{i=1}^N\langle X_n^{i,\lambda,u},h_{\lambda,u}^{x}(\theta_{n}^{\lambda,u},X_{n}^{i,\lambda,u}) \rangle \\&+\dfrac{\lambda^2}{N^{2p+1}}\sum_{i=1}^N|h_\lambda^{\theta}(\theta_{n}^{\lambda,u},X_{n}^{i,\lambda,u})|^2 +\dfrac{\lambda^2}{N^{2p+1}}\sum_{i=1}^N|h_{\lambda,u}^{x}(\theta_{n}^{\lambda,u},X_{n}^{i,\lambda,u})|^2.
\end{align*}
Recalling that the two drift terms can be written together as $(h^{\theta}_{\lambda,u},h^x_{\lambda,u})=h_{\lambda,u}$, leads to
\begin{align*}
\mathbb{E}\left[\left|Z_{n+1}^{\lambda,u}\right|^2|Z_{n}^{\lambda,u}\right]&\leq \left|Z_{n}^{\lambda,u}\right|^2-\dfrac{2\lambda}{N^{p+1}}\sum_{i=1}^N\langle V_n^{i,\lambda,u},h_{\lambda,u}(V_{n}^{i,\lambda,u})\rangle\\
&+\dfrac{\lambda^2}{N^{2p+1}}\sum_{i=1}^N|h_{\lambda,u}(V_{n}^{i,\lambda,u})|^2+\dfrac{2\lambda}{N^{p}}(d^{\theta}/N+d^x).
\end{align*}
Now, by applying Properties \hyperlink{Property1}{1} and \hyperlink{Property3}{3} of the taming function, we simplify further
\begin{align*}
\mathbb{E}\left[\left|Z_{n+1}^{\lambda,u}\right|^2|Z_{n}^{\lambda,u}\right]&\leq \left|Z_{n}^{\lambda,u}\right|^2-\dfrac{\mu\lambda}{N^{p+1}}\sum_{i=1}^N\left|V_{n}^{i,\lambda,u}\right|^2+\dfrac{2\lambda b}{N^{p}}\\&+\dfrac{2\lambda^2\mu^2}{N^{2p+1}}\sum_{i=1}^N \left|V_{n}^{i,\lambda,u}\right|^2+\dfrac{2\lambda^2\lambda^{-1}N^{p}}{N^{2p}}+\dfrac{2\lambda}{N^{p}}(d^{\theta}/N+d^x)\\
&\leq\left(1-\dfrac{\lambda\mu}{N^p}+\dfrac{2\lambda^2\mu^2}{N^{2p}}\right)\left|Z_{n}^{\lambda,u}\right|^2+\dfrac{2\lambda}{N^p}\left(b+1+d^{\theta}/N+d^x\right).
\end{align*}
Iterating this inequality, and assuming the restriction $\lambda<N^p/(4\mu)$, yields the desired bound
\begin{align*}
\mathbb{E}\left[\left|Z_{n+1}^{\lambda,u}\right|^2\right]&\leq \left(1-\dfrac{\mu\lambda}{2N^p}\right)^n\mathbb{E}\left[\left|Z_{0}^{\lambda,u}\right|^2\right]+\dfrac{1-(1-\mu\lambda/2)^n}{\mu\lambda/2N^p}\dfrac{2\lambda}{N^p}(b+1+d^{\theta}/N+d^x)\\ &\leq \mathbb{E}\left[\left|Z_{0}^{\lambda,u}\right|^2\right]+\dfrac{4}{\mu}(b+1+d^{\theta}/N+d^x)\\
&\leq \left(\mathbb{E}\left[\left|Z_{0}^{\lambda,u}\right|^2\right]+\dfrac{4}{\mu}(b+1)\right)(1+d^{\theta}/N+d^x)\\
&\leq C_{|z_0|,\mu,b}(1+d^{\theta}/N+d^x).
\end{align*}\end{proof}
\noindent \begin{lemma}\hypertarget{lemma2}{}\label{lemma2} Let \hyperlink{Assum1}{A1}, \hyperlink{Assum2}{A2} and \hyperlink{Assum4}{A4} hold. Then, for any $0\leq\lambda<N^p/(4\mu)$ and $q\in[2,2\ell+1)\cap\mathbb{N}$, it holds that,
$$\mathbb{E}\left[|Z^{\lambda,u}_{n}|^{2q}\right]\leq C_{|z_0|,b,q,\mu}(1+d^{\theta}/N+d^x)^q,$$
for a constant $C>0$, independent of $N,n,\lambda, d^x \text{and } d^{\theta}$, given explicitly within the proof.\end{lemma}
\begin{proof}
To make the following calculations clearer, we define the auxiliary processes
\begin{gather*}
\Delta_n^{\lambda,\theta}=\theta_{n}^{\lambda,u}-\dfrac{\lambda}{N^{p+1}}\sum_{i=1}^N h^{\theta}_{\lambda,u}\left(\theta_n^{\lambda},X^{i,\lambda}_n\right), \ G_n^{\lambda,\theta}=\sqrt{\dfrac{2\lambda}{N^{p+1}}}\xi^{(0)}_{n+1},\\
\Delta_n^{\lambda,x,i}=X_{n}^{i,\lambda,u}-\dfrac{\lambda}{N^p} h^{x}_{\lambda,u}\left(\theta_n^{\lambda},X^{i,\lambda}_n\right), \ G_n^{\lambda,x,i}=\sqrt{\dfrac{2\lambda}{N^p}}\xi^{(i)}_{n+1}.
\end{gather*}
For the sake of simplicity in the calculations, we denote $\lambda/N^p$ as $\lambda$ from this point onwards. Recall from Lemma \hyperlink{lemma1}{2} that we have the bound
\begin{gather*}
A_n^{\lambda}:=|\Delta_n^{\lambda,\theta}|^2+\dfrac{1}{N}\sum_{i=1}^N|\Delta_n^{\lambda,i,x}|^2\leq \left(1-\dfrac{\lambda\mu}{2}\right)|Z_n^{\lambda,u}|^2+2\lambda C.
\end{gather*}
Now, we define the auxiliary quantity
\begin{gather*}
B^{\lambda}_n=2\langle \Delta_n^{\theta},G_n^{\lambda,\theta}\rangle+\dfrac{2}{N}\sum_{i=1}^N\langle \Delta_n^{\lambda,x,i},G_n^{\lambda,x,i}\rangle+|G_n^{\lambda,\theta}|^2+\dfrac{1}{N}\sum_{i=1}^N|G_n^{\lambda,x,i}|^2.
\end{gather*}
For the $2q$-th moment, we can write
\begin{align}
|Z_{n+1}^{\lambda,u}|^{2q}&=\left(A_n^{\lambda}+B^{\lambda}_n\right)^q\nonumber \\
&\leq (A_n^{\lambda})^q+2q(A_n^{\lambda})^{q-1}B^{\lambda}_n+\sum_{k=2}^q\binom{q}{k}|A_n^{\lambda}|^{q-k}|B^{\lambda}_n|^k. \label{eq:010}
\end{align}
We will now handle each term separately
\begin{align}
\mathbb{E}\left[(A_n^{\lambda})^q|Z_n^{\lambda,u}\right]&=(A_n^{\lambda})^{q}\leq \left(\left(1-\dfrac{\lambda\mu}{2}\right)|Z_n^{\lambda,u}|^2+2\lambda C\right)^q\nonumber\\
&\leq \left(1+\dfrac{\lambda \mu}{4}\right)^{q-1}\left(1-\dfrac{\lambda\mu}{2}\right)^q|Z_n^{\lambda,u}|^{2q}+\left(1+\dfrac{4}{\lambda\mu }\right)^{q-1}2^q{\lambda}^q C^q\nonumber\\
&\leq \left(1-\dfrac{\lambda\mu}{4}\right)^{q-1}\left(1-\dfrac{\lambda\mu}{2}\right)|Z_n^{\lambda,u}|^{2q}+\left(\lambda+\dfrac{4}{\mu }\right)^{q-1}\lambda (2C)^q\nonumber\\
&\leq r_q^{\lambda}|Z_n^{\lambda,u}|^{2q}+w_q^{\lambda},\label{eq:A1}
\end{align}
where $r_q^{\lambda}=(1-\lambda\mu/4)^{q-1}(1-\lambda\mu/2)$ and $w_q^{\lambda}=(\lambda+4/\mu)^{q-1}\lambda (2C)^q$. Notice that we used the elementary equation $(r+s)^q\leq(1+\epsilon)^{q-1}r^q+(1+1/\epsilon)^{q-1}s^q$ with $\epsilon=\lambda\mu/4$. Additionally for simplicity, we denote $d=d^{\theta}/N+d^x$. Now, similarly
\begin{align}
\mathbb{E}\left[2q(A_n^{\lambda})^{q-1}B^{\lambda}_n|Z_n^{\lambda,u}\right]&=2q(A_n^{\lambda})^{q-1}\mathbb{E}\left[B^{\lambda}_n|Z_n^{\lambda,u}\right]=4q\lambda\left(d^{\theta}/N+d^x\right)(A_n^{\lambda})^{q-1}\nonumber\\
&\leq 4q\lambda\left(d^{\theta}/N+d^x\right) \left(r_{q-1}^{\lambda}|Z_n^{\lambda,u}|^{2(q-1)}+w_{q-1}^{\lambda}\right)\nonumber\\
&\leq  4q\lambda d \left(r_{q-1}^{\lambda}|Z_n^{\lambda,u}|^{2(q-1)}+w_{q-1}^{\lambda}\right).\label{eq:A2}
\end{align}
The third term in equation \eqref{eq:010} can be expanded to
\begin{align}
\sum_{k=2}^q\binom{q}{k}|A_n^{\lambda}|^{q-k}|B^{\lambda}_n|^k&=\sum_{m=0}^{q-2}\binom{q}{m+2}|A_n^{\lambda}|^{q-2-m}|B^{\lambda}_n|^{m+2}\nonumber\\
&=\dfrac{q}{m+2}\dfrac{q-1}{m+1}\sum_{m=0}^{q-2}\binom{q-2}{m}|A_n^{\lambda}|^{q-2-m}|B^{\lambda}_n|^{m}|B^{\lambda}_n|^{2}\nonumber\\
&\leq q(q-1)\left(|A_n^{\lambda}|+|B_n^{\lambda}|\right)^{q-2}|B^{\lambda}_n|^2\nonumber\\
&\leq q(q-1)2^{q-3}|A_n^{\lambda}|^{q-2}|B_n^{\lambda}|^2+q(q-1)2^{q-3}|B_n^{\lambda}|^q\nonumber\\
&:=D+F.\label{eq:A3}
\end{align}
Taking the expectation in equation \eqref{eq:A3} yields
\begin{align*}
\mathbb{E}[D|Z_n^{\lambda,u}]&=q(q-1)2^{q-3}|A_n^{\lambda}|^{q-2}\mathbb{E}\left[|B_n^{\lambda}|^2|Z_n^{\lambda,u}\right],
\end{align*}
where
\begin{align*}
&\mathbb{E}\left[|B_n^{\lambda}|^2|Z_n^{\lambda,u}\right]\\&=\mathbb{E}\left[{\left|2\langle{ \Delta_n^{\theta},G_n^{\lambda,\theta}}\rangle+\dfrac{2}{N}\sum_{i=1}^N\langle{ \Delta_n^{\lambda,x,i},G_n^{\lambda,x,i}}\rangle+{\left|G_n^{\lambda,\theta}\right|}^2+\dfrac{1}{N}\sum_{i=1}^N{\left|G_n^{\lambda,x,i}\right|}^2\right|}^2|Z_n^{\lambda,u}\right]\\
&\leq 4\left(\mathbb{E}\left[4|\Delta_n^{\lambda,\theta}|^2|G_n^{\lambda,\theta}|^2|Z_n^{\lambda,u}\right]+\mathbb{E}\left[\dfrac{4}{N}\sum_{i=1}^N|\Delta_n^{\lambda,x,i}|^2|G_n^{\lambda,x,i}|^2|Z_n^{\lambda,u}\right]\right.\\
&+\left.\mathbb{E}\left[|G_n^{\lambda,x,i}|^4|Z_n^{\lambda,u}\right]+\mathbb{E}\left[\dfrac{1}{N}\sum_{i=1}^N|G_n^{\lambda,x,i}|^4|Z_n^{\lambda,u}\right]\right)\\
&\leq 4\left(4|\Delta_n^{\lambda,\theta}|^2\mathbb{E}\left[|G_n^{\lambda,\theta}|^2|Z_n^{\lambda,u}\right]+\dfrac{4}{N}\sum_{i=1}^N|\Delta_n^{\lambda,x,i}|^2\mathbb{E}\left[|G_n^{\lambda,x,i}|^2|Z_n^{\lambda,u}\right]\right.\\
&+\left.\mathbb{E}\left[|G_n^{\lambda,x,i}|^4|Z_n^{\lambda,u}\right]+\mathbb{E}\left[\dfrac{1}{N}\sum_{i=1}^N|G_n^{\lambda,x,i}|^4|Z_n^{\lambda,u}\right]\right).
\end{align*}
Recalling that each $G_n^{\lambda,\cdot,i}$ follows a Gaussian distribution, we subsequently derive
\begin{align*}
&\mathbb{E}\left[|B_n^{\lambda}|^2|Z_n^{\lambda,u}\right]\\&\leq 4\left(4(2\lambda d^{\theta}/N)|\Delta_n^{\lambda,\theta}|^2+\dfrac{4}{N}\sum_{i=1}^N(2\lambda d^x)|\Delta_n^{\lambda,x,i}|^2+3(2\lambda d^{\theta}/N)^2+3(2\lambda d^x)^2\right)\\
&\leq 16(2\lambda d)\left(|\Delta_n^{\lambda,\theta}|^2+\dfrac{1}{N}\sum_{i=1}^N|\Delta_n^{\lambda,x,i}|^2\right)+24(2\lambda d)^2\\
&\leq  16(2\lambda d)A_n^{\lambda}+24(2\lambda d)^2.
\end{align*}
Substituting this into the term for $D$, we get
\begin{align}
\mathbb{E}[D|Z_n^{\lambda,u}]&=q(q-1)2^{q-3}|A_n^{\lambda}|^{q-2}\left(16(2\lambda d)A_n^{\lambda}+24(2\lambda d)^2\right)\nonumber\\
&=16(2\lambda d)q(q-1)2^{q-3} \left(r_{q-1}^{\lambda}|Z_n^{\lambda,u}|^{2(q-1)}+w_{q-1}^{\lambda}\right)\nonumber\\&+24(2\lambda d)^2q(q-1)2^{q-3} \left(r_{q-2}^{\lambda}|Z_n^{\lambda,u}|^{2(q-2)}+w_{q-2}^{\lambda}\right).\label{eq:A4}
\end{align}
Additionally
\begin{align}
\mathbb{E}[F|Z_n^{\lambda,u}]&=q(q-1)2^{q-3}|B_n^{\lambda}|^q\notag\\
&=q(q-1)2^{q-3}\mathbb{E}\left[\left|2\langle \Delta_n^{\theta},G_n^{\lambda,\theta}\rangle+\dfrac{2}{N}\sum_{i=1}^N\langle \Delta_n^{\lambda,x,i},G_n^{\lambda,x,i}\rangle\right.\right.\notag\\
&\left.\left.+|G_n^{\lambda,\theta}|^2+\dfrac{1}{N}\sum_{i=1}^N|G_n^{\lambda,x,i}|^2\right|^q|Z_n^{\lambda,u} \right]\notag\\
&\leq q(q-1)2^{q-3}4^{q-1}\left(\mathbb{E}\left[2^q|\Delta_n^{\lambda,\theta}|^q|G_n^{\lambda,\theta}|^q|Z_n^{\lambda,u}\right]\right.\notag\\
&+\mathbb{E}\left[\dfrac{2^q}{N^q} \left|\sum_{i=1}^N |\Delta_n^{\lambda,x,i}||G_n^{\lambda,x,i}|\right|^q|Z_n^{\lambda,u}\right]+\mathbb{E}\left[|G_n^{\lambda,\theta}|^{2q}|Z_n^{\lambda,u}\right]\notag\\
&\left.+\mathbb{E}\left[\dfrac{1}{N^q}N^{q-1}\sum_{i=1}^N|G_n^{\lambda,x,i}|^{2q}|Z_n^{\lambda,u}\right]\right).\label{eq:011}
\end{align}
We handle the second term in equation \eqref{eq:011} by applying the multinomial expansion
$$\left|\sum_{i=1}^N |\Delta_n^{\lambda,x,i}||G_n^{\lambda,x,i}|\right|^q=\sum_{k_1+\ldots+k_N=q}\binom{q}{k_1,\ldots,k_N}\prod_{i=1}^N\left(|\Delta_n^{\lambda,x,i}||G_n^{\lambda,x,i}|\right)^{k_i},$$
and using the fact that $\Delta_n^{\lambda,x,i}$ are $Z_n^{\lambda,u}-$measurable, we obtain
\begin{align*}
&\dfrac{2^q}{N^q}\mathbb{E}\left[\left|\sum_{i=1}^N |\Delta_n^{\lambda,x,i}|G_n^{\lambda,x,i}|\right|^q|Z_n^{\lambda,u}\right]\\
&= \dfrac{2^q}{N^q}\mathbb{E}\left[\sum_{k_1+\ldots+k_N=q}\binom{q}{k_1,\ldots,k_N}\prod_{i=1}^N\left(|\Delta_n^{\lambda,x,i}||G_n^{\lambda,x,i}|\right)^{k_i}|Z_n^{\lambda,u}\right]\\
&=\dfrac{2^q}{N^q}\sum_{k_1+\ldots+k_N=q}\binom{q}{k_1,\ldots,k_N}\prod_{i=1}^N |\Delta_n^{\lambda,x,i}|^{k_i}\prod_{i=1}^N\mathbb{E}\left[|G_n^{\lambda,x,i}|^{k_i}|Z_n^{\lambda,u}\right]\\
&=\dfrac{2^q}{N^q}\sum_{k_1+\ldots+k_N=q}\binom{q}{k_1,\ldots,k_N}\prod_{i=1}^N |\Delta_n^{\lambda,x,i}|^{k_i}(2\lambda d^x)^{\sum_{i=1}^Nk_i/2}k_i!!\\
&\leq \dfrac{2^q q!!(2\lambda d^x)^{q/2}}{N^q}\sum_{k_1+\ldots+k_N=q}\binom{q}{k_1,\ldots,k_N}\prod_{i=1}^N |\Delta_n^{\lambda,x,i}|^{k_i}\\
&=\dfrac{2^q q!!(2\lambda d^x)^{q/2}}{N^q}\left(\sum_{i=1}^N |\Delta_n^{\lambda,x,i}|\right)^q\leq \dfrac{2^q q!!(2\lambda d^x)^{q/2}}{N^{q-1}}\left(\sum_{i=1}^N |\Delta_n^{\lambda,x,i}|^2\right)^{q/2}\\
&\leq 2^q q!!(2\lambda d^x)^{q/2}\left(\dfrac{1}{N}\sum_{i=1}^N |\Delta_n^{\lambda,x,i}|^2\right)^{q/2}.
\end{align*}Substituting these results into inequality \eqref{eq:011}, we derive
\begin{align}
\mathbb{E}[F|Z_n^{\lambda,u}]&\leq q(q-1)2^{3q-5}\left(2^qq!!(2\lambda d^{\theta}/N)^{q/2}\left(|\Delta_n^{\lambda,\theta}|^2\right)^{q/2}\right.\nonumber\\
&\left.+2^q q!!(2\lambda d^x)^{q/2}\left(\dfrac{1}{N}\sum_{i=1}^N |\Delta_n^{\lambda,x,i}|^2\right)^{q/2}+(2q)!!(2\lambda d^{\theta}/N)^{q}+(2q)!!(2\lambda d^{x})^{q}\right)\nonumber\\
&\leq q(q-1)2^{3q-5}\left(2^qq!!(2\lambda d)^{q/2}\left(|\Delta_n^{\lambda,\theta}|^2+\dfrac{1}{N}\sum_{i=1}^N |\Delta_n^{\lambda,x,i}|^2\right)^{q/2}+2(2q)!!(2\lambda d)^q\right)\nonumber\\
&\leq q(q-1)2^{4q-5}(2q)!!\left((2\lambda d)^{q/2} \left|A_n^{\lambda}\right|^{q/2}+(2\lambda d)^q\right)\nonumber\\
&\leq q(q-1)2^{4q-5}(2q)!!(2\lambda d)^{q/2}\left(r^{\lambda}_{q/2}|Z_n^{\lambda,u}|^q+w^{\lambda}_{q/2}\right)+q(q-1)2^{4q-5}(2q)!!(2\lambda d)^q.\label{eq:A5}
\end{align}
Combining the bounds in \eqref{eq:A4}-\eqref{eq:A5} yields the following control over the third term in \eqref{eq:010}
\begin{align}
\mathbb{E}\left[\sum_{k=2}^q\binom{q}{k}|A_n^{\lambda}|^{q-k}|B^{\lambda}_n|^k\right]
&\leq 16(2\lambda d)q(q-1)2^{q-3} \left(r_{q-1}^{\lambda}|Z_n^{\lambda,u}|^{2(q-1)}+w_{q-1}^{\lambda}\right)\nonumber\\&+24(2\lambda d)^2q(q-1)2^{q-3} \left(r_{q-2}^{\lambda}|Z_n^{\lambda,u}|^{2(q-2)}+w_{q-2}^{\lambda}\right)\nonumber\\
&+ (2q)!!(2\lambda d)^{q/2}q(q-1)2^{4q-5}\left(r^{\lambda}_{q/2}|Z_n^{\lambda,u}|^q+w^{\lambda}_{q/2}\right)\nonumber\\
&+(2q)!!(2\lambda d)^qq(q-1)2^{4q-5}.\label{eq:A6}
\end{align}
Hence, we obtain the following bound for \eqref{eq:010} via \eqref{eq:A1},\eqref{eq:A2} and \eqref{eq:A6}
\begin{align}
\mathbb{E}\left[|Z_{n+1}|^{2q}|Z_n\right]&\leq r_q^{\lambda}|Z_n^{\lambda,u}|^{2q}+w_q^{\lambda}\nonumber\\
&+4q\left(2\lambda d\right)(1+4(q-1)2^{q-3}) \left(r_{q-1}^{\lambda}|Z_n^{\lambda,u}|^{2(q-1)}+w_{q-1}^{\lambda}\right)\nonumber\\
&+24(2\lambda d)^2q(q-1)2^{q-3} \left(r_{q-2}^{\lambda}|Z_n^{\lambda,u}|^{2(q-2)}+w_{q-2}^{\lambda}\right)\nonumber\\
&+(2q)!!(2\lambda d)^{q/2}q(q-1)2^{4q-5}\left(r^{\lambda}_{q/2}|Z_n^{\lambda,u}|^q+w^{\lambda}_{q/2}\right)\nonumber\\
&+(2q)!!(2\lambda d)^qq(q-1)2^{4q-5}.\label{eq:A7}
\end{align}
We consider $|Z_n^{\lambda,u}|\geq \sqrt{8d/\mu}\left\{(2q)!!q(q-1)2^{4q-5}\right\}^{1/2}\geq \sqrt{8d/\mu}\left\{(2q)!!q(q-1)2^{4q-5}\right\}^{1/q}.$ Thus we compute
\begin{align*}
\mathbb{E}\left[|Z_{n+1}^{\lambda,u}|^{2q}|Z_n\right]&\leq r_q^{\lambda}|Z_n^{\lambda,u}|^{2q}+w_q^{\lambda}\\
&+\dfrac{\lambda\mu}{2\cdot 4}|Z_n^{\lambda,u}|^2\left(r_{q-1}^{\lambda}|Z_n^{\lambda,u}|^{2(q-1)}\right)+\dfrac{(\lambda\mu)^2}{2\cdot 4^2}|Z_n^{\lambda,u}|^4\left(r_{q-2}^{\lambda}|Z_n^{\lambda,u}|^{2(q-2)}\right)\\&+\dfrac{(\lambda\mu)^{q/2}}{2\cdot 4^{q/2}}|Z_n^{\lambda,u}|^{q}\left(r^{\lambda}_{q/2}|Z_n^{\lambda,u}|^q\right)+(2q)!!q(q-1)2^{4q-5}(2\lambda d)^q\\&+(2q)!!q(q-1)2^{4q-5}\left((2\lambda d)w_{q-1}^{\lambda}+(2\lambda d)^2w_{q-2}^{\lambda}+(2\lambda d)^{q/2}w_{q/2}^{\lambda}\right)\\
&\leq r^{\lambda}_{q/2}\left(\left(1-\dfrac{\lambda\mu}{4}\right)^{q/2}+\dfrac{1}{2}\left(\dfrac{\lambda\mu}{4}\right)\left(1-\dfrac{\lambda\mu}{4}\right)^{q/2-1}\right.\\
&\left.+\dfrac{1}{2}\left(\dfrac{\lambda\mu}{4}\right)^{2}\left(1-\dfrac{\lambda\mu}{4}\right)^{q/2-2}+\dfrac{1}{2}\left(\dfrac{\lambda\mu}{4}\right)^{q/2}\right)|Z_n^{\lambda,u}|^{2q}\\
&+w_q^{\lambda}+(2q)!!q(q-1)2^{4q-5}(2\lambda d)^q\\
&+(2q)!!q(q-1)2^{4q-5}\left((2\lambda d)w_{q-1}^{\lambda}+(2\lambda d)^2w_{q-2}^{\lambda}+(2\lambda d)^{q/2}w_{q/2}^{\lambda}\right).
\end{align*}
Using the fact that $\lambda\mu\leq 1$ we get
\begin{align*}
&\mathbb{E}\left[|Z_{n+1}^{\lambda,u}|^{2q}|Z_n\right]\\&\leq \left(1-\dfrac{\lambda\mu}{2}\right)|Z_n^{\lambda,u}|^{2q}+w_q^{\lambda}+(2q)!!q(q-1)2^{4q-5}(2\lambda d)^q\\
&+(2q)!!q(q-1)2^{4q-5}\left((2\lambda d)w_{q-1}^{\lambda}+(2\lambda d)^2w_{q-2}^{\lambda}+(2\lambda d)^{q/2}w_{q/2}^{\lambda}\right)\\
&\leq   \left(1-\dfrac{\lambda\mu}{2}\right)|Z_n^{\lambda,u}|^{2q}+(2q)!!q(q-1)2^{4q-2}(8/\mu)^{q-1}\lambda(2C)^q.
\end{align*}
Consequently, on $\{|Z_n^{\lambda,u}|\leq \sqrt{8d/\mu}\left\{(2q)!!q(q-1)2^{4q-5}\right\}^{1/2}\}$, we have
\begin{align*}
\mathbb{E}\left[|Z_{n+1}^{\lambda,u}|^{2q}|Z_n\right]&\leq  \left(1-\dfrac{\lambda\mu}{2}\right)|Z_n^{\lambda,u}|^{2q}\\&+(2q)!!q(q-1)2^{4q-2}(8/\mu)^{q-1}\lambda(2C)^q\\&+\left((2q)!!q(q-1)2^{4q-2}\right)^{q/2}(8/\mu)^{q-1}\lambda(2C)^q.
\end{align*}
Thus overall, we refine the bound in \eqref{eq:A7} to
\begin{align*}
\mathbb{E}\left[|Z_{n+1}^{\lambda,u}|^{2q}\right]&\leq  \left(1-\dfrac{\lambda\mu}{2}\right)\mathbb{E}\left[|Z_n^{\lambda,u}|^{2q}\right] +M_q(8/\mu)^{q} C^q\\
&\leq C_{|z_0|,b,q,\mu}(1+d^{\theta}/N+d^x)^q,
\end{align*}
where $C_{|z_0|,b,q,\mu}=\mathbb{E}\left[|Z_{0}^{\lambda}|^{2q}\right]M_q(8/\mu)^q(b+1)^{q}$ and $M_q=(2q)!!q(q-1)2^{6q-5}$.
\end{proof}
\noindent \begin{lemma} Let \hyperlink{Assum1}{A1}, \hyperlink{Assum2}{A2} and \hyperlink{Assum4}{A4} hold. Then, for every $\lambda_0<N^{p}/(4\mu)$, there exists a constant $C>0$, independent of $N,n,\lambda,d^x \text{and }d^{\theta}$, such that for any $\lambda\in(0,\lambda_0)$ one has
$$\mathbb{E}\left[|Z^{\lambda,u}_{n+1}-Z^{\lambda,u}_{n}|^4\right]\leq \lambda^2N^{-2p} C_{|z_0|,\mu,b}(1+d^{\theta}/N+d^x)^2,$$where $C$ is given explicitly within the proof.\end{lemma}
\begin{proof} We start by expanding the difference
\begin{align*}
|Z^{\lambda,u}_{n+1}-Z^{\lambda,u}_{n}|^4&=\left(|\theta^{\lambda,u}_{n+1}-\theta^{\lambda,u}_n|^2+\dfrac{1}{N}\sum_{i=1}^N|X_{n+1}^{i,\lambda,u}-X_n^{i,\lambda,u}|^2\right)^2\\
&=\left(\left|\dfrac{-\lambda}{N^{p+1}}\sum_{i=1}^N h^{\theta}_{\lambda,u}(\theta_n^{\lambda,u},X_n^{i,\lambda,u})+\sqrt{\dfrac{2\lambda}{N^{p+1}}}\xi_{n+1}^{(0)}\right|^2\right.\\
&\left.+\dfrac{1}{N}\sum_{i=1}^N\left|-\dfrac{\lambda}{N^p} h^{x}_{\lambda}(\theta_n^{\lambda,u},X_n^{i,\lambda,u})+\sqrt{\dfrac{2\lambda}{N^p}}\xi_{n+1}^{(i)}\right|^2\right)^2\\
&\leq \left(\dfrac{\lambda^2}{N^{2(p+1)}}|\sum_{i=1}^Nh^{\theta}_{\lambda,u}(\theta_n^{\lambda,u},X_n^{i,\lambda,u})|^2+\dfrac{\lambda^2}{N^{2p+1}}\sum_{i=1}^N| h^{x}_{\lambda}(\theta_n^{\lambda,u},X_n^{i,\lambda,u})|^2\right.\\
&-2\langle\dfrac{\lambda}{N^{p+1}}\sum_{i=1}^N h^{\theta}_{\lambda,u}(\theta_n^{\lambda,u},X_n^{i,\lambda,u}),\sqrt{\dfrac{2\lambda}{N^{p+1}}}\xi_{n+1}^{(0)}\rangle\\
&\left.-2\dfrac{\lambda}{N^{p+1}}\sum_{i=1}^N\langle h^x_{\lambda,u}(\theta_n^{\lambda,u},X_n^{i,\lambda,u}),\sqrt{\dfrac{2\lambda}{N^p}}\xi_{n+1}^{(i)} \rangle+\dfrac{2\lambda}{N^{p+1}}|\xi_{n+1}^{(0)}|^2+\dfrac{2\lambda}{N^{p+1}}\sum_{i=1}^N|\xi_{n+1}^{(i)} |^2 \right)^2
\end{align*}
Next, we further expand the expression as follows
\begin{gather}
|Z^{\lambda,u}_{n+1}-Z^{\lambda,u}_{n}|^4\leq\dfrac{\lambda^4}{N^{4p+2}}\left(\sum_{i=1}^N|h_{\lambda,u}(\theta_n^{\lambda,u},X_n^{i,\lambda,u})|^2\right)^2+2\left(\dfrac{\lambda^2}{N^{2p+1}}\sum_{i=1}^N|h_{\lambda,u}(\theta_n^{\lambda,u},X_n^{i,\lambda,u})|^2\right)\nonumber\\
\times\left(-2\langle\dfrac{\lambda}{N^{p+1}}\sum_{i=1}^Nh^{\theta}_{\lambda,u}(\theta_n^{\lambda,u},X_n^{i,\lambda,u}),\sqrt{\dfrac{2\lambda}{N^{p+1}}}\xi_{n+1}^{(0)}\rangle\right.\nonumber\\ \left.-2\dfrac{\lambda}{N^{p+1}}\sum_{i=1}^N\langle h^x_{\lambda,u}(\theta_n^{\lambda,u},X_n^{i,\lambda,u}),\sqrt{\dfrac{2\lambda}{N^p}}\xi_{n+1}^{(i)}  \rangle
+\dfrac{2\lambda}{N^{p+1}}|\xi_{n+1}^{(0)}|^2+\dfrac{2\lambda}{N^{p+1}}\sum_{i=1}^N|\xi_{n+1}^{(i)} |^2\right)\nonumber\\
+\left|-2\langle\dfrac{\lambda}{N^{p+1}}\sum_{i=1}^Nh^{\theta}_{\lambda,u}(\theta_n^{\lambda,u},X_n^{i,\lambda,u}),\sqrt{\dfrac{2\lambda}{N^{p+1}}}\xi_{n+1}^{(0)}\rangle\right.\nonumber\\
\left.-2\dfrac{\lambda}{N^{p+1}}\sum_{i=1}^N\langle h^x_{\lambda,u}(\theta_n^{\lambda,u},X_n^{i,\lambda,u}),\sqrt{\dfrac{2\lambda}{N^p}}\xi_{n+1}^{(i)}  \rangle +\dfrac{2\lambda}{N^{p+1}}|\xi_{n+1}^{(0)}|^2+\dfrac{2\lambda}{N^{p+1}}\sum_{i=1}^N|\xi_{n+1}^{(i)} |^2\right|^2.\label{eq:a1}
\end{gather}
The first term in \eqref{eq:a1} is directly bounded by the moments of $Z^{\lambda,u}_{n}$,
\begin{align}
&\mathbb{E}\left[\dfrac{\lambda^4}{N^{4p+2}}\left(\sum_{i=1}^N|h_{\lambda,u}(\theta_n^{\lambda,u},X_n^{i,\lambda,u})|^2\right)^2|Z^{\lambda,u}_{n}\right]\leq \dfrac{\lambda^4}{N^{4p+2}}\left(2\mu^2\sum_{i=1}^N|V_n^{i,\lambda,u}|^2+2\lambda^{-1}N^p\right)^2\nonumber\\
&\leq 4\mu^4\dfrac{\lambda^4}{N^{4p}}\dfrac{1}{N^2}\left(\sum_{i=1}^N|V_n^{i,\lambda,u}|^2\right)^2+4\dfrac{\lambda^2}{N^{2p}}\leq \mu^2\dfrac{\lambda^2}{N^{2p}}|Z^{\lambda,u}_{n}|^4+4\dfrac{\lambda^2}{N^{2p}}.\label{b3}
\end{align}
where the last inequality follows due to the restriction $\lambda<N^p/4\mu$.
In the second term, the inner products $\langle\cdot,\cdot\rangle$ vanish under expectation, leading to
\begin{align}
&2\left(\dfrac{\lambda^2}{N^{2p+1}}\sum_{i=1}^N|h_{\lambda,u}(\theta_n^{\lambda,u},X_n^{i,\lambda,u})|^2\right)\times\left(\dfrac{2\lambda}{N^{p+1}}\mathbb{E}\left[|\xi_{n+1}^{(0)}|^2|Z^{\lambda,u}_{n}\right]+\dfrac{2\lambda}{N^{p+1}}\sum_{i=1}^N\mathbb{E}\left[|\xi_{n+1}^{(i)} |^2|Z^{\lambda,u}_{n}\right]\right)\nonumber\\&\leq 2\dfrac{\lambda^2}{N^{2p}}\left(2\mu^2|Z^{\lambda,u}_{n}|^2+2\lambda^{-1}N^p\right)\left(2\dfrac{\lambda}{N^p} d^{\theta}/N+2\dfrac{\lambda}{N^p} d^x\right)\leq 8\dfrac{\lambda^3}{N^{3p}} d(\mu^2|Z^{\lambda,u}_{n}|^2+\lambda^{-1}N^p)\nonumber\nonumber\\
&\leq 2d\mu \dfrac{\lambda^2}{N^{2p}}|Z^{\lambda,u}_{n}|^2+8d\dfrac{\lambda^2}{N^{2p}}.\label{eq:b1}
\end{align}
By applying the usual elementary inequality and the Cauchy-Schwartz to the last term of \eqref{eq:a1}, we obtain
\begin{align}
&\dfrac{16\lambda^2}{N^{2p+2}}\left|\sum_{i=1}^Nh^{\theta}_{\lambda,u}(\theta_n^{\lambda,u},X_n^{i,\lambda,u})\right|^2\mathbb{E}\left[|\sqrt{\dfrac{2\lambda}{N^{p+1}}}\xi_{n+1}^{(0)}|^2\right]+\dfrac{16\lambda^2}{N^{2p+1}}\sum_{i=1}^N|h^{x}_{\lambda}(\theta_n^{\lambda,u},X_n^{i,\lambda,u})|^2\mathbb{E}\left[|\sqrt{\dfrac{2\lambda}{N^p}}\xi_{n+1}^{(i)} |^2\right]\nonumber\\&+\dfrac{16\lambda^2}{N^{2p+2}}\mathbb{E}\left[|\xi_{n+1}^{(0)}|^4\right]+\dfrac{16\lambda^2}{N^{2p+1}}\sum_{i=1}^N\mathbb{E}\left[|\xi_{n+1}^{(i)}|^4\right]\leq 16\dfrac{\lambda^2}{N^{2p}}\left(2\dfrac{\lambda}{N^p} d^{\theta}/N\right)\dfrac{1}{N}\sum_{i=1}^N|h^{\theta}_{\lambda,u}(\theta_n^{\lambda,u},X_n^{i,\lambda,u})|^2\nonumber\\&+16\dfrac{\lambda^2}{N^{2p}}\left(2\dfrac{\lambda}{N^p} d^x\right)\dfrac{1}{N}\sum_{i=1}^N|h^{x}_{\lambda}(\theta_n^{\lambda,u},X_n^{i,\lambda,u})|^2+16\dfrac{\lambda^2}{N^{2p}}(d^{\theta}/N)^2+16\dfrac{\lambda^2}{N^{2p}}(d^x)^2\nonumber
\end{align}
which, by using Property \hyperlink{Property1}{1} of the taming function, can be written as \begin{align}
&\dfrac{16\lambda^2}{N^{2p+2}}\left|\sum_{i=1}^Nh^{\theta}_{\lambda,u}(\theta_n^{\lambda,u},X_n^{i,\lambda,u})\right|^2\mathbb{E}\left[|\sqrt{\dfrac{2\lambda}{N^{p+1}}}\xi_{n+1}^{(0)}|^2\right]\nonumber\\
&\leq 16\dfrac{\lambda^3}{N^{3p}}(2d)\dfrac{1}{N}\sum_{i=1}^N\left(2\mu^2|V_n^{i,\lambda,u}|^2+2\lambda^{-1}N^p\right)+16\dfrac{\lambda^2}{N^{2p}}d^2\nonumber\\
&\leq 16\cdot 4\dfrac{\lambda^3}{N^{3p}}\mu^2 d|Z^{\lambda,u}_{n}|^2+16\cdot 4\dfrac{\lambda^2}{N^{2p}} d+16\dfrac{\lambda^2}{N^{2p}} d^2\nonumber\\
&\leq  16d\mu\dfrac{\lambda^2}{N^{2p}}|Z^{\lambda,u}_{n}|^2+64\dfrac{\lambda^2}{N^{2p}} d+16\dfrac{\lambda^2}{N^{2p}} d^2.\label{eq:b2}
\end{align}
Thus, combining \eqref{b3},\eqref{eq:b1} and \eqref{eq:b2} with the bound in \eqref{eq:a1}, we further refine the inequality as 
\begin{align*}
\mathbb{E}\left[|Z^{\lambda,u}_{n+1}-Z^{\lambda,u}_{n}|^4|Z^{\lambda,u}_{n}\right]&\leq \mu^2\dfrac{\lambda^2}{N^{2p}}|Z^{\lambda,u}_{n}|^4+18d\mu\dfrac{\lambda^2}{N^{2p}}|Z^{\lambda,u}_{n}|^2+4\dfrac{\lambda^2}{N^{2p}}(4+18d+d^2)\\
&\leq \dfrac{\lambda^2}{N^{2p}}\left(\mu^2|Z^{\lambda,u}_{n}|^4+18d\mu|Z^{\lambda,u}_{n}|^2+4(4+18d+d^2)\right).
\end{align*}
Finally, taking the expectation and using Lemmas \hyperlink{lemma1}{2} and  \hyperlink{lemma2}{3}, we conclude
\begin{align*}
\mathbb{E}\left[|Z^{\lambda,u}_{n+1}-Z^{\lambda,u}_{n}|^4\right]&\leq \dfrac{\lambda^2}{N^{2p}}81\left(\mu^2C^2_{|z_0|,\mu,b}+18\mu C_{|z_0|,b,2,\mu}+4\right)(1+d^{\theta}/N+d^x)^2.
\end{align*}
\end{proof}
\subsubsection{Proof of Proposition 3}\label{Proof3}
\begin{proof}
    Consider the rescaled version \eqref{eq:001} of the time-scaled dynamics \eqref{eq:02}-\eqref{eq:03} and denote by $\pi^N_{z,*}$, the rescaled invariant measure that corresponds to $\pi^N_*$, so that the rescaled initiation condition $\zeta=(\zeta_\theta,N^{-1/2}\zeta_{x^1},\ldots,N^{-1/2}\zeta_{x^N})$ follows the law $\mathcal{L}(\zeta)=\pi^N_{z,*}$. Now let $\mathcal{Z}_t^{u,N,\zeta}$ denote the aforementioned dynamics with fixed initial condition $Z_0^{u,N}=\zeta$. Then it follows that $\mathcal{L}(\mathcal{Z}_t^{u,N,\zeta})=\pi^N_{z,*}$. Since the difference $ \mathcal{Z}_t^{u,N}-\mathcal{Z}_t^{u,N,\zeta}$ has vanishing diffusion, using It\^o's formula and taking the expectation, one writes
    \begin{align*}
        \mathbb{E}\left[|\mathcal{Z}_t^{u,N}-\mathcal{Z}_t^{u,N,\zeta}|^2\right]&=\mathbb{E}\left[|\theta_0^N-\zeta_{\theta}|^2\right]+\dfrac{1}{N}\mathbb{E}\left[|X_0^{i,N}-\zeta_{x^i}|^2\right]\\&-\dfrac{2}{N^{p+1}}\sum_{i=1}^N\mathbb{E}\left[\int_0^t \left\langle \nabla_{\theta}U(\mathcal{V}_s^{i,u,N})-\nabla_{\theta}U(\mathcal{V}_s^{i,u,N,\zeta}),\vartheta_s^{u,N}-\vartheta_s^{u,N,\zeta} \right\rangle ds\right]\\&-\dfrac{2}{N^{p+1}}\sum_{i=1}^N\mathbb{E}\left[\int_0^t \left\langle \nabla_{x}U(\mathcal{V}_s^{i,u,N})-\nabla_{x}U(\mathcal{V}_s^{i,u,N,\zeta}),\mathcal{X}_s^{i,u,N}-\mathcal{X}_s^{i,u,N,\zeta} \right\rangle ds\right]\\
        &=\mathbb{E}\left[|Z_0^N-\zeta|^2\right]\\
        &-\dfrac{2}{N^{p+1}}\sum_{i=1}^N\mathbb{E}\left[\int_0^t \left\langle \nabla U(\mathcal{V}_s^{i,u,N})-\nabla U(\mathcal{V}_s^{i,u,N,\zeta}),\mathcal{V}_s^{i,u,N}-\mathcal{V}_s^{i,u,N,\zeta} \right\rangle ds\right].
        \end{align*}
    Then by differentiating both sides, and applying assumption \hyperlink{Assum2}{A2}, we obtain
    \begin{align*}
        \dfrac{d}{dt}\mathbb{E}\left[|\mathcal{Z}_t^{u,N}-\mathcal{Z}_t^{u,N,\zeta}|^2\right]\leq\dfrac{-2\mu}{N^{p+1}}\sum_{i=1}^N\mathbb{E}\left[|\mathcal{V}_t^{i,u,N}-\mathcal{V}_t^{i,u,N,\zeta}|^2\right]=\dfrac{-2\mu}{N^p}\mathbb{E}\left[|\mathcal{Z}_t^{u,N}-\mathcal{Z}_t^{u,N,\zeta}|^2\right],
    \end{align*}
    that is
    \begin{align*}
    \dfrac{d}{dt}\left(\exp(2\mu/N^p)\mathbb{E}\left[|\mathcal{Z}_t^{u,N}-\mathcal{Z}_t^{u,N,\zeta}|^2\right]\right)\leq 0.    
    \end{align*}
    Integrating over $[0,t]$ and dividing by $\exp(2\mu/N^p)$ yields
    \begin{align*}
    \mathbb{E}\left[|\mathcal{Z}_t^{u,N}-\mathcal{Z}_t^{u,N,\zeta}|^2\right]\leq \exp(-2\mu/N^p)\mathbb{E}\left[|\mathcal{Z}_0^{u,N}-\mathcal{Z}_0^{u,N,\zeta}|^2\right]= \exp(-2\mu/N^p)\mathbb{E}\left[|Z_0^N-\zeta|^2\right].
    \end{align*}
    Hence, one writes
    \begin{align*}
        W_2\left(\mathcal{L}(\vartheta_t^{u,N}),\pi_{\Theta}^N\right)^2&=\mathbb{E}\left[|\vartheta_t^{u,N}-\vartheta_t^{u,N,\zeta}|^2\right]\leq\mathbb{E}\left[|\mathcal{Z}_t^N-\mathcal{Z}_t^{N,\zeta}|^2\right]\\&\leq\exp(-2\mu/N^p)\mathbb{E}\left[|Z_0^N-\zeta|^2\right]\\
        &\leq\exp(-2\mu/N^p)\left(\mathbb{E}\left[|Z_0^N-z^*|^2\right]+\mathbb{E}\left[|z^*-\zeta|^2\right]\right)\\
        &\leq\exp(-2\mu/N^p)\left(\mathbb{E}\left[|Z_0^N-z^*|^2\right]+\dfrac{d^x N+d^{\theta}}{\mu N^{p+1}}\right).
    \end{align*}
    To bound the last term one considers the difference between $\mathcal{Z}_t^{u,N}$ and $z^*$, where $\nabla U(z^*)=0$, since $z^*$ is a minimiser. Similarly with the aforementioned calculations, one uses It\^o's formula, takes the expectation and after applying assumption \hyperlink{Assum2}{A2} and computing the derivative, yields
    \begin{align*}
        \dfrac{d}{dt}\mathbb{E}\left[|\mathcal{Z}_t^{u,N}-z^*|^2\right]\leq -2\mu\mathbb{E}\left[|\mathcal{Z}_t^{u,N}-z^*|^2\right]+2\dfrac{d^x N+d^{\theta}}{ N^{p+1}},
    \end{align*}
    which through integration implies
    \begin{align*}
     \mathbb{E}\left[|\mathcal{Z}_t^{u,N}-z^*|^2\right]\leq\dfrac{d^x N+d^{\theta}}{\mu N^{p+1}}.
    \end{align*}
    Therefore one immediately obtains
    \begin{align*}
    W_2(\mathcal{L}(\zeta),\delta_{z^*})&=W_2(\pi^N_{z,*},\delta_{z^*})\leq \limsup_{t\to\infty}W_2(\mathcal{L}(\mathcal{Z}_t^{u,N}),\delta_{z^*})\leq \sqrt{\dfrac{d^x N+d^{\theta}}{\mu N^{p+1}}}.
    \end{align*}
\end{proof}
\subsubsection{Proof of Proposition 5}\label{Proof5}
\begin{proof}
    \noindent We begin by considering the standard decomposition of the difference between the interpolation \eqref{rescale31} and the scaled dynamics \eqref{rescale21}. By applying the It\^{o}'s formula to $x\to |x|^2$, we obtain
\begin{gather*}
|\overline{Z}^{\lambda,u}_t-{\mathcal{Z}}_{\lambda t}^{u,N}|^2=|\overline{\theta}^{\lambda,u}_t-{\vartheta}_{\lambda t}^{u,N}|^2+\dfrac{1}{N}\sum_{i=1}^N|\overline{X}^{i,\lambda,u}_t-\mathcal{X}_{\lambda t}^{i,u,N}|^2\\
=-2\lambda\int_0^t\left\langle \dfrac{1}{N^{p+1}}\sum_{i=1}^N h^{\theta}_{\lambda,u}(\theta_{\lfloor{t}\rfloor}^{\lambda,u},X^{i,\lambda,u}_{\lfloor{t}\rfloor})-\dfrac{1}{N^{p+1}}\sum_{i=1}^Nh^{\theta}({\vartheta}_{\lambda t}^{u,N},\mathcal{X}_{\lambda t}^{i,u,N}),\overline{{\theta}}_{t}^{\lambda,u}-{\vartheta}_{\lambda t}^{u,N}\right\rangle ds\\
-\dfrac{2\lambda}{N^{p+1}}\sum_{i=1}^N\int_0^t\left\langle h^{x}_{\lambda,u}(\theta_{\lfloor{t}\rfloor}^{\lambda,u},X^{i,\lambda,u}_{\lfloor{t}\rfloor})-h^{x}({\vartheta}_{\lambda t}^{u,N},\mathcal{X}_{\lambda t}^{i,u,N}), \overline{X}_{t}^{i,\lambda,u}-\mathcal{X}_{\lambda t}^{i,u,N}\right\rangle ds\\
=-\dfrac{2\lambda}{N^{p+1}}\sum_{i=1}^N\int_0^t \langle h_{\lambda,u}(V_{\lfloor{t}\rfloor}^{i,\lambda,u})-h(\mathcal{V}_{\lambda t}^{i,u,N}) , \overline{V}_{t}^{i,\lambda,u}-\mathcal{V}_{\lambda t}^{i,u,N}\rangle ds.
\end{gather*}
To make the subsequent calculations more accessible, we introduce the following notations
$$e^i_t=\overline{V}_{t}^{i,\lambda,u}-\mathcal{V}_{\lambda t}^{i,u,N}\ \text{and}\ e_t=\overline{Z}^{\lambda,u}_t-\mathcal{Z}_{\lambda t}^{u,N},$$ this implies that, $1/N\sum_{i=1}^N |e^i_t|^2=|e_t|^2.$ Taking the derivative of this expression, we obtain
\begin{align}
\dfrac{d}{dt}|e_t|^2&=-\dfrac{2\lambda}{N^{p+1}}\sum_{i=1}^N\langle  h_{\lambda,u}(V_{\lfloor{t}\rfloor}^{i,\lambda,u})-h(\mathcal{V}_{\lambda t}^{i,u,N}), e_t^i\rangle\nonumber\\
&=-\dfrac{2\lambda}{N^{p+1}}\sum_{i=1}^N\langle  h(\overline{V}_{t}^{i,\lambda,u})-h(\mathcal{V}_{\lambda t}^{i,u,N}), e_t^i\rangle-\dfrac{2\lambda}{N^{p+1}}\sum_{i=1}^N\langle  h_{\lambda,u}(V_{\lfloor{t}\rfloor}^{i,\lambda,u})-h(V_{\lfloor{t}\rfloor}^{i,\lambda,u}), e_t^i\rangle\nonumber\\
&-\dfrac{2\lambda}{N^{p+1}}\sum_{i=1}^N\langle  h(V_{\lfloor{t}\rfloor}^{i,\lambda,u})-h(\overline{V}_{t}^{i,\lambda,u}), e_t^i\rangle\nonumber\\
&:=k_1(t)+k_2(t)+k_3(t).\label{c1}
\end{align}
The first term, $k_1(t)$ is controlled through the convexity of the initial potential, while the taming error $k_2(t)$ is controlled by the properties of the taming function and the moment bounds established earlier. The same bounds also control the discretisation error $k_3(t)$. Specifically we have
\begin{align}
k_1(t)&=-\dfrac{2\lambda}{N^{p+1}}\sum_{i=1}^N\langle  h(\overline{V}_{t}^{i,\lambda,u})-h(\mathcal{V}_{\lambda t}^{i,u,N}), e_t^i\rangle\nonumber\\
&\leq-\dfrac{2\lambda\mu}{N^{p+1}}\sum_{i=1}^N |e_t^i|^2=-\dfrac{2\lambda\mu}{N^p}|e_t|^2\label{c2}.
\end{align}
By applying the Cauchy-Schwarz inequality and the $\epsilon$-Young inequality with $\epsilon=\mu/2$, we can further deduce
\begin{align}
k_2(t)&=-\dfrac{2\lambda}{N^{p+1}}\sum_{i=1}^N\langle  h_{\lambda,u}(V_{\lfloor{t}\rfloor}^{i,\lambda,u})-h(V_{\lfloor{t}\rfloor}^{i,\lambda,u}), e_t^i\rangle\nonumber\\
&\leq\dfrac{2\lambda}{N^{p+1}}\sum_{i=1}^N\left(\dfrac{\mu}{4}|e_t^i|^2+\dfrac{1}{\mu}|h_{\lambda,u}(V_{\lfloor{t}\rfloor}^{i,\lambda,u})-h(V_{\lfloor{t}\rfloor}^{i,\lambda,u})|^2\right)\nonumber\\
&\leq \dfrac{2\lambda\mu}{4N^p} |e_t|^2+\dfrac{2\lambda}{\mu N^{p+1}}\sum_{i=1}^N\dfrac{\lambda C_1}{N^p}(1+|V_{\lfloor{t}\rfloor}^{i,\lambda,u}|^{4(\ell+1)})\nonumber\\
&\leq \dfrac{\lambda\mu}{2N^p} |e_t|^2+\dfrac{2\lambda^2C_1}{\mu N^{2p}} +\dfrac{2\lambda^2C_1}{\mu N^p}\left(\dfrac{1}{N}\sum_{i=1}^N|V_{\lfloor{t}\rfloor}^{i,\lambda,u}|^2\right)^{2(\ell+1)}\nonumber\\&\leq\dfrac{\lambda\mu}{2N^p} |e_t|^2+\dfrac{2\lambda^2C_1}{\mu N^{p}} +\dfrac{2\lambda^2C_1}{\mu N^p}|Z_{\lfloor{t}\rfloor}^{\lambda,u}|^{4(\ell+1)}.\label{c3}
\end{align}
Note that the constant $C_1$ is introduced in Property \hyperlink{Property2}{2} and $p$ was chosen such that $p=2\ell+1$, ensuring that the latter expression involves a power of $|Z_{\lfloor{t}\rfloor}^{\lambda,u}|$. Similarly, for the final term, we derive
\begin{align*}
k_3(t)&=-\dfrac{2\lambda}{N^{p+1}}\sum_{i=1}^N\langle  h(V_{\lfloor{t}\rfloor}^{i,\lambda,u})-h(\overline{V}_{t}^{i,\lambda,u}), e_t^i\rangle\\
&\leq\dfrac{2\lambda}{N^{p+1}}\sum_{i=1}^N\left(\dfrac{\mu}{4}|e_t^i|^2+\dfrac{1}{\mu}| h(V_{\lfloor{t}\rfloor}^{i,\lambda,u})-h(\overline{V}_{t}^{i,\lambda,u})|^2\right)\\
&\leq\dfrac{\lambda\mu}{2N^p}|e_t|^2+\dfrac{2\lambda C_2}{\mu N^{p+1}}\sum_{i=1}^N \left((1+|V_{\lfloor{t}\rfloor}^{i,\lambda,u}|^{2\ell}+|\overline{V}_{t}^{i,\lambda,u}|^{2\ell})|V_{\lfloor{t}\rfloor}^{i,\lambda,u}-\overline{V}_{t}^{i,\lambda,u}|^2\right)\\
&\leq\dfrac{\lambda\mu}{2N^p}|e_t|^2+\dfrac{2\lambda C_2}{\mu N^{p+1}}\sum_{i=1}^N \left(1+|V_{\lfloor{t}\rfloor}^{i,\lambda,u}|^{2\ell}+|\overline{V}_{t}^{i,\lambda,u}|^{2\ell}\right)\sum_{i=1}^N|V_{\lfloor{t}\rfloor}^{i,\lambda,u}-\overline{V}_{t}^{i,\lambda,u}|^2.
\end{align*}
Rearranging terms and taking expectations, and then applying the Holder's inequality, leads us to
\begin{align}
\mathbb{E}\left[k_3(t)\right]&\leq \dfrac{\lambda\mu}{2N^p}\mathbb{E}\left[|e_t|^2\right]\nonumber\\
&+\dfrac{2\lambda C_2}{\mu N^{(p+1)/2}}\mathbb{E}\left[\dfrac{1}{N^{\ell}}\sum_{i=1}^N \left(1+|V_{\lfloor{t}\rfloor}^{i,\lambda,u}|^{2\ell}+|\overline{V}_{t}^{i,\lambda,u}|^{2\ell}\right)\dfrac{1}{N}\sum_{i=1}^N|V_{\lfloor{t}\rfloor}^{i,\lambda,u}-\overline{V}_{t}^{i,\lambda,u}|^2\right]\nonumber\\
&\leq\dfrac{\lambda\mu}{2N^p}\mathbb{E}\left[|e_t|^2\right]+\dfrac{2\lambda C_2}{\mu}\mathbb{E}\left[\left(1+|Z_{\lfloor{t}\rfloor}^{\lambda,u}|^{2\ell}+|\overline{Z}_{t}^{\lambda,u}|^{2\ell}\right)|Z_{\lfloor{t}\rfloor}^{\lambda,u}-\overline{Z}_{t}^{\lambda,u}|^2\right]\nonumber\\
&\leq\dfrac{\lambda\mu}{2N^p}\mathbb{E}\left[|e_t|^2\right]+\dfrac{6\lambda C_2}{\mu}\mathbb{E}^{1/2}\left[\left(1+|Z_{\lfloor{t}\rfloor}^{\lambda,u}|^{4\ell}+|\overline{Z}_{t}^{\lambda,u}|^{4\ell}\right)\right]\mathbb{E}^{1/2}\left[|Z_{\lfloor{t}\rfloor}^{\lambda,u}-\overline{Z}_{t}^{\lambda,u}|^4\right]\nonumber\\
&\leq\dfrac{\lambda\mu}{2N^p}\mathbb{E}\left[|e_t|^2\right]+\dfrac{6\lambda^2 C_2C_3}{\mu N^p}\mathbb{E}^{1/2}\left[\left(1+|Z_{\lfloor{t}\rfloor}^{\lambda,u}|^{4\ell}+|\overline{Z}_{t}^{\lambda,u}|^{4\ell}\right)\right](1+d^{\theta}/N+d^x)\nonumber\\
&\leq\dfrac{\lambda\mu}{2N^p}\mathbb{E}\left[|e_t|^2\right]+\dfrac{6\lambda^2 C_2C_3C_4}{\mu N^p}(1+d^{\theta}/N+d^x)^{\ell+1}.\label{c4}
\end{align}
Next, taking the expectation of the remaining terms in \eqref{c2} and \eqref{c3} respectively and combining them with \eqref{c4}, we arrive, in light of \eqref{c1}, at
\begin{align*}
\dfrac{d}{dt}\mathbb{E}\left[|e_t|^2\right]&\leq -\dfrac{\mu\lambda}{N^p}\mathbb{E}\left[|e_t|^2\right]+\dfrac{2\lambda^2C_1}{\mu N^{p}} +\dfrac{2\lambda^2C_1C_5}{\mu N^p}(1+d^{\theta}/N+d^x)^{2(\ell+1)}\\
&+\dfrac{6\lambda^2 C_2C_3C_4}{\mu N^p}(1+d^{\theta}/N+d^x)^{\ell+1}\\
&\leq  -\dfrac{\mu\lambda}{N^p}\mathbb{E}\left[|e_t|^2\right]+\dfrac{\lambda^2 C_6}{\mu N^p}(1+d^{\theta}/N+d^x)^{2(\ell+1)},
\end{align*}
Finally, multiplying both sides by the integrating factor $\exp{(\lambda\mu t/N^p)}$ and rearranging the term yields
\begin{gather*}
\dfrac{d}{dt}\left(\exp{\left(\lambda\mu t/N^p\right)}\mathbb{E}\left[|e_t|^2\right]\right)\leq\dfrac{\lambda^2 C_6}{\mu N^p}(1+d^{\theta}/N+d^x)^{2(\ell+1)} \exp{\left(\lambda\mu t/N^p\right)}\\
\mathbb{E}\left[|e_t|^2\right]\leq \lambda \dfrac{C_6}{\mu^2}(1+d^{\theta}/N+d^x)^{2(\ell+1)}.
\end{gather*}
\end{proof}
\subsection{Coordinate-wise tamed scheme}
\subsubsection{Key quantities for the proof of the main Lemmas.}
The following definitions that refer to the rescaled dynamics \eqref{eq:006}-\eqref{eq:007} of the algorithm \hyperref[algo2]{tIPLAc} and its continuous time interpolations \eqref{eq:008}-\eqref{eq:009} are given as\begin{align} Z^{\lambda,c}_{n}=\left(\theta_{n+1}^{\lambda,c},N^{-1/2}X_{n+1}^{1,\lambda,c},\ldots,N^{-1/2}X_{n+1}^{N,\lambda,c}\right),\label{rescale_cord}\end{align}
\begin{align} \overline{Z}^{\lambda,c}_t=\left(\overline{\theta}^{\lambda,c}_t,N^{-1/2}\overline{X}^{1,\lambda,c}_t,\ldots,N^{-1/2}\overline{X}^{N,\lambda,c}_t\right),\label{rescale2_cord}\end{align}
\begin{align} \mathcal{Z}_{\lambda t}^{c,N}=\left(\vartheta_{\lambda t}^{c,N},N^{-1/2}\mathcal{X}_{\lambda t}^{1,c,N},\ldots,N^{-1/2}\mathcal{X}_{\lambda t}^{N,c,N}\right).\label{rescale3_cord} \end{align}
\subsubsection{Moment and increment bounds}
\begin{lemma}\label{lemma4} Let \hyperlink{Assum1}{A1}, \hyperlink{Assum3a}{A3-i} and \hyperlink{Assum4}{A4} hold. Then, for any $0<\lambda<1/(4\mu)$, it holds that,
$$\mathbb{E}\left[|V_{n}^{i,\lambda,c}|^2\right]\leq C_{|z_0|,\mu,b}(1+d^{\theta}+d^x),$$
for a constant $C>0$ independent of $N,n,\lambda, d^x \text{and } d^{\theta}$, $\forall i\in\{1,\ldots,N\}$. If assumption \hyperlink{Assum3b}{A3-ii} is used in place of \hyperlink{Assum3a}{A3-i}, the lemma holds under the updated timestep constraint, $\lambda<1/(8\mu)$.\end{lemma}
\begin{proof} Consider the rescaled iterates as described in equation (34), which expands as
    \begin{align*}
\left|V_{n+1}^{i,\lambda,c}\right|^2&=\left|\theta_{n+1}^{\lambda,c}\right|^2+ \left|X_{n+1}^{i,\lambda,c}\right|^2\\ 
&=\left|\theta_{n}^{\lambda,c}-\dfrac{\lambda}{N}\sum_{i=1}^N h^{\theta}_{\lambda,c}(\theta_{n}^{\lambda,c},X_{n}^{i,\lambda,c})\right|^2+\dfrac{2\lambda}{N}|\xi^{(0)}_{n+1}|^2\\&+2\sqrt{\dfrac{2\lambda}{N}}\left\langle \theta_{n}^{\lambda,c}-\dfrac{\lambda}{N}\sum_{i=1}^N h^{\theta}_{\lambda,c}(\theta_{n}^{\lambda,c},X_{n}^{i,\lambda,c}),\xi_{n+1}^{(0)}\right\rangle\\
&+\left|X_{n}^{i,\lambda,c}-\lambda h^{x}_{\lambda,c}(\theta_{n}^{\lambda,c},X_{n}^{i,\lambda,c})\right|^2+2\lambda|\xi^{(i)}_{n+1}|^2\\&+2\sqrt{2\lambda}\left\langle X_{n}^{i,\lambda,c}-\lambda h^{x}_{\lambda,c}(\theta_{n}^{\lambda,c},X_{n}^{i,\lambda,c}),\xi_{n+1}^{(i)}\right\rangle.
\end{align*}
Taking the conditional expectation on both sides with respect to the filtration generated by $V_{n}^{i,\lambda,c}$, the cross terms vanish to $0$, leading to
\begin{align*}
\mathbb{E}\left[\left|V_{n+1}^{i,\lambda,c}\right|^2|V_{n}^{i,\lambda,c}\right]&=\mathbb{E}\left[\left|\theta_{n}^{\lambda,c}\right|^2|V_{n}^{i,\lambda,c}\right]-\dfrac{2\lambda}{N}\sum_{i=1}^N\mathbb{E}\left[\langle \theta_n^{\lambda,c},h_{\lambda,c}^{\theta}(\theta_{n}^{\lambda,c},X_{n}^{i,\lambda,c})\rangle |V_{n}^{i,\lambda,c}\right]\\
&+\dfrac{\lambda^2}{N^2}\mathbb{E}\left[\left|\sum_{i=1}^Nh_{\lambda,c}^{\theta}(\theta_{n}^{\lambda,c},X_{n}^{i,\lambda,c})\right|^2|V_{n}^{i,\lambda,c}\right]+\dfrac{2\lambda d^{\theta}}{N}\\
&+\mathbb{E}\left[\left|X_{n}^{i,\lambda,c}\right|^2|V_{n}^{i,\lambda,c}\right]-2\lambda\mathbb{E}\left[\langle X_n^{i,\lambda,c},h_{\lambda,c}^{x}(\theta_{n}^{\lambda,c},X_{n}^{i,\lambda,c})\rangle|V_{n}^{i,\lambda,c}\right]\\&+{\lambda}^2\mathbb{E}\left[|h_{\lambda,c}^x(\theta_{n}^{\lambda,c},X_{n}^{i,\lambda,c})|^2|V_{n}^{i,\lambda,c}\right]+2\lambda d^x.
\end{align*}
Furthermore, by using the elementary inequality $(t_1+\ldots+t_m)^p\leq m^{p-1}(t_1^p+\ldots,t_m^p)$ and recognizing that all expressions within the conditional expectations are measurable, we obtain
\begin{align*}
\mathbb{E}\left[\left|V_{n+1}^{i,\lambda,c}\right|^2|V_{n}^{i,\lambda,c}\right]&\leq |\theta_{n}^{\lambda,c}|-\dfrac{2\lambda}{N}\sum_{i=1}^N\langle \theta_n^{\lambda,c},h_{\lambda,c}^{\theta}(\theta_{n}^{\lambda,c},X_{n}^{i,\lambda,c})\rangle+\dfrac{\lambda^2}{N}\sum_{i=1}^N|h_\lambda^{\theta}(\theta_{n}^{\lambda,c},X_{n}^{i,\lambda,c})|^2+\dfrac{2\lambda d^{\theta}}{N} \\ &+|X_n^{i,\lambda,c}|^2-2\lambda\langle X_n^{i,\lambda,c},h_{\lambda,c}^{x}(\theta_{n}^{\lambda,c},X_{n}^{i,\lambda,c}) \rangle+{\lambda}^2|h_{\lambda,c}^{x}(\theta_{n}^{\lambda,c},X_{n}^{i,\lambda,c})|^2+2\lambda d^x\\
&\leq \left|V_{n}^{i,\lambda,c}\right|^2-\dfrac{2\lambda}{N}\sum_{i=1}^N\sum_{j=1}^{d^{\theta}}\theta_{n,j}^{\lambda}\cdot h^{\theta_j}_{\lambda,c}(\theta_n^{\lambda,c},X_n^{i,\lambda,c})\\
&+\dfrac{\lambda^2}{N}\sum_{i=1}^N\sum_{j=1}^{d^{\theta}}|h_{\lambda,c}^{\theta_j}(\theta_{n}^{\lambda,c},X_{n}^{i,\lambda,c})|^2-2\lambda\sum_{j=1}^{d^{x}}X_{n,j}^{i,\lambda}\cdot h^{x_j}_{\lambda,c}(\theta_n^{\lambda,c},X_n^{i,\lambda,c})\\
&+\lambda^2\sum_{j=1}^{d^{x}}|h_{\lambda,c}^{x_j}(\theta_{n}^{\lambda,c},X_{n}^{i,\lambda,c})|^2+2\lambda(d^{\theta}/N+d^x)\\
&\leq \left|V_{n}^{i,\lambda,c}\right|^2-\dfrac{2\lambda}{N}\sum_{i=1}^N\sum_{j=1}^{d^{\theta}}\left(\dfrac{\mu}{2}|\theta^{\lambda,c}_{n,j}|^2-b\right)\\
&+\dfrac{\lambda^2}{N}\sum_{i=1}^N\sum_{j=1}^{d^{\theta}}\left(2\mu^2|\theta^{\lambda,c}_{n,j}|^2+2\lambda^{-1}\right)-2\lambda\sum_{j=1}^{d^{x}}\left(\dfrac{\mu}{2}|X^{i,\lambda,c}_{n,j}|^2-b\right)\\&+\lambda^2\sum_{j=1}^{d^{x}}\left(2\mu^2|X^{i,\lambda,c}_{n,j}|^2+2\lambda^{-1}\right)+2\lambda(d^{\theta}/N+d^x).
\end{align*}
Proceeding to execute the summations, we arrive at
\begin{align*}
\mathbb{E}\left[\left|V_{n+1}^{i,\lambda,c}\right|^2|V_{n}^{i,\lambda,c}\right]&\leq \left|V_{n}^{i,\lambda,c}\right|^2-\lambda\mu|\theta^{\lambda,c}_{n}|^2-\lambda\mu|X^{i,\lambda,c}_{n}|^2+2\lambda^2\mu^2|\theta^{\lambda,c}_{n}|^2\\
&+2\lambda^2\mu^2|X^{i,\lambda,c}_{n}|^2+2\lambda b(d^{\theta}+d^x)+2\lambda(d^{\theta}+d^x)+2\lambda(d^{\theta}/N+d^x)\\
&\leq (1-\lambda\mu+2\lambda^2\mu^2)\left|V_{n}^{i,\lambda,c}\right|^2+4\lambda(b+1)(d^{\theta}+d^x).
\end{align*}
Now, consider the restriction $\lambda\mu<1/4$, which implies
\begin{align*}
\mathbb{E}\left[\left|V_{n+1}^{i,\lambda,c}\right|^2\right]&\leq (1-\lambda\mu/2)\mathbb{E}\left[\left|V_{n}^{i,\lambda,c}\right|^2\right]+2\lambda(b+1)(d^{\theta}+d^x)+2\lambda(d^{\theta}/N+d^x).
\end{align*}
By iterating the above bound, we finally deduce 
\begin{align*}
\mathbb{E}\left[\left|V_{n+1}^{i,\lambda,c}\right|^2\right]&\leq (1-\lambda\mu/2)^n\left|V_{0}^{i,\lambda,c}\right|^2+\dfrac{1-(1-\lambda\mu/2)^n}{\lambda\mu/2}4\lambda(b+1)(d^{\theta}+d^x)\\
&\leq \left(\left|V_{0}\right|^2+ (8/\mu)(b+1)\right)(1+d^{\theta}+d^x).
\end{align*}
To establish the second part of the Lemma, we repeat the same arguments up to the point where \hyperlink{Assum3a}{A3-i} is invoked. This leads to
\begin{align*}
    \mathbb{E}\left[\left|V_{n+1}^{i,\lambda,c}\right|^2|V_{n}^{i,\lambda,c}\right]&\leq \left|V_{n}^{i,\lambda,c}\right|^2-\dfrac{2\lambda}{N}\sum_{i=1}^N\sum_{j=1}^{d^{\theta}}\left(\dfrac{\mu}{2}|\theta^{\lambda,c}_{n,j}|^2-\rho|X_{n,j}^{i,\lambda,c}|^2-b\right)\\&+ \dfrac{\lambda^2}{N}\sum_{i=1}^N\sum_{j=1}^{d^{\theta}}\left(2\mu^2|\theta^{\lambda,c}_{n,j}|^2+2\lambda^{-1}\right)\\&-2\lambda\sum_{j=1}^{d^{x}}\left(\dfrac{\mu}{2}|X^{i,\lambda,c}_{n,j}|^2-\rho|\theta^{\lambda,c}_{n,j}|^2-b\right)\\&+\lambda^2\sum_{j=1}^{d^{x}}\left(2\mu^2|X^{i,\lambda,c}_{n,j}|^2+2\lambda^{-1}\right)+2\lambda(d^{\theta}/N+d^x).\\
    &\leq \left(1-\lambda\mu+2\lambda\rho+2\lambda^2\mu^2\right)\left|V_{n}^{i,\lambda,c}\right|^2+4\lambda(b+1)(d^x+d^{\theta})
\end{align*}
Finally, according to \hyperlink{Assum3b}{A3-ii}, since $4\rho<\mu$ implies $2\rho-\mu<-\mu/2$, and considering the constrain $\lambda\mu<1/8$ we derive
\begin{align*}
    \mathbb{E}\left[\left|V_{n+1}^{i,\lambda,c}\right|^2\right]&\leq (1-\lambda\mu/4)\mathbb{E}\left[\left|V_{n}^{i,\lambda,c}\right|^2\right]+4\lambda(b+1)(d^x+d^{\theta})\\
    &\leq \left(|V_0|^2+(16/\mu)(b+1)\right)\left(1+d^x+d^{\theta}\right).
\end{align*}
\end{proof}
\noindent \begin{lemma} \label{lemma6}Let \hyperlink{Assum1}{A1}, \hyperlink{Assum3a}{A3-i} and \hyperlink{Assum4}{A4} hold. Then, for any $0<\lambda<1/(4\mu)$ and $q\in[2,2\ell+1)\cap\mathbb{N}$, it holds that, 
$$\mathbb{E}\left[|V_{n}^{i,\lambda,c}|^{2q}\right]\leq C_{|z_0|,\mu,b,p}(1+d^{\theta}+d^x)^q,$$
for a constant $C>0$ independent of $N,n,\lambda, d^x \text{and } d^{\theta}$, $\forall i\in\{1,\ldots,N\}$. If assumption \hyperlink{Assum3b}{A3-ii} is used in place of \hyperlink{Assum3a}{A3-i}, the lemma holds under the updated timestep constraint $\lambda<1/(8\mu)$.\end{lemma}
\begin{proof}
Similarly to the proof of Lemma \hyperlink{lemma2}{3}, we define the following auxiliary quantities
\begin{gather*}
A_n^{\lambda}:=|\Delta_n^{\lambda,\theta}|^2+|\Delta_n^{\lambda,i,x}|^2,\\
B^{\lambda}_n=2\langle \Delta_n^{\theta},G_n^{\lambda,\theta}\rangle+2\langle \Delta_n^{\lambda,x,i},G_n^{\lambda,x,i}\rangle+|G_n^{\lambda,\theta}|^2+|G_n^{\lambda,x,i}|^2.
\end{gather*}
Now, for the 2q-th moment, we write
\begin{align}
|V^{i,\lambda,c}_{n+1}|^{2q}&=\left(A_n^{\lambda}+B^{\lambda}_n\right)^q\nonumber\\
&\leq (A_n^{\lambda})^q+2q(A_n^{\lambda})^{q-1}B^{\lambda}_n+\sum_{k=2}^q\binom{q}{k}|A_n^{\lambda}|^{q-k}|B^{\lambda}_n|^k\label{d1}
\end{align}
We will handle each term separately. For the first time, we have
\begin{align}
\mathbb{E}\left[|A_n^{\lambda}|^q|Z_n^{\lambda}\right]&=(A_n^{\lambda})^{q}\leq \left((1-\lambda\mu/2)\left|V_{n}^{i,\lambda,c}\right|^2+2\lambda(b+1)(d^{\theta}+d^x)\right)^q\nonumber\\
&\leq \left(1+\dfrac{\lambda \mu}{4}\right)^{q-1}\left(1-\dfrac{\lambda\mu}{2}\right)^q|V^{i,\lambda,c}_{n}|^{2q}+\left(1+\dfrac{4}{\lambda\mu }\right)^{q-1}2^q{\lambda}^q C^q\nonumber\\
&\leq \left(1-\dfrac{\lambda\mu}{4}\right)^{q-1}\left(1-\dfrac{\lambda\mu}{2}\right)|V^{i,\lambda,c}_{n}|^{2q}+\left(\lambda+\dfrac{4}{\mu }\right)^{q-1}\lambda (2C)^q\nonumber\\
&\leq r_q^{\lambda}|V^{i,\lambda,c}_{n}|^{2q}+w_q^{\lambda}.\label{d2}
\end{align}
where $r_q^{\lambda}=(1-\lambda\mu/4)^{q-1}(1-\lambda\mu/2)$ and $w_q^{\lambda}=(\lambda+4/\mu)^{q-1}\lambda (2C)^q$.\\
\noindent Next, we handle the second term as follows
\begin{align}
\mathbb{E}\left[2q(A_n^{\lambda})^{q-1}B^{\lambda}_n|V^{i,\lambda,c}_{n}\right]&=2q(A_n^{\lambda})^{q-1}\mathbb{E}\left[B^{\lambda}_n|Z_n^{\lambda}\right]=4q\lambda\left(d^{\theta}/N+d^x\right)(A_n^{\lambda})^{q-1}\nonumber\\
&\leq 4q\lambda\left(d^{\theta}/N+d^x\right) \left(r_{q-1}^{\lambda}|V^{i,\lambda,c}_{n}|^{2(q-1)}+w_{q-1}^{\lambda}\right)\nonumber\\
&\leq  4q\left(2\lambda d\right) \left(r_{q-1}^{\lambda}|V^{i,\lambda,c}_{n}|^{2(q-1)}+w_{q-1}^{\lambda}\right).\label{d3}
\end{align}
The third term on the right-hand side of \eqref{d1} can be expanded further as
\begin{align}
\sum_{k=2}^q\binom{q}{k}|A_n^{\lambda}|^{q-k}|B^{\lambda}_n|^k&=\sum_{\ell=0}^{q-2}\binom{q}{\ell+2}|A_n^{\lambda}|^{q-2-k}|B^{\lambda}_n|^{k+2}\nonumber\\
&=\dfrac{q}{\ell+2}\dfrac{q-1}{\ell+1}\sum_{\ell=0}^{q-2}\binom{q}{\ell}|A_n^{\lambda}|^{q-2-k}|B^{\lambda}_n|^{k}|B^{\lambda}_n|^{2}\nonumber\\
&\leq q(q-1)\left(|A_n^{\lambda}|+|B_n^{\lambda}|\right)^{q-2}|B^{\lambda}_n|^2\nonumber\\
&\leq q(q-1)2^{q-3}|A_n^{\lambda}|^{q-2}|B_n^{\lambda}|^2+q(q-1)2^{q-3}|B_n^{\lambda}|^q.\nonumber\\
&=D+F \label{d4}
\end{align}
Taking the expectation of the above terms, we get
\begin{align*}
\mathbb{E}[D|V^{i,\lambda,c}_{n}]&=q(q-1)2^{q-3}|A_n^{\lambda}|^{q-2}\mathbb{E}\left[|B_n^{\lambda}|^2|V^{i,\lambda,c}_{n}\right],
\end{align*}
where
\begin{align*}
&\mathbb{E}\left[|B_n^{\lambda}|^2|V^{i,\lambda,c}_{n}\right]\\&=\mathbb{E}\left[\left|2\langle \Delta_n^{\theta},G_n^{\lambda,\theta}\rangle+2\langle \Delta_n^{\lambda,x,i},G_n^{\lambda,x,i}\rangle+|G_n^{\lambda,\theta}|^2+|G_n^{\lambda,x,i}|^2\right|^2|V^{i,\lambda,c}_{n} \right]\\
&\leq 4\left(\mathbb{E}\left[4|\Delta_n^{\lambda,\theta}|^2|G_n^{\lambda,\theta}|^2|V^{i,\lambda,c}_{n}\right]+\mathbb{E}\left[4|\Delta_n^{\lambda,x,i}|^2|G_n^{\lambda,x,i}|^2|V^{i,\lambda,c}_{n}\right]\right.\\
&+\left.\mathbb{E}\left[|G_n^{\lambda,x,i}|^4|V^{i,\lambda,c}_{n}\right]+\mathbb{E}\left[|G_n^{\lambda,x,i}|^4|V^{i,\lambda,c}_{n}\right]\right)\\
&\leq 4\left(4(2\lambda d^{\theta}/N)|\Delta_n^{\lambda,\theta}|^2+4(2\lambda d^x)|\Delta_n^{\lambda,x,i}|^2+12(2\lambda d^{\theta}/N)^2+12(2\lambda d^x)^2\right)\\
&\leq 16(2\lambda d)\left(|\Delta_n^{\lambda,\theta}|^2+|\Delta_n^{\lambda,x,i}|^2\right)+96(2\lambda d)^2\\
&\leq  16(2\lambda d)A_n^{\lambda}+96(2\lambda d)^2.
\end{align*}
with $d=\max\{d^\theta/N,d^x\}$. Substituting this result for the term D, we derive
\begin{align*}
\mathbb{E}[D|V^{i,\lambda,c}_{n}]&=q(q-1)2^{q-3}|A_n^{\lambda}|^{q-2}\left(16(2\lambda d)A_n^{\lambda}+96(2\lambda d)^2\right)\\
&=16(2\lambda d)q(q-1)2^{q-3} \left(r_{q-1}^{\lambda}|V_n^{\lambda}|^{2(q-1)}+w_{q-1}^{\lambda}\right)\\&+96(2\lambda d)^2q(q-1)2^{q-3} \left(r_{q-2}^{\lambda}|V_n^{\lambda}|^{2(q-2)}+w_{q-2}^{\lambda}\right).
\end{align*}
Additionally
\begin{align*}
&\mathbb{E}[F|V^{i,\lambda,c}_{n}]=q(q-1)2^{q-3}|B_n^{\lambda}|^q\\
&=q(q-1)2^{q-3}\mathbb{E}\left[\left|2\langle \Delta_n^{\theta},G_n^{\lambda,\theta}\rangle+2\langle \Delta_n^{\lambda,x,i},G_n^{\lambda,x,i}\rangle+|G_n^{\lambda,\theta}|^2+|G_n^{\lambda,x,i}|^2\right|^q|V^{i,\lambda,c}_{n} \right]\\
&\leq q(q-1)2^{q-3}4^{q-1}\left(\mathbb{E}\left[2^q|\Delta_n^{\lambda,\theta}|^q|G_n^{\lambda,\theta}|^q|V^{i,\lambda,c}_{n}\right]+\mathbb{E}\left[2^q  |\Delta_n^{\lambda,x,i}|^q|G_n^{\lambda,x,i}|^q|V^{i,\lambda,c}_{n}\right]\right.\\
&\left.+\mathbb{E}\left[|G_n^{\lambda,\theta}|^{2q}|V^{i,\lambda,c}_{n}\right]+\mathbb{E}\left[|G_n^{\lambda,x,i}|^{2q}|V^{i,\lambda,c}_{n}\right]\right)\\
&\leq q(q-1)2^{3q-5}\left(2^q(2\lambda d^{\theta}/N)^{q/2}q!!\left(|\Delta_n^{\lambda,\theta}|^2\right)^{q/2}+2^q q!!(2\lambda d^x)^{q/2}\left(|\Delta_n^{\lambda,x,i}|^2\right)^{q/2}\right.\\
&\left.+(2q)!!(2\lambda d^{\theta}/N)^{q}+(2q)!!(2\lambda d^{x})^{q}\right)\\
&\leq (q(q-1))^{q+1}2^{4q-5}\left((2\lambda d)^{q/2}\left(|\Delta_n^{\lambda,\theta}|^2+ |\Delta_n^{\lambda,x,i}|^2\right)^{q/2}+2(2\lambda d)^q\right)\\
&\leq (q(q-1))^{q+1}2^{4q-5}\left((2\lambda d)^{q/2} \left|A_n^{\lambda}\right|^{q/2}+2(2\lambda d)^q\right)\\
&\leq (q(q-1))^{q+1}2^{4q-5}(2\lambda d)^{q/2}\left(r^{\lambda}_{q/2}|V_n^{\lambda}|^q+w^{\lambda}_{q/2}\right)\\
&+(q(q-1))^{q+1}2^{4q-5}2(2\lambda d)^q.
\end{align*}
Combining the above results in view of \eqref{d4}, yields
\begin{align}
&\mathbb{E}\left[\sum_{k=2}^q\binom{q}{k}|A_n^{\lambda}|^{q-k}|B^{\lambda}_n|^k\right]\nonumber\\
&\leq 16(2\lambda d)q(q-1)2^{q-3} \left(r_{q-1}^{\lambda}|V_n^{\lambda}|^{2(q-1)}+w_{q-1}^{\lambda}\right)\nonumber\\&+96(2\lambda d)^2q(q-1)2^{q-3} \left(r_{q-2}^{\lambda}|V_n^{\lambda}|^{2(q-2)}+w_{q-2}^{\lambda}\right)\nonumber\\
&+ (q(q-1))^{q+1}2^{4q-5}(2\lambda d)^{q/2}\left(r^{\lambda}_{q/2}|V_n^{\lambda}|^q+w^{\lambda}_{q/2}\right)\nonumber\\
&+(q(q-1))^{q+1}2^{4q-4}(2\lambda d)^q.\label{d5}
\end{align}
Substituting \eqref{d2},\eqref{d3} and \eqref{d5} into \eqref{d1} gives the following
\begin{align}
\mathbb{E}\left[|V^{i,\lambda,c}_{n+1}|^{2q}|V^{i,\lambda,c}_{n}\right]&\leq r_q^{\lambda}|V^{i,\lambda,c}_{n}|^{2q}+w_q^{\lambda}\nonumber\\
&+4q\left(2\lambda d\right)(1+4(q-1)2^{q-3}) \left(r_{q-1}^{\lambda}|V^{i,\lambda,c}_{n}|^{2(q-1)}+w_{q-1}^{\lambda}\right)\nonumber\\
&+96(2\lambda d)^2q(q-1)2^{q-3} \left(r_{q-2}^{\lambda}|V^{i,\lambda,c}_{n}|^{2(q-2)}+w_{q-2}^{\lambda}\right)\nonumber\\
&+ (q(q-1))^{q+1}2^{4q-5}(2\lambda d)^{q/2}\left(r^{\lambda}_{q/2}|V^{i,\lambda,c}_{n}|^q+w^{\lambda}_{q/2}\right)\nonumber\\
&+(q(q-1))^{q+1}2^{4q-4}(2\lambda d)^q.\label{d6}
\end{align}
Now, consider $$|V^{i,\lambda,c}_{n}|\geq \sqrt{8d/\mu}\left\{2(q(q-1))^{q+1}2^{4(q-1)}\right\}^{1/2}\geq \sqrt{8d/\mu}\left\{2(q(q-1))^{q+1}2^{4(q-1)}\right\}^{1/q}.$$ From this, we observe
\begin{align*}
&\mathbb{E}\left[|V^{i,\lambda,c}_{n+1}|^{2q}|V^{i,\lambda,c}_{n}\right]\\&\leq r_q^{\lambda}|V_n^{i,\lambda,c}|^{2q}+w_q^{\lambda}+\dfrac{\lambda\mu}{2\cdot 4}|V^{i,\lambda,c}_{n}|^2\left(r_{q-1}^{\lambda}|V^{i,\lambda,c}_{n}|^{2(q-1)}\right)\\
&+\dfrac{(\lambda\mu)^2}{2\cdot 4^2}|V^{i,\lambda,c}_{n}|^4\left(r_{q-2}^{\lambda}|V^{i,\lambda,c}_{n}|^{2(q-2)}\right)\\&+\dfrac{(\lambda\mu)^{q/2}}{2\cdot 4^{q/2}}|V_n^{i,\lambda,c}|^{q}\left(r^{\lambda}_{q/2}|V_n^{i,\lambda,c}|^q\right)+(q(q-1))^{q+1}2^{4(q-1)}(2\lambda d)^q\\&+(q(q-1))^{q+1}2^{4(q-1)}\left((2\lambda d)w_{q-1}^{\lambda}+(2\lambda d)^2w_{q-2}^{\lambda}+(2\lambda d)^{q/2}w_{q/2}^{\lambda}\right)\\
&\leq r^{\lambda}_{q/2}\left(\left(1-\dfrac{\lambda\mu}{4}\right)^{q/2}+\dfrac{1}{2}\left(\dfrac{\lambda\mu}{4}\right)\left(1-\dfrac{\lambda\mu}{4}\right)^{q/2-1}\right.\\
&\left.+\dfrac{1}{2}\left(\dfrac{\lambda\mu}{4}\right)^{2}\left(1-\dfrac{\lambda\mu}{4}\right)^{q/2-2}+\dfrac{1}{2}\left(\dfrac{\lambda\mu}{4}\right)^{q/2}\right)|V_n^{i,\lambda,c}|^{2q}\\
&+w_q^{\lambda}+(q(q-1))^{q+1}2^{4(q-1)}(2\lambda d)^q\\
&+(q(q-1))^{q+1}2^{4(q-1)}\left((2\lambda d)w_{q-1}^{\lambda}+(2\lambda d)^2w_{q-2}^{\lambda}+(2\lambda d)^{q/2}w_{q/2}^{\lambda}\right).
\end{align*}
Using the fact that $\lambda\mu\leq 1$, we further simplify the bound
\begin{align*}
&\mathbb{E}\left[|V^{i,\lambda,c}_{n+1}|^{2q}|V^{i,\lambda,c}_{n}\right]\\&\leq \left(1-\dfrac{\lambda\mu}{2}\right)|V^{i,\lambda,c}_{n}|^{2q}\\
&+w_q^{\lambda}+(q(q-1))^{q+1}2^{4(q-1)}(2\lambda d)^q\\
&+(q(q-1))^{q+1}2^{4(q-1)}\left((2\lambda d)w_{q-1}^{\lambda}+(2\lambda d)^2w_{q-2}^{\lambda}+(2\lambda d)^{q/2}w_{q/2}^{\lambda}\right)\\
&\leq   \left(1-\dfrac{\lambda\mu}{2}\right)|V^{i,\lambda,c}_{n}|^{2q}+(q(q-1))^{q+1}2^{4(q-1)}10^qC^q(d\lambda)^q.
\end{align*}
Consequently, on $\{|V^{i,\lambda,c}_{n}|\leq \sqrt{8d/\mu}\left\{2(q(q-1))^{q+1}2^{4(q-1)}\right\}^{1/2}\}$, we derive
\begin{align*}
\mathbb{E}\left[|V^{i,\lambda,c}_{n+1}|^{2q}|V^{i,\lambda,c}_{n}\right]&\leq  \left(1-\dfrac{\lambda\mu}{2}\right)|V^{i,\lambda,c}_{n}|^{2q}\\&+(q(q-1))^{q+1}2^{4(q-1)}10^qC^q(d\lambda)^q\\&+3\left(q(q-1)^{q+1}2^{4(q-1)}\right)^q(2d)^q\left(\dfrac{4}{\mu}\right)^q2^{q-1}.
\end{align*}
Thus, the bound provided in \eqref{d6} is improved to
\begin{align*}
\mathbb{E}\left[|V^{i,\lambda,c}_{n+1}|^{2q}\right]&\leq  \left(1-\dfrac{\lambda\mu}{2}\right)\mathbb{E}\left[|V^{i,\lambda,c}_{n}|^{2q}\right]+C(d,q,\lambda,\mu),
\end{align*}
Hence, the usual bound for recursive sequences follows, completing the proof.
\end{proof}
\subsubsection{Proof of Proposition 6}\label{Proof6}
\begin{proof}
    We begin by considering the standard decomposition of the difference between the interpolation (35) and the scaled dynamics (36). By applying the It\^{o}'s formula to $x\to |x|^2$, we obtain
\begin{gather*}
|\overline{Z}^{\lambda,c}_t-\mathcal{Z}_{\lambda t}^{c,N}|^2=|\overline{\theta}^{\lambda,c}_t-{\vartheta}_{\lambda t}^{c,N}|^2+\dfrac{1}{N}\sum_{i=1}^N|\overline{X}^{i,\lambda,c}_t-\mathcal{X}_{\lambda t}^{i,c,N}|^2\\
=-2\lambda\int_0^t\left\langle \dfrac{1}{N}\sum_{i=1}^N h^{\theta}_{\lambda,c}(\theta_{\lfloor{t}\rfloor}^{\lambda,c},X^{i,\lambda,c}_{\lfloor{t}\rfloor})-\dfrac{1}{N}\sum_{i=1}^Nh^{\theta}({\vartheta}_{\lambda t}^{c,N},\mathcal{X}_{\lambda t}^{i,c,N}),\overline{{\theta}}_{t}^{i,\lambda,c}-{\vartheta}_{\lambda t}^{c,N}\right\rangle ds\\
-\dfrac{2\lambda}{N}\sum_{i=1}^N\int_0^t\left\langle h^{x}_{\lambda,c}(\theta_{\lfloor{t}\rfloor}^{\lambda,c},X^{i,\lambda,c}_{\lfloor{t}\rfloor})-h^{x}({\vartheta}_{\lambda t}^{c,N},\mathcal{X}_{\lambda t}^{i,c,N}), \overline{X}_{t}^{i,\lambda,c}-\mathcal{X}_{\lambda t}^{i,c,N}\right\rangle ds\\
=-\dfrac{2\lambda}{N}\sum_{i=1}^N\int_0^t \langle h_{\lambda,c}(V_{\lfloor{t}\rfloor}^{i,\lambda,c})-h(\mathcal{V}_{\lambda t}^{i,c,N}) , \overline{V}_{t}^{i,\lambda,c}-\mathcal{V}_{\lambda t}^{i,c,N}\rangle ds.
\end{gather*}
In order to make the forthcoming calculations more readable, we define the following quantities
$$e^i_t=\overline{V}_{t}^{i,\lambda,c}-\mathcal{V}_{\lambda t}^{i,c,N}\ \text{and}\ e_t=\overline{Z}^{\lambda,c}_t-\mathcal{Z}_{\lambda t}^{c,N}.$$ Hence it follows that $1/N\sum_{i=1}^N |e^i_t|^2=|e_t|^2.$ Taking the derivative of this expression, we arrive at
\begin{align}
\dfrac{d}{dt}|e_t|^2&=-\dfrac{2\lambda}{N}\sum_{i=1}^N\langle  h_{\lambda,c}(V_{\lfloor{t}\rfloor}^{i,\lambda,c})-h(\mathcal{V}_{\lambda t}^{i,c,N}), e_t^i\rangle\nonumber\\
&=-\dfrac{2\lambda}{N}\sum_{i=1}^N\langle  h(\overline{V}_{t}^{i,\lambda,c})-h(\mathcal{V}_{\lambda t}^{i,c,N}), e_t^i\rangle\nonumber\\
&-\dfrac{2\lambda}{N}\sum_{i=1}^N\langle  h_{\lambda,c}(V_{\lfloor{t}\rfloor}^{i,\lambda,c})-h(V_{\lfloor{t}\rfloor}^{i,\lambda,c}), e_t^i\rangle\nonumber\\
&-\dfrac{2\lambda}{N}\sum_{i=1}^N\langle  h(V_{\lfloor{t}\rfloor}^{i,\lambda,c})-h(\overline{V}_{t}^{i,\lambda,c}), e_t^i\rangle\nonumber\\
&=k_1(t)+k_2(t)+k_3(t).\label{f1}
\end{align}
The first term $k_1(t)$ is controlled through the convexity of the initial potential. The taming error $k_2(t)$ is controlled by the properties of the taming function and the additional moment bounds we have already established. These bounds also allow us to control the discretisation error $k_3(t)$. Specifically, we have
\begin{align}
k_1(t)&=-\dfrac{2\lambda}{N}\sum_{i=1}^N\langle  h(\overline{V}_{t}^{i,\lambda,c})-h(\mathcal{V}_{\lambda t}^{i,c,N}), e_t^i\rangle\leq-\dfrac{2\lambda\mu}{N}\sum_{i=1}^N |e_t^i|^2=-2\lambda\mu|e_t|^2.\label{f2}
\end{align}
By applying the Cauchy-Schwarz inequality and the $\epsilon$-Young inequality with $\epsilon=\mu/2$, we get
\begin{align}
k_2(t)&=-\dfrac{2\lambda}{N}\sum_{i=1}^N\langle  h_{\lambda,c}(V_{\lfloor{t}\rfloor}^{i,\lambda,c})-h(V_{\lfloor{t}\rfloor}^{i,\lambda,c}), e_t^i\rangle\leq\dfrac{2\lambda}{N}\sum_{i=1}^N\left(\dfrac{1}{\mu}|h_{\lambda,c}(V_{\lfloor{t}\rfloor}^{i,\lambda,c})-h(V_{\lfloor{t}\rfloor}^{i,\lambda,c})|^2+\dfrac{\mu}{4}|e_t^i|^2\right)\nonumber\\
&\leq 2\dfrac{\lambda\mu}{4} |e_t|^2+\dfrac{2\lambda}{\mu N}\sum_{i=1}^N2\lambda C_2^2(1+|V_{\lfloor{t}\rfloor}^{i,\lambda,c}|^{4(\ell+1)})\leq 2\dfrac{\lambda\mu}{4} |e_t|^2+\dfrac{4\lambda^2}{\mu } C_2^2+\dfrac{4\lambda^2}{\mu N}\sum_{i=1}^N|V_{\lfloor{t}\rfloor}^{i,\lambda,c}|^{4(\ell+1)}.\label{f3}
\end{align}
Similarly, for the final term, we derive
\begin{align}
k_3(t)&=-\dfrac{2\lambda}{N}\sum_{i=1}^N\langle  h(V_{\lfloor{t}\rfloor}^{i,\lambda,c})-h(\overline{V}_{t}^{i,\lambda,c}), e_t^i\rangle\leq\dfrac{2\lambda}{N}\sum_{i=1}^N\left(\dfrac{1}{\mu}| h(V_{\lfloor{t}\rfloor}^{i,\lambda,c})-h(\overline{V}_{t}^{i,\lambda,c})|^2+\dfrac{\mu}{4}|e_t^i|^2\right)\nonumber\\
&\leq2\dfrac{\lambda\mu}{4}|e_t|^2+\dfrac{2\lambda L^2}{\mu N}\sum_{i=1}^N \left((1+|V_{\lfloor{t}\rfloor}^{i,\lambda,c}|^{\ell}+|\overline{V}_{t}^{i,\lambda,c}|^{\ell})^2|V_{\lfloor{t}\rfloor}^{i,\lambda,c}-\overline{V}_{t}^{i,\lambda,c}|^2\right).\label{f4}
\end{align}
Now, combining the bounds in \eqref{f2},\eqref{f3} and \eqref{f4} and substituting them in \eqref{f1}, we can take the expectation, which leads to
\begin{align*}
\dfrac{d}{dt}\mathbb{E}\left[|e_t|^2\right]&\leq -\mu\lambda\mathbb{E}\left[|e_t|^2\right]+\dfrac{4\lambda^2C_2^2}{\mu}+\dfrac{4\lambda^2}{\mu}\dfrac{1}{N}\sum_{i=1}^N\mathbb{E}\left[|V_{\lfloor{t}\rfloor}^{i,\lambda,c}|^{4(\ell+1)}\right]\\
&+\dfrac{2\lambda L^2}{\mu}\dfrac{1}{N}\sum_{i=1}^N\mathbb{E}^{1/2}\left[(1+|V_{\lfloor{t}\rfloor}^{i,\lambda,c}|^{\ell}+|\overline{V}_{t}^{i,\lambda,c}|^{\ell})^4\right]\mathbb{E}^{1/2}\left[|V_{\lfloor{t}\rfloor}^{i,\lambda,c}-\overline{V}_{t}^{i,\lambda,c}|^4\right]\\
&\leq -\mu\lambda\mathbb{E}\left[|e_t|^2\right]+\dfrac{4\lambda^2 (C_2^2+C_3)}{\mu}\\&+\dfrac{2\lambda L^2}{\mu}\left(6\max\{1,2C_4\}\right)\lambda\left(\mu C_5+6\sqrt{d\mu}C_6+10d+2\right)\\
&\leq -\mu\lambda\mathbb{E}\left[|e_t|^2\right]+\lambda^2 M(L,\mu,d,\ell)\leq \lambda\left(-\dfrac{\mu}{2}\mathbb{E}\left[|e_t|^2\right]+\lambda M\right).
\end{align*}
Finally, multiplying both sides by the integrating factor $\exp{\left(\lambda\mu t/2\right)}$ and rearranging the terms, we conclude
\begin{gather*}
\dfrac{d}{dt}\left(\exp{(\lambda\mu t/2)}\mathbb{E}\left[|e_t|^2\right]\right)\leq \lambda^2 M \exp{(\lambda\mu t/2)}\\
\mathbb{E}\left[|e_t|^2\right]\leq \lambda \dfrac{2 M}{\mu}.
\end{gather*}
\end{proof}
\section{Verifying Assumptions in Model Examples}
\subsection{Bayesian logistic regression with thin tails}
\begin{remark}\label{remark3}
    Consider the potential $U(\theta,x)$ as given in \eqref{regU}, it satisfies \hyperlink{Assum1}{A1}, i.e. it is locally Lipschitz with polynomial growth.
\end{remark}
\begin{proof}
     It is quite straightforward to compute
\begin{align*}
    |\nabla U(v)-\nabla U(v')|&\leq |\nabla U_x(v)-\nabla U_x(v')|+ |\nabla U_{\theta}(v)-\nabla U_{\theta}(v')|\\
    &\leq \dfrac{8}{\sigma^4}\left(x-\theta-x'+\theta')\right)\left(3/2|x-\theta|^2+3/2|x'-\theta'|^2\right)\\&+\dfrac{4}{\sigma^2}\left(|x-x'|+|\theta-\theta'|\right)+1/4\sum_{j=1}^{d^y}|u_j|^2|x-x'|\\
    &\leq \dfrac{16}{\sigma^4}\left(|v-v'|\right)\left(4|v|^2+4|v'|^2\right)+\left(\dfrac{8}{\sigma^2}+1/4\sum_{j=1}^{d^y}|u_j|^2\right)|v-v'|\\
    &\leq \max\left(\dfrac{64}{\sigma^4},\dfrac{8}{\sigma^2}+1/4\sum_{j=1}^{d^y}|u_j|^2\right)\left(1+|v|^2+|v'|^2\right)|v-v'|.
\end{align*}
We note that the last term in the second inequality follows from the fact that the logistic function is Lipschitz continuous with constant $1/4$.
\end{proof}
\begin{remark}\label{remark4}
    Consider the potential $U(\theta,x)$ as given in \eqref{regU}, it satisfies \hyperlink{Assum2}{A2} with $\mu=2/\sigma^2$, i.e. it is strongly convex.
\end{remark}
\begin{proof}
    First, we compute the second-order partial derivatives as follows
    \begin{align*}
\dfrac{\partial^2 U_y}{\partial \theta_i \theta_r}&=\delta_{ir}\left(\dfrac{12}{\sigma^4}(x_i-\theta_i)^2+\dfrac{2}{\sigma^2}\right),\\
\dfrac{\partial^2 U_y}{\partial x_i x_r}&=\delta_{ir}\left(\dfrac{12}{\sigma^4}(x_i-\theta_i)^2+\dfrac{2}{\sigma^2}\right)+\sum_{j=1}^{dy}s(u_j^Tx)\left(1-s(u_j^Tx)\right)u_{j,i}u_{j,r},\\
\dfrac{\partial^2 U_y}{\partial x_i\theta_r}&=-\delta_{ir}\left(\dfrac{12}{\sigma^4}(x_i-\theta_i)^2+\dfrac{2}{\sigma^2}\right).
\end{align*}
    \noindent We define the diagonal matrix $A$, with elements $a_{ir}=\delta_{ir}\left(12/\sigma^4(x_i-\theta_i)^2+2/\sigma^2\right)$, to write the Hessian of $U(\theta,x)$ in the following form
\begin{align*}H\left(U(\theta,x)\right)&=\begin{bmatrix}
A& -A\\
-A & A
\end{bmatrix}+\left[0,0;0,\sum_{j=1}^{dy}s(u_j^Tx)\left(1-s(u_j^Tx)\right)u_j\otimes u_j\right].
\end{align*} 
From this we observe that the leftmost term is a positive definite matrix due to its diagonal block structure, while the second term is semi-positive definite (see further details in \cite{Kuntz}). Hence $H(U(\theta,x))$ is positive definite, making $U(\theta,x)$ strongly convex with a convexity constant $\mu=2/\sigma^2$. 
\end{proof}
\subsection{Higher order - Toy Example}
\begin{remark}\label{remark5}
    Consider the potential $U(\theta,x)=|x|^{4m}+(|x|^{2m}+1)(|\theta|^{2m}+1)+|\theta|^{4m}$. Then $U(\theta,x)$ satisfies \hyperlink{Assum1}{A1}, i.e. it is locally Lipschitz with polynomial growth and constant $L=2m2^{4m}$.
\end{remark}
\begin{proof}
    We decompose the difference of the gradient of $U$ as follows
\begin{align}
    |\nabla U(v)-\nabla U(v')|\leq|\nabla_x U(\theta,x)-\nabla_x U(x',\theta')|+|\nabla_{\theta} U(\theta,x)-\nabla_{\theta} U(x',\theta')|\label{B1}
\end{align}
Due to the symmetry of $U$ it suffices to consider only the first term incorporating $\nabla_x U$,
\begin{align}
    &|\nabla_x U(\theta,x)-\nabla_x U(x',\theta')|\nonumber\\&=\left|4m|x|^{4m-2}x+2m\left(|\theta|^{2m}+1\right)|x|^{2m-2}x\right.\nonumber\\ 
    &\left.-4m|x'|^{4m-2}x'-2m\left(|\theta'|^{2m}+1\right)|x'|^{2m-2}x'\right|\nonumber\\ 
    &\leq 4m\left||x|^{4m-2}x-|x'|^{4m-2}x'\right|+2m\left|\left(|\theta|^{2m}+1\right)|x|^{2m-2}x\right.\nonumber\\
    &\left.-\left(|\theta'|^{2m}+1\right)|x'|^{2m-2}x'\right|\nonumber\\
    &\leq 4m|x|\left||x|^{4m-2}-|x'|^{4m-2}\right|+4m|x'|^{4m-2}|x-x'|\nonumber\\
    &+2m\left(|\theta|^{2m}+1\right)|x|^{2m-2}|x-x'|\nonumber\\
    &+2m|x'|\left|\left(|\theta|^{2m}+1\right)|x|^{2m-2}-\left(|\theta'|^{2m}+1\right)|x'|^{2m-2}\right|\nonumber\\
    &\leq4m\sum_{i=0}^{4m-3}|x|^{4m-2-i}|x'|^{i}|x-x'|+4m|x'|^{4m-2}|x-x'|\nonumber\\
    &+2m|x|^{2m-2}|x-x'|+2m|v|^{4m-2}|x-x'|\nonumber\\
    &+2m|x'|\left||x|^{2m-2}-|x'|^{2m-2}\right|+2m|x'||\theta|^{2m}\left||x|^{2m-2}-|x'|^{2m-2}\right|\nonumber\\
    &+2m|x'|^{2m-1}\left||\theta|^{2m}-|\theta'|^{2m}\right|\label{B2}.
\end{align}
Next, we define the expression $J_x$ which contains all the terms with common factor $|x-x'|$. Notice that due to the symmetry of the problem, the second term of \eqref{B1} yields only one such term, corresponding to the last term in \eqref{B2}, that is $2m|\theta'|^{2m-1}\left||x|^{2m}-|x'|^{2m}\right|$. Hence $J_x$ is bounded as follows
\begin{align}
    J_x&=2m\left(2\sum_{i=0}^{4m-2}|x|^{4m-2-i}|x'|^{i}+|v|^{4m-2}+\sum_{i=0}^{2m-2}|x|^{2m-2-i}|x'|^{i}\right.\nonumber\\
    &\left.+|\theta|^{2m}\sum_{i=0}^{2m-3}|x|^{2m-3-i}|x'|^{i+1}+|\theta'|^{2m-1}\sum_{i=0}^{2m-1}|x|^{2m-1-i}|x'|^i\right)|x-x'|\nonumber\\
    &\leq2m\left(3\sum_{i=0}^{4m-2}|x|^{4m-2-i}|x'|^{i}+|v|^{4m-2}+\dfrac{m(2m-3)}{2m-1}|\theta|^{4m-2}\right.\nonumber\\
    &+\dfrac{m-1}{2m-1}\sum_{i=0}^{2m-3}|x|^{(2m-3-i)(2m-1/m-1)}|x'|^{(i+1)(2m-1/m-1)}\nonumber\\
    &\left.+\dfrac{2m-1}{2}|\theta'|^{4m-2}+\dfrac{1}{2}\sum_{i=0}^{4m-2}|x|^{4m-2-i}|x'|^{i}\right)|x-x'|\nonumber\\
    &\leq2m\left(\dfrac{3}{2}\sum_{i=0}^{4m-2}|x|^{4m-2-i}|x'|^{i}+|v|^{4m-2}+(2m-3)|\theta|^{4m-2}+\dfrac{1}{2}|\theta'|^{4m-2}\right.\nonumber\\
    &\left.+\dfrac{1}{2}\left(|x|^{2m-1/m-1}+|x'|^{2m-1/m-1}\right)^{2m-2}\right)|x-x'|\nonumber\\
    &\leq2m\left(\left(3\cdot2^{4m-4}+2^{2m-4}\right)\left(|x|^{4m-2}+|x'|^{4m-2}\right)\right.\nonumber\\
    &\left.+|v|^{4m-2}+(2m-3)|\theta|^{4m-2}+\dfrac{1}{2}|\theta'|^{4m-2}\right)|x-x'|\nonumber\\
    &\leq 2m\left(3\cdot 2^{4m-4}+2^{2m-4}+1\right)\left(|v|^{4m-2}+|v'|^{4m-2}\right)|v-v'|.\label{B3}
\end{align}
The same bound as in \eqref{B3} can be derived for $J_{\theta}$, by following the above calculations. Thus, we obtain the following bound for \eqref{B1}
\begin{align}
    &|\nabla U(v)-\nabla U(v')|\leq J_x+J_{\theta}\nonumber\\&\leq \dfrac{m}{4}\left(3\cdot2^{4m}+2^{2m}+16\right)\left(1+|v|^{4m-2}+|v'|^{4m-2}\right)|v-v'|\nonumber\\
    &\leq 2m2^{4m}\left(1+|v|^{4m-2}+|v'|^{4m-2}\right)|v-v'|.\label{B4}
\end{align}
\end{proof}
\begin{remark}\label{remark6}
    Consider the potential $U(\theta,x)$ as given in \eqref{toyU}, it satisfies \hyperlink{Assum2}{A2}, i.e. it is strongly convex with $\mu=2$.
\end{remark}
\begin{proof}
    First, we compute the second-order partial derivatives as follows
\begin{align*}
     \dfrac{\partial^2 U}{\partial x_i\partial x_j}&=\left(4m|x|^{4m-2}+2m(|\theta|^{2m}+1)|x|^{2m-2}+4|x|^2+2(|\theta|^2+1)\right)\delta_{ij}\\&+\left(8m(2m-1)|x|^{4m-4}+4m(m-1)(|\theta|^{2m}+1)|x|^{2m-4}+8\right)x_ix_j\\&:=A(x,\theta)\delta_{ij}+B(x,\theta)x_ix_j+8x_i\theta_j\\
    \dfrac{\partial^2 U}{\partial x_i\partial \theta_j}&=(4m^2|x|^{2m-2}|\theta|^{2m-2}+4)x_i\theta_j:=C(x,\theta)x_i\theta_j+4x_i\theta_j\\
     \dfrac{\partial^2 U}{\partial \theta_i\partial \theta_j}&=\left(4m|\theta|^{4m-2}+2m(|x|^{2m}+1)|\theta|^{2m-2}+4|\theta|^2+2(|x|^2+1)\right)\delta_{ij}\\&+\left(8m(2m-1)|\theta|^{4m-4}+4m(m-1)(|x|^{2m}+1)|\theta|^{2m-4}+8\right)\theta_i\theta_j\\&=A(\theta,x)\delta_{ij}+B(\theta,x)\theta_i\theta_j+8x_i\theta_j
\end{align*}
Next, we decompose the Hessian of $U(\theta,x)$ in the following way
\begin{align*}
    H(U(\theta,x)))&=\begin{bmatrix}
A(x,\theta) I_{d^x}& 0\\
0 & A(\theta,x) I_{d^{\theta}}
\end{bmatrix}+\begin{bmatrix}
B(x,\theta)J_{d^x\times d^x}& C(x,\theta)J_{d^x\times d^{\theta}}\\
C(x,\theta)J_{d^{\theta}\times d^x} & B(\theta,x)J_{d^{\theta}\times d^{\theta}}
\end{bmatrix}\odot\begin{bmatrix}
x\\
\theta 
\end{bmatrix} \begin{bmatrix}
x&\theta
\end{bmatrix}\\
&+\begin{bmatrix}
8J_{d^x\times d^x}& 4J_{d^x\times d^{\theta}}\\
4J_{d^{\theta}\times d^x} & 8J_{d^{\theta}\times d^{\theta}}
\end{bmatrix}\odot\begin{bmatrix}
x\\
\theta 
\end{bmatrix} \begin{bmatrix}
x&\theta
\end{bmatrix},
\end{align*}
where $\odot$ denotes the Hadamard product and $J_{n,m}$ the $n\times m$ ones-matrix. One observes that the leftmost matrix is positive definite with constant $\mu=2$ due to its positive diagonal structure, while the second term is semi-positive definite as the Hadamard product of two semi-positive definite matrices. To see that the first factor is indeed semi-positive definite one calculates the Schur complement;
\begin{align*}
    &B-C{\hat{B}}^{-1}C \propto B\hat{B}-C^2\\
    &=64m^2(2m-1)^2|x|^{4m-4}|\theta|^{4m-4}+32m^2(m-1)(2m-1)|x|^{2m-4}|\theta|^{6m-4}\\&+32m^2(m-1)(2m-1)|x|^{2m-4}|\theta|^{4m-4}+32m^2(m-1)(2m-1)|x|^{6m-4}|\theta|^{2m-4}\\&+16m^2(m-1)^2|x|^{4m-4}|\theta|^{4m-4}+16m^2(m-1)^2|x|^{4m-4}|\theta|^{2m-4}\\
    &+32m^2(m-1)(2m-1)|x|^{4m-4}|\theta|^{2m-4}+16m^2(m-1)^2|x|^{2m-4}|\theta|^{4m-4}\\
    &+16m^2(m-1)^2|x|^{2m-4}|\theta|^{2m-4}-16m^4|x|^{4m-4}|\theta|^{4m-4}\geq 0.
\end{align*}
\noindent Likewise one argues that the rightmost matrix is also semi-positive definite, thus the Hessian $H(U(\theta,x))$ is positive definite and $U$ strongly convex.
\end{proof}
\subsection{Mixed term - Toy Example}
\begin{remark}\label{remark7}
    Consider the potential $U(\theta,x)=x\theta+|x|^4+(|x|^2+1)(|\theta|^2+1)+|\theta|^4$, it satisfies \hyperlink{Assum2}{A2}, i.e. it is strongly convex with $\mu=1$.
\end{remark}
\begin{proof}
We decompose the Hessian of $U(\theta,x)$ as follows
\begin{align*}
&H\left(U(\theta,x)\right)\\&=\begin{bmatrix}
\left(4|x|^2+2|\theta|^2\right) I_{d}& 0\\
0 & \left(4|\theta|^2+2|x|^2\right) I_{d}
\end{bmatrix}+\begin{bmatrix}
2 I_{d}& I_d\\
I_d & 2 I_{d}
\end{bmatrix}+\begin{bmatrix}
8 J_{d}&4 J_d\\
4J_d & 8 J_{d}
\end{bmatrix}\odot \begin{bmatrix}
x\\
\theta 
\end{bmatrix} \begin{bmatrix}
x&\theta
\end{bmatrix},
\end{align*}
One observes that the leftmost matrix is semi-positive definite due to its non-negative diagonal structure, while the rightmost matrix is also semi-positive definite as the product of a positive and semi-positive definite matrices. The intermediate term is positive definite and its smallest eigenvalue grants us $\mu=1$ for the convexity parameter of $U(\theta,x)$.
\end{proof}
\end{document}